\documentclass[10pt]{amsart}
\usepackage{amsthm}
\usepackage{amsmath}
\usepackage{amscd,amsthm,amssymb,amsfonts}
\usepackage{mathrsfs}
\usepackage{dsfont}
\usepackage{stmaryrd}
\usepackage{euscript}
\usepackage{expdlist}
\usepackage{enumerate}

\input xy
\xyoption{all}
\usepackage[OT2,T1]{fontenc}

\theoremstyle{plain}
\newtheorem{thm}{Theorem}[section]
\newtheorem{thmA}{Theorem}

\newtheorem*{hypA}{Hypothesis (L)}

\newtheorem*{thm*}{Theorem}

\newtheorem{lm}[thm]{Lemma}
\newtheorem{cor}[thm]{Corollary}
\newtheorem*{cor*}{Corollary}
\newtheorem{prop}[thm]{Proposition}
\newtheorem*{conj*}{Conjecture}
\newtheorem{conj}{Conjecture}

\theoremstyle{remark}

\theoremstyle{definition}
\newtheorem*{defn*}{Definition}
\newtheorem{Remark}[thm]{Remark}
\newtheorem{I_Remark*}{Remark}
\newtheorem{defn}[thm]{Definition}

\newcommand{\nc}{\newcommand}

\newcommand{\beq}{\begin{equation}}
\newcommand{\eeq}{\end{equation}}
\newcommand{\bpmx}{\begin{pmatrix}}
\newcommand{\epmx}{\end{pmatrix}}
\newcommand{\bbmx}{\begin{bmatrix}}
\newcommand{\ebmx}{\end{bmatrix}}
\newcommand{\wh}{\widehat}
\newcommand{\wtd}{\widetilde}

\newcommand{\beqcd}[1]{\begin{equation*}\label{#1}\tag{#1}}
\newcommand{\eeqcd}{\end{equation*}}

\numberwithin{equation}{section}

\def\parref#1{\ref{#1}}
\def\thmref#1{Theorem~\parref{#1}}

\def\propref#1{Proposition~\parref{#1}}
\def\corref#1{Corollary~\parref{#1}}     
\def\secref#1{\S\parref{#1}}

\def\lmref#1{Lemma~\parref{#1}}
\def\subsecref#1{\S\parref{#1}}
\def\defref#1{Definition~\parref{#1}}

\def\conjref#1{Conjecture~\parref{#1}}
\def\makeop#1{\expandafter\def\csname#1\endcsname
  {\mathop{\rm #1}\nolimits}\ignorespaces}

\makeop{Hom}   \makeop{End}   \makeop{Aut}   
\makeop{Pic} \makeop{Gal}       \makeop{Div} \makeop{Lie}
\makeop{PGL}   \makeop{Corr} \makeop{PSL} \makeop{sgn} \makeop{Spf}
 \makeop{Tr} \makeop{Nr} \makeop{Fr} \makeop{disc}
\makeop{Proj} \makeop{supp} \makeop{ker}   \makeop{Im} \makeop{dom}
\makeop{coker} \makeop{Stab} \makeop{SO} \makeop{SL} \makeop{SL}
\makeop{Cl}    \makeop{cond} \makeop{Br} \makeop{inv} \makeop{rank}
\makeop{id}    \makeop{Fil} \makeop{Frac}  \makeop{GL} \makeop{SU}
\makeop{Trd}   \makeop{Sp} \makeop{Tr}    \makeop{Trd} \makeop{Res}
\makeop{ind} \makeop{depth} \makeop{Tr} \makeop{st} \makeop{Ad}
\makeop{Int} \makeop{tr}    \makeop{Sym} \makeop{can} \makeop{SO}
\makeop{torsion} \makeop{GSp} \makeop{Tor}\makeop{Ker} \makeop{rec}
\makeop{Ind} \makeop{Coker}
 \makeop{vol} \makeop{Ext} \makeop{gr} \makeop{ad}
 \makeop{Gr}\makeop{corank} \makeop{Ann}
\makeop{Hol} 
\makeop{Fitt} \makeop{Mp} \makeop{CAP}

\def\Sel{Sel}

\def\Ord{{\mathrm{ord}}}

\def\Spec{\mathrm{Spec}\,}

\DeclareMathAlphabet{\mathpzc}{OT1}{pzc}{m}{it}
\DeclareSymbolFont{cyrletters}{OT2}{wncyr}{m}{n}
\DeclareMathSymbol{\SHA}{\mathalpha}{cyrletters}{"58}

\def\makebb#1{\expandafter\def
  \csname bb#1\endcsname{{\mathbb{#1}}}\ignorespaces}
\def\makebf#1{\expandafter\def\csname bf#1\endcsname{{\bf
      #1}}\ignorespaces}
\def\makegr#1{\expandafter\def
  \csname gr#1\endcsname{{\mathfrak{#1}}}\ignorespaces}
\def\makescr#1{\expandafter\def
  \csname scr#1\endcsname{{\EuScript{#1}}}\ignorespaces}
\def\makecal#1{\expandafter\def\csname cal#1\endcsname{{\mathcal
      #1}}\ignorespaces}

\def\doLetters#1{#1A #1B #1C #1D #1E #1F #1G #1H #1I #1J #1K #1L #1M
                 #1N #1O #1P #1Q #1R #1S #1T #1U #1V #1W #1X #1Y #1Z}
\def\doletters#1{#1a #1b #1c #1d #1e #1f #1g #1h #1i #1j #1k #1l #1m
                 #1n #1o #1p #1q #1r #1s #1t #1u #1v #1w #1x #1y #1z}
\doLetters\makebb   \doLetters\makecal  \doLetters\makebf
\doLetters\makescr
\doletters\makebf   \doLetters\makegr   \doletters\makegr

\normalsize

\makeop{Ram} \makeop{Rep} \makeop{mass}

\makeop{Bl}
\def\abs#1{\left|#1\right|}

\def\Fpbar{\ol{\mathbb F}_p}

\def\Qp{\Q_p}
\def\Qbar{\ol{\Q}}
\def\Zbar{\ol{\Z}}

\def\Zp{\Z_p}

\def\rmN{{\mathrm N}}

\def\cA{{\mathcal A}}  

\def\cD{\mathcal D}
\def\cE{{\mathcal E}}
\def\cF{{\mathcal F}}  
\def\cG{{\mathcal G}}
\def\cL{{\mathcal L}}

\def\cJ{\mathcal J}
\def\cK{{\mathcal K}}  
\def\cM{\mathcal M}

\def\cO{\mathcal O}
\def\cS{{\mathcal S}}
\def\cf{{\mathcal f}}
\def\cW{{\mathcal W}}

\def\cP{{\mathcal P}}
\def\cC{\mathcal C}
\def\cJ{\mathcal J}

\def\bfc{\mathbf c}

\def\bfK{\mathbf K}

\def\bfM{\mathbf M}
\def\bfU{\mathbf U}
\def\bfT{{\mathbf T}}
\def\bfL{\mathbf L}

\def\bff{\mathbf f}

\def\bfu{\mathbf u}

\def\bftheta{\boldsymbol{\theta}}

\def\bfal{\boldsymbol{\alpha}}

\def\sH{\mathscr H}
\def\sF{\mathscr F}
\def\sL{\mathscr L}

\def\sV{\mathscr V}
\def\sK{\mathscr K}

\def\sR{\mathscr R}

\def\sW{{\mathscr W}}

\def\bbI{\mathbb I}

\newcommand{\Z}{\mathbf Z}
\newcommand{\Q}{\mathbf Q}
\newcommand{\R}{\mathbf R}
\newcommand{\C}{\mathbf C}
\newcommand{\A}{\mathbf A}    

\def\bbE{{\mathbb E}}

\def\fraka{{\mathfrak a}}
\def\frakb{{\mathfrak b}}
\def\frakc{{\mathfrak c}}
\def\frakd{\mathfrak d}
\def\frakf{\mathfrak f}

\def\frakp{{\mathfrak p}}
\def\frakP{\mathfak P}
\def\frakq{\mathfrak q}

\def\frakm{\mathfrak m}
\def\frakn{\mathfrak n}
\def\frakl{\mathfrak l}
\def\frakP{\mathfrak P}
\def\frakF{{\mathfrak F}}

\def\frakL{{\mathfrak L}}
\def\frakD{\mathfrak D}
\def\frakM{\mathfrak M}

\def\frakH{{\mathfrak H}}

\def\frakA{\mathfrak A}

\def\frakC{{\mathfrak C}}

\def\frakX{\mathfrak X}
\def\frako{\mathfrak o}
\def\frakQ{\mathfrak Q}

\def\frakN{\mathfrak N}

\def\bfone{{\mathbf 1}}

\def\Zhat{\widehat{\Z}}
\def\wbar{\bar{w}}

\def\wbar{\ol{w}}

\def\pont{Pontryagin } 

\def\BS{Bruhat-Schwartz }

\def\Teich{Teichm\"{u}ller }

\def\Frob{\mathrm{Frob}}

\newcommand{\<}{\langle}   
\renewcommand{\>}{\rangle} 

\def\isoto{\stackrel{\sim}{\to}}

\def\surjto{\twoheadrightarrow}
\def\imply{\Rightarrow}
\def\ot{\otimes}

\def\hookto{\hookrightarrow}
\def\longto{\longrightarrow}
\def\ol{\overline}  \nc{\opp}{\mathrm{opp}} \nc{\ul}{\underline}

\newcommand{\pair}[2]{\< #1, #2\>}

\newcommand{\pairing}{\pair{\,}{\,}}

\def\XYmatrix{\xymatrix@M=8pt} 
\def\ncmd{\newcommand}
\ncmd{\xysubset}[1][r]{\ar@<-2.5pt>@{^(-}[#1]\ar@<2.5pt>@{_(-}[#1]}
\ncmd{\XYmatrixc}[1]{\vcenter{\XYmatrix{#1}}}
\ncmd{\xyto}[1][r]{\ar@{->}[#1]}
\ncmd{\xyinj}[1][r]{\ar@{^(->}[#1]}
\ncmd{\xysurj}[1][r]{\ar@{->>}[#1]}
\ncmd{\xyline}[1][r]{\ar@{-}[#1]}
\ncmd{\xydotsto}[1][r]{\ar@{.>}[#1]}
\ncmd{\xydots}[1][r]{\ar@{.}[#1]}
\ncmd{\xyleadsto}[1][r]{\ar@{~>}[#1]}
\ncmd{\xyeq}[1][r]{\ar@{=}[#1]} \ncmd{\xyequal}[1][r]{\ar@{=}[#1]}
\ncmd{\xyequals}[1][r]{\ar@{=}[#1]}
\ncmd{\xymapsto}[1][r]{l\ar@{|->}[#1]}\ncmd{\xyimplies}[1][r]{\ar@{=>}[#1]}
\ncmd{\xyiso}{\ar[r]_-{\sim}}
\def\injxy{\ar@{^(->}}

\newcommand{\pMX}[4]{\begin{pmatrix}
{#1}& {#2}\\
{#3}&{#4}\end{pmatrix} }

 \newcommand{\pDII}[2]{\begin{pmatrix}{#1}&0
 \\0&{#2}\end{pmatrix}}

\newcommand{\seesaw}[4]{{#1}\ar@{-}[rd]\ar@{-}[d]&{#2}\ar@{-}[d]\\
{#3}\ar@{-}[ru]&{#4}}

\def\ie{i.e. }

\def\cf{\mbox{{\it cf.} }}

\def\uf{\varpi} 
\def\Abs{{|\!\cdot\!|}} 

\def\Sg{{\varSigma}}  
\def\Sgbar{\Sg^c}
\def\ndivides{\nmid}
\def\ndivide{\nmid}
\def\x{{\times}}

\def\onehalf{{\frac{1}{2}}}
\def\e{\varepsilon} 
\def\al{\alpha}

\def\Lam{\Lambda}

\def\om{\omega}
\def\dirlim{\varinjlim}
\def\prolim{\varprojlim}
\def\iso{\simeq}
\def\con{\equiv}
\def\bksl{\backslash}
\newcommand\stt[1]{\left\{#1\right\}}
\def\ep{\epsilon}
\def\varep{\varepsilon}

\def\lam{\lambda}

\def\sg{\sigma}
\def\vp{\varphi}
\def\disjoint{\sqcup}

\def\setp{{(p)}}

\def\divides{\mid}

\newcommand{\powerseries}[1]{\llbracket{#1}\rrbracket}

\renewcommand\pmod[1]{\,(\mbox{mod }{#1})}

\newcommand\Dmd[1]{\left<{#1}\right>} 

\def\Cp{\C_p}

\usepackage{xcolor}
\usepackage[utf8]{inputenc}
\usepackage{hyperref}
\usepackage[T1]{fontenc}
\DeclareMathSymbol{\shortminus}{\mathbin}{AMSa}{"39}

\usepackage{hyperref}

\renewenvironment{abstract}{{\bf Abstract}. }{}

\renewenvironment{abstract}{{\bf Abstract}. }{}
\title[CM congruence and trivial zeros of the Katz $p$-adic $L$-functions]{CM congruence and trivial zeros of the Katz $p$-adic $L$-functions for CM fields}
\author[A. Betina and M.-L. Hsieh]{Adel Betina and Ming-Lun Hsieh}
\address{Faculty of Mathematics, University of Vienna, Oskar-Morgenstern-Platz 1, A-1090 Wien, Austria}
\email{adelbetina@gmail.com}
\address{Institute of Mathematics, Academia Sinica, Taipei 10617, Taiwan}
\email{mlhsieh@math.sinica.edu.tw}


\def\cmpt{\varsigma}
\def\Mat{\mathrm{M}}
\def\rmH{{\rm H}}

\def\rmd{\mathrm{d}}

\def\G{{\rm G}}
\def\vp{\varphi}

\def\Fv{F_v}

\def\loc{{\rm loc}}

\def\Cp{\ol{\Q}_p}
\def\Frob{{\rm Fr}}

\def\Om{\boldsymbol \om}
\def\rmt{{\rm t}}

\def\brch{\phi}
\def\e{\boldsymbol\varepsilon}
\def\cyc{\Dmd{\e_{\rm cyc}}}
 \def\aA{\cC(\brch,\brch^c)}
\def\bB{\cC(\brch^c,\brch)}

\def\Pbar{\ol{\frakP}}
\def\Sgbar{\ol{\Sg}}

\def\addchar{\bfe}
\def\ulk{\ul{k}}
\def\OFv{\cO_{F_v}}
\def\Csplit{\frakF}
\def\Cinert{\frakC^-}
\def\OK{\frako_K}

\def\OF{{\frako}}
\def\cmptv{\varsigma}
\def\skewhf{\vartheta}
\def\Sg{\Sigma}
\def\kLog{\kappa_{\rm cyc}}
\def\kOrd{\kappa^{\rm ur}}
\thanks{Betina was supported by the START-Prize Y966 of the Austrian Science Fund (FWF) and Hsieh was partially supported by MOST grants 108-2628-M-001-009-MY4 and 110-2628-M-001-004-. }
\begin{document}
\maketitle
\begin{abstract}The aim of this paper is to investigate the trivial zeros of the Katz $p$-adic $L$-functions by the CM congruence. We prove the existence of trivial zeros of the Katz $p$-adic $L$-functions for general CM fields and establish a first derivative formula of the cyclotomic $p$-adic $L$-functions at trivial zeros under some Leopoldt hypothesis. The crucial ingredients in our proof are a special case of $p$-adic Kronecker limit formula for CM fields and a leading term formula of anticyclotomic $p$-adic $L$-functions at trivial zeros via the explicit congruences between CM and non-CM Hida families of Hilbert cusp forms.\end{abstract}

\section{Introduction}

\subsection{Conjectures on trivial zeros of the Katz $p$-adic $L$-functions}Let $F$ be a totally real number field of degree $d$ and absolute discriminant $\Delta_F$. Let $K$ be a totally imaginary quadratic extension of $F$. Let $p$ be a prime number. Fix once and for all embeddings $\iota_\infty:\Qbar\hookto \C$ and $\iota_p:\Qbar\hookto \Cp$
singling out a decomposition subgroup of $G_\Q:=\Gal(\Qbar/\Q)$ at $p$. Let $c\in \Gal(K/F)$ be the non-trivial automorphism. We suppose throughout this paper the following ordinary hypothesis on the CM field $K$:
\beqcd{ord}
\text{Every prime factor of $p$ in $F$ splits in $K$.}
\eeqcd
 Let $\Sg\subset \Hom(K,\C)$ be a $p$-ordinary CM-type of $K$, \ie \begin{itemize}\item $\Sg \disjoint\Sg c= \Hom(K,\C)$ and $\Sg\cap \Sg c=\emptyset$; 
\item  the set $\Sg_p$ of $p$-adic places induced by elements in $\Sg$ composed with $\iota_p\iota_\infty^{-1}$ is disjoint from $\Sg_pc$.
 \end{itemize}
 The hypothesis \eqref{ord} assures the existence of $p$-ordinary CM-types and we fix one such $\Sg$. Denote by $\wtd K_{\infty}$  for the  compositum of all the $\Z_p$-extensions of $K$ and by  $K_{\Sigma_p}$ for the maximal sub-extension of $\wtd K_{\infty}$ unramified outside $\Sigma_p$. In fact, $\Gal(K_{\Sigma_p}/K)\simeq \Z_p^{1+\delta_{F,p}}$ and $\Gal(\wtd K_{\infty}/K)\simeq \Z_p^{d+1+\delta_{F,p}}$, where $\delta_{F,p}$ is the Leopoldt defect for $F$ and $p$. Let $\cW$ be a finite extension of the completion $\wh\Z_p^{
 \rm ur}$ of the maximal unramified extension of $\Zp$ in $\Cp$. To each choice of $p$-ordinary CM type $\Sg$ and a ray class character $\chi$ of $K$ valued in $\cW$, one can attach the Katz $p$-adic $L$-function $\cL_{\Sg}(\chi)\in \cW\powerseries{\Gal(\wtd K_{\infty}/K)}$ constructed by Katz \cite{Katz78Inv} and Hida--Tilouine \cite{HidaTil93ASENS}. This $p$-adic $L$-function $\cL_{\Sg}(\chi)$ is uniquely characterized by the following interpolation property: there exists a complex period 
\[\Omega=(\Omega_{\sg})_{\sg\in\Sg}\in (\C^\x)^\Sg\]
and a $p$-adic period
\[\Omega_p=(\Omega_{p,\sg})_{\sg\in \Sg}\in (\cW^\x)^{\Sg}\]
such that for every crystalline character $\phi:\Gal(\wtd K_{\infty}/K)\to\Cp^\x$ of weight $k\Sigma+j(1-c)$ for some integer $k$ and $j=\sum\limits_{\sg\in\Sg}j_\sg\sg\in\Z[\Sg]$ with either $k>0$ and $j_\sg\geq 0$ or $k\leq 1$ and $j_\sg+k\geq 1$, we have 
\beq\label{E:Katz1}\begin{aligned}\frac{\nu_\phi(\cL_{\Sg}(\chi))}{\Omega_p^{k\Sigma+2j}}&=\frac{[\frako_K^\x:\frako^\x_F]}{2^d\sqrt{\Delta_F}}\cdot\frac{1}{(\sqrt{-1})^{k\Sigma+j}}\cdot \frac{\Gamma_\C(k\Sg+j)L_{\rm fin}(0,\chi\phi_\infty)}{\Omega^{k\Sigma+2j}}\\
&\times\prod_{\frakP\in\Sigma_p}(1-\chi\phi_\infty(\Pbar))(1-\chi\phi_\infty(\frakP^{-1})\rmN\frakP^{-1}),\end{aligned}\eeq
where \begin{itemize}\item $\nu_\phi: \cW\powerseries{\Gal(\wtd K_{\infty}/K)} \to\Qbar_p$ is the unique $\cW$-algebras homomorphism such that $\nu_{\phi \mid \Gal(\wtd K_{\infty}/K)}=\phi$, \item $\phi_\infty$ is the grossencharacter of type $A_0$ associated with $\phi$ such that $\phi_\infty(\frakQ)$ is the value of $\phi$ at the geometric Frobenius $\Frob_\frakQ$ at primes $\frakQ\ndivides p$, \item $L_{\rm fin}(s,\chi\phi_\infty)=\sum_{\fraka}\chi\phi_\infty(\fraka)\rmN\fraka^{-s}$ is the classical $L$-series of $\chi\phi_\infty$, where $\fraka$ runs over all integral ideals of $K$ relatively prime to the conductor of $\chi\phi_\infty$.
\item $\Pbar=c(\frakP)$.
 \end{itemize}
 In the expression of \eqref{E:Katz1}, we have used the convention for $\xi=\sum_{\sg\in\Sg}\xi_\sg\sg\in\Z[\Sg]$, 
 and for $x=(x_\sg)\in\C^\Sg$:
 \[x^\xi=\prod_{\sg\in\Sg}x_\sg^{\xi_\sg};\quad \Gamma_\C(\xi)=\prod_{\sg\in\Sg}2(2\pi)^{-\xi_\sg}\Gamma(\xi_\sg).\]
 Any element $a\in\C$ is considered to be an element of $\C^\Sg$ via the diagonal embedding $a\mapsto (a)_{\sg\in\Sg}$.
 
For each positive integer $m$, denote by $\Q(\zeta_m)$ the $m$-th cyclotomic field. Let $\Q(\zeta_{p^\infty})=\cup_{n=1}^\infty \Q(\zeta_{p^n})$ and denote by $\e_{\rm cyc}:\Gal(\Q(\zeta_{p^\infty})/\Q)\isoto\Zp^\x$ be the usual $p$-adic cyclotomic character. Let $\Q_\infty\subset \Q(\zeta_{p^\infty})$ be the cyclotomic $\Zp$-extension of $\Q$. We view $\Gal(\Q_\infty/\Q)$ as a subgroup of $\Gal(\Q(\zeta_{p^\infty})/\Q)$ via the isomorphism $\Gal(\Q(\zeta_{p^\infty})/\Q)\iso\Gal(\Q_\infty/\Q)\times \Gal(\Q(\zeta_p)/\Q),\, \sg\mapsto (\sg|_{\Q_\infty},\sg|_{\Q(\zeta_p)})$. Let $K_\infty^+=\Q_\infty K$ be the cyclotomic $\Zp$-extension of $K$ and let $\cyc:\Gal(\wtd K_\infty/K)\to \Zp^\x$ be the character defined by the composition \[\Gal(\wtd K_\infty/K)\to \Gal(K_\infty^+/K)\to\Gal(\Q_\infty/\Q)\hookto\Gal(\Q(\zeta_{p^\infty})/\Q)\overset{\e_{\rm cyc}}\isoto \Zp^\x.\] We consider the cyclotomic Katz $p$-adic $L$-function defined by 
\[L_{\Sg}(s,\chi):=\cyc^s\left(\cL_{\Sg}(\chi)\right), \quad s  \in \Z_p.\]
Let us put 
\begin{align*}\Sg_p^{\rm irr}&=\stt{\frakP\in\Sg_p\mid \chi(\Pbar)=1};\\
r_{\Sg}(\chi)&=\#\Sg_p^{\rm irr}-\dim_{\Cp}\!\!\rmH^0(K,\chi).\end{align*}
In view of the interpolation formula \eqref{E:Katz1}, we expect that $L_\Sg(s,\chi)$ has a \emph{trivial zero} whenever $r_\Sg(\chi)>0$.  By analogy to a conjecture of Gross regarding the order of vanishing of Deligne-Ribet $p$-adic $L$-functions, we have the following conjecture regarding the order of the vanishing of $L_\Sg(s,\chi)$ at $s=0$. \begin{conj}\label{conj1}We always have  \[\Ord_{s=0}L_{\Sg}(s,\chi) = r_\Sg(\chi).\]  
\end{conj}
Note that when $K$ is an imaginary quadratic field, the implication $r_\Sg(\chi)>0\imply L_\Sg(0,\chi)=0$ follows from the $p$-adic Kronecker limit formula, but in general it is not clear if $L_\Sg(s,\chi)$ possesses a zero at $s=0$ when $r_\Sg(\chi)>0$ since $s=0$ is a not a critical point for $\chi$ and is outside the range of the interpolation of $L_{\Sg}(s,\chi)$. 

We briefly recall the conjectural formula of the leading term of $L_\Sg(s,\chi)$ at the trivial zero in \cite{BS}. First we prepare some crucial ingredients in the statement of this formula.  Following the notation in \cite{Gross81Tokyo}, if $X$ is an abelian group and $R$ is a ring, we denote the $R$-module $R\ot_\Z X$ simply by $RX$. Let $H:=\Qbar^{\ker\chi}$ be the finite cyclic extension of $K$ cut out by $\chi$ and $G:=\Gal(H/K)$. Let $\frako_H^\times$ the group of units of $H$. Put
\[\frako_H^\x[\chi]:=e_\chi\left(\Cp\frako_H^\x\right).\]
where $e_\chi$ is the usual idempotent
\[e_\chi=\frac{1}{[H:K]}\sum_{\tau\in G}\chi(\tau^{-1})\ot\tau.\]
Let $\log_p:\Cp^\times\to\Cp$ denote the Iwasawa's $p$-adic logarithm with $\log_p(p)=0$. Let $\Sg_H$ be the set of embeddings $\sg_H:H\hookto \C$ extending those in $\Sg$ and let $G$ act on $\Sg_H$ by $\tau\cdot \sg=\sg\circ\tau^{-1}$ for $\tau\in G$ and $\sg\in\Sg_H$. Let $Y_{H,\Sg}$ be the free abelian group on $\Sg_H$, regarded as a $\Z[G]$-module with the the action of $G$-action on $\Sg_H$. Define the $p$-adic regulator map
\[\log_{\Sg,p}:H^\x\to \Cp Y_{H,\Sg},\quad \log_{\Sg,p}(x)=\sum_{\sg\in\Sg_H}\log_p(\iota_p\sg(x))\sg.\]
We extend the map $\log_{\Sg,p}$ by linearity to the map $\log_{\Sg,p}:\Cp H^\x\to \Cp Y_{H,\Sg}$ of $\Cp[G]$-modules and obtain a map between the $\chi$-isotypic parts on the both sides
\beq\label{E:preg}
\log_{\Sg,p}:\frako_H^\x[\chi]\to Y_{H,\Sg}[\chi]:=e_\chi(\Cp Y_{H,\Sg}).
\eeq
\begin{conj}[The $\Sigma$-Leopoldt conjecture for $\chi$]\label{conj2} The map $\log_{\Sg,p}$ in \eqref{E:preg} is injective.
\end{conj}
We remark that \conjref{conj2} is a consequence of the well-known $p$-adic Schanuel conjecture as explained in \cite[Lemma 1.2.1]{HidaTil94Inv};  the four exponential conjecture implies \conjref{conj2} when $F$ is a real quadratic field. Fixing an isomorphism $\iota:\C\iso\Cp$ such that $\iota\circ\iota_\infty=\iota_p:\Qbar\hookto \Cp$, the Dirichlet regulator map is defined by 
\[\log_{\Sigma,\infty}:H^\x\to \Cp Y_{H,\Sg},\quad \log_{\Sigma,\infty}(x)=\sum_{\sg\in\Sg_H}\iota\left(\log\abs{\sg(x)}\right)\sg.\]
Suppose that $\chi$ is not the trivial character $\bfone$. By Dirichlet's units Theorem, the induced map $\log_{\Sigma,\infty}:\frako_H^\times[\chi]\isoto Y_{H,\Sg}[\chi]$ is an isomorphism of $d$-dimensional $\Cp$-vector spaces. It follows that if \conjref{conj2} holds, then the $p$-adic regulator map $\log_{\Sg,p}:\frako_H^\x[\chi]\iso Y_{H,\Sg}[\chi]$ is also an isomorphism. Define the regulator $\mathscr{R}_{\Sigma}(\chi)$ in $\Cp$ by 
\[\sR_\Sg(\chi):=\det\left(\log_{\Sg,p}\circ\log_{\Sg,\infty}^{-1}\vert\,Y_{H,\Sg}[\chi]\right).\]
It is clear that the $\Sg$-Leopoldt conjecture for $\chi$ is equivalent to the non-vanishing of $\sR_\Sg(\chi)$. Following Buyukboduk and Sakamoto \cite[Definition 5.2]{BS}, we next introduce the cyclotomic $\sL$-invariant for $\chi\neq \bfone$ under the $\Sg$-Leopoldt conjecture. Let $\frako_{H,\Sgbar_p}^\x$ be the $\Sgbar_p$-units in $H$ and put
\[\frako_{H,\Sgbar_p}^\x[\chi]:=e_\chi\left(\Cp\frako_{H,\Sgbar_p}^\x\right).\]
We consider the space \[\mathbf{U}_\chi:= \ker\stt{\frako_{H,\Sgbar_p}^\x[\chi]\overset{\log_{\Sg,p}}\longto \Cp Y_{H,\Sg}}.\]
If $\frakP$ is a prime ideal of the ring $\frako_K$ of integers of $K$, let $\Ord_{\frakP}:K_\frakP^\x\to\Z$ be the normalized valuation map and $\log_{\frakP}:=\log_p\circ\rmN_{K_\frakP/\Qp}:K_\frakP^\x\to\Zp$ be the cyclotomic $p$-adic logarithm, where $\rmN_{K_\frakP/\Qp}$ is the norm map from $K_\frakP$ to $\Qp$. For each $\frakP\in \Sg_p^{\rm irr}$, we choose a prime $\Pbar_H$ of $H$ above $\Pbar$. Then via the inclusion $\frako_{H,\Sgbar_p}^\x\subset H^\x\subset H_{\Pbar_H}^\x=K_{\Pbar}^\x$, we can evaluate $\Ord_{\Pbar}$ and $\log_{\Pbar}$ on $\frako_{H,\Sgbar_p}^\x$ and extend by linearity to maps
\begin{align*}
{\mathbf O}_p:\bfU_\chi&\xrightarrow{\bigoplus_{\frakP\in \Sg_p^{\rm irr}}\Ord_{\Pbar}} \Cp^{\Sg_p^{\rm irr}};\\
{\mathbf L}_p:\bfU_\chi&\xrightarrow{\bigoplus_{\frakP\in \Sg_p^{\rm irr}}\log_{\Pbar}} \Cp^{\Sg_p^{\rm irr}}.
\end{align*}
If $\chi\neq \bfone$ and the $\Sigma$-Leopoldt conjecture for $\chi$ holds,
the first map ${\bf O}_p$ is an isomorphism;  we can define the cyclotomic $\sL$-invariant of $\chi$ by 
\[\sL_\chi:=(-1)^{r_\Sg(\chi)}\det\left(\bfO_p^{-1}\circ\bfL_p\right).\]
Note that the definition of $\sL_\chi$ does not depend on the choice of the prime $\Pbar_H$ above $\Pbar$. 

For any $\Zp$-extension $K_\Gamma$ of $K$ with the Galois group $\Gamma$, we can also consider the trivial zeros of the Katz $p$-adic $L$-function $\cL_\Sg(\chi)|_{\Gamma}\in \cW\powerseries{\Gamma}$ obtained by the restriction of $\cL_\Sg(\chi)$ to $\Gamma$.
It is shown in \cite[Theorem 1.1]{BS} that the leading term formula of $\cL_\Sg(\chi)|_{\Gamma}$ at the trivial zero can be obtained as a consequence of several outstanding open conjectures in algebraic number theory, eg. Rubin-Stark conjecture and the Reciprocity Conjecture. The following is the special case of this conjectural leading term formula when $K_\Gamma=K_\infty^+$ is the cyclotomic $\Zp$-extension.  \begin{conj}\label{conj3}Assume the validity of the $\Sg$-Leopoldt conjecture for $\chi$ and $\chi\neq \bfone$. Then we have $\Ord_{s=0}L_\Sg(s,\chi)\geq r_\Sg(\chi)$ and 
\begin{align*} \lim_{s\to 0}\frac{L_\Sg(s,\chi)}{s^{r_\Sg(\chi)}}&=(-1)^{r_\Sg(\chi)}\sL_\chi\prod_{\frakP\in\Sigma_p \backslash \Sigma_p^{\mathrm{irr}}  }(1-\chi(\Pbar))\\&\quad\times\prod_{\frakP\in\Sg_p}(1-\chi(\frakP^{-1})\rmN\frakP^{-1})\cdot \sR_\Sg(\chi)\cdot L_{\rm fin}^*(0,\chi), \end{align*} where $L_{\rm fin}^*(0,\chi):=\iota\left(\lim\limits_{s\to 0}s^{-d}L_{\rm fin}(s,\chi)\right)\in\Cp^\x$ is the leading term of the $L$-series $L_{\rm fin}(s,\chi)$ .
\end{conj}
When $K$ is an imaginary quadratic field (so $r_\Sg(\chi)\leq 1$), \conjref{conj1} is equivalent to the non-vanishing of the $\sL$-invariant, which is known to be a consequence of the four exponentials conjecture (\cite[Prop. 1.11]{BeD21Adv}), \conjref{conj2} holds thanks to the Baker-Brumer Theorem and \conjref{conj3} is proved in \cite[Theorem 1.6]{BS} via the Euler system of elliptic units. These conjectures are still wide open for general CM fields. 
\subsection{Statement of the main results}
The main goal of this article is to prove the existence of the trivial zeros and provide some evidence towards \conjref{conj1} for general CM fields in the rank zero and one case under the Leopoldt hypotheses. Our method does not reply on the existence of conjectural Rubin-Stark units and conjectural reciprocity law. In the course of the proofs, we also obtain a leading term formula of the \emph{anticyclotomic} Katz $p$-adic $L$-functions. 

Let $h_{K/F}$ be the relative class number of $K/F$ and $\frakD_{K/F}$ be the relative different of $K/F$. For simplicity, it is assumed throughout this introduction that $(\chi,K,p)$ satisfies the following technical conditions.
\begin{align}
\label{Iw1}\tag{H1}&\text{ The prime }p\text{ does not divide } 6\Delta_Fh_{K/F}; \\
\label{Iw2}\tag{H2} &\text{$\chi$ is of prime-to-$p$ order and is unramified at primes dividing $p\frakD_{K/F}$}.
\end{align}
Note that these assumptions are satisfied by all but finitely primes $p$ once $(\chi,K)$ is fixed. Our first result gives a sufficient and necessary condition for the existence of the trivial zero, \ie  \conjref{conj1} in the rank zero case under the following Leopoldt hypothesis.
\begin{hypA}\label{hyL}
 The Leopoldt conjecture for $F$ and the $\Sg$-Leopoldt conjecture for $\chi$ both hold. 
\end{hypA}

\begin{thmA}\label{T:thmA}Suppose that Hypothesis (L) is valid. Then we have 
\[\Ord_{s=0}L_\Sg(s,\chi)=0\text{ if and only if }r_\Sg(\chi)=0.\]
\end{thmA}
We say $\chi$ is \emph{anticyclotomic} if $\chi(c\sg c)=\chi(\sg^{-1})$ for $\sg\in \Gal(\ol{K}/K)$. If $\chi$ is further assumed to be anticyclotomic, it was shown in \cite{BDF21MZ} that $L_\Sg(0,\chi)\neq 0$ if $r_\Sg(\chi)=0$ under Hypothesis (L). Our second result relates the rank one case of \conjref{conj1} to the non-vanishing of the $\sL$-invariant $\sL_\chi$ for anticyclotomic characters $\chi$.

\begin{thmA}\label{T:thmB} In addition to the validity of Hypothesis (L), we assume further that $\chi$ is a non-trivial anticyclotomic character. Then we have  
\[\Ord_{s=0}L_\Sg(s,\chi)=1\text{ if and only if }r_\Sg(\chi)= 1\text{ and }\sL_\chi\neq 0.\]
\end{thmA}
\subsection{The method}Although our results are mainly concerned with the first derivative of the cyclotomic $p$-adic $L$-function, the method crucially relies on the use of certain two-variable $p$-adic $L$-function, which we describe as follows. Denote by $\sV:G_F^{\rm ab}\to G_K^{\rm ab}$ for the transfer map and fix a $p$-adic character $\e_\Sg:\Gal(K_{\Sg_p}/K)\to \Qbar_p^\x$  such that $\e_{\Sg}\circ\sV$ induces the $p$-adic cyclotomic character of $\Gal(F_\infty/F)$. Let $\e_{\Sgbar}$ denote the conjugate character $\e_{\Sg}^c$ and define the analytic function $\cL_\Sg(-,-,\chi):\Zp^2\to \Cp$ by \begin{align}
 \cL_{\Sg}(s,t,\chi):&=\e_{\Sg}^s\e_{\Sgbar}^t\left(\cL_{\Sg}(\chi)\right),\quad (s,t)\in\Zp^2.
 \end{align}
 We prove the following result in \thmref{T:improved} and \propref{P:nonvanishing}.
\begin{thm}\label{T:main1} 
 There exists an analytic function $\cL_\Sg^*(s,\chi)$ on $\Zp-\stt{0}$ such that \[\cL_{\Sg}(s,0,\chi)=\cL_\Sg^*(s,\chi)\prod_{\frakP\in\Sg_p}(1-\chi(\Pbar)\e_\Sg^s(\Fr_{\Pbar})).\]
 Moreover, the following statements hold:
 \begin{enumerate}
\item[(i)] If the Leopoldt conjecture for $F$ holds, then $\cL_\Sg^*(s,\chi)$ is analytic at $s=0$ if $\chi$ is non-trivial, and has a simple pole at $s=0$ if $\chi=\bfone$ is the trivial character.
\item[(ii)] If Hypothesis (L) is valid, then $\cL_\Sg^*(0,\chi)\neq 0$ for $\chi\neq \bfone$.\end{enumerate}
\end{thm}
 By definition, we have $\cL_\Sg(s,s,\chi)=L_\Sg(s,\chi)$ for $s$ in a sufficiently small neighborhood of $0$, and hence \thmref{T:thmA} directly follows from \thmref{T:main1}. We shall call $\cL_\Sg^*(s,\chi)$ the \emph{improved} $p$-adic $L$-function for $\chi$ in the sequel. The proof of \thmref{T:main1} is briefly outlined as follows. Recall that the Katz $p$-adic $L$-function $\cL_\Sg(\chi)$ for CM fields is obtained by the evaluation of Katz's Eisenstein measure at CM points. To construct this improved $p$-adic $L$-function, we modify the Katz's $p$-adic Eisenstein measure and construct a nice  \emph{ordinary} $p$-adic Eisenstein measure; then obtain the improved $p$-adic $L$-function $\cL_\Sg^*(s,\chi)$ by evaluating this $p$-adic ordinary Eisenstein measure at CM points. The interpolation formula of $\cL_\Sg^*(s,\chi)$ is proved via an explicit calculation of the toric period integral of Eisenstein series (\propref{P:period1.2}), and then the relation between $\cL_\Sg(\chi)$ and $\cL_\Sg^*(s,\chi)$ can be seen immediately from the interpolation formula. The analyticity of $\cL_\Sg^*(s,\chi)$ is determined by that of the constant term of our ordinary $p$-adic Eisenstein series, which is essentially the Deligne-Ribet $p$-adic $L$-function $L_p(1-s,\chi_+^{-1})$ with $\chi_+:=\chi\circ\sV$. We thus need the Leopoldt conjecture for $F$ to conclude the analyticity of $\cL_\Sg^*(s,\chi)$ at $s=0$ in some special cases (See \propref{P:Kronecker}). To see the non-vanishing of $\cL^*_\Sg(0,\chi)$, we first observe that the $\Sg$-Leopoldt conjecture is equivalent to the vanishing of the Bloch-Kato Selmer group  $\rmH^1_{(\emptyset,f)}(K,\chi)$, consisting classes in $\rmH^1(K,\chi)$ unramified outside $\Sg_p$ (\lmref{L:21}). On the other hand, we deduce a one-sided divisibility in the main conjecture for the improved $p$-adic $L$-function $\cL^*_\Sg(s,\chi)$ from \cite{Hsieh14JAMS}, which in turn implies that $\cL^*_\Sg(0,\chi)\neq 0$ if $\rmH^1_{(\emptyset,f)}(K,\chi)=\stt{0}$. This allows us to conclude the non-vanishing of $\cL^*_\Sg(0,\chi)$ under the Hypothesis (L).

In general the value $\cL_\Sg^*(0,\chi)$ is expected to be always nonzero whenever $\chi$ is non-trivial. When $K$ is an imaginary quadratic field, $\cL_\Sg^*(s,\chi)$ is nothing but the $p$-adic $L$-function $\cL_{\frakp}^*(s,\chi)$ introduced in \cite[\S 3.3]{CH21}, and the value $\cL^*_\Sg(0,\chi)$ equals the $p$-adic logarithm of Robert's units by the $p$-adic Kronecker limit formula (\cf\cite[Theorem 5.2, page 88]{Shalit87book}, \cite[(3.5)]{CH21}). In this case, the non-vanishing of $\cL_\Sg^*(0,\chi)$ follows from the Baker-Brumer Theorem. Since the generalization of Katz' $p$-adic Kronecker limit formula to CM fields is not available yet, the value $\cL^*_{\Sg}(0,\chi)$ remains very mysterious. Nonetheless,  \conjref{conj3} leads us to propose the following
\begin{conj}\label{conj4} If $\chi\neq \bfone$, then we have 
\[ \cL_\Sg^*(0,\chi)=\sR_\Sg(\chi)\prod_{\frakP\in\Sg_p}(1-\chi(\frakP^{-1})\rmN\frakP^{-1})\cdot L_{\rm fin}^*(0,\chi). \]
\end{conj}
Now we turn to the rank one case of \conjref{conj1} for anticyclotomic characters. We prove the following first derivative formula of $L_\Sg(s,\chi)$ in \corref{C:main}. Combined with the non-vanishing of $\cL^*_\Sg(0,\chi)$ in \thmref{T:main1}, this formula yields our \thmref{T:thmB}.  
\begin{thm}\label{T:main2}Let $\chi$ be a non-trivial anticyclotomic character.  Suppose that the Hypothesis (L) holds. If $r_{\Sg}(\chi)>0$ and $\frakP_1\in\Sg_p^{\rm irr}$, then $L
_\Sg(0,\chi)=0$, and \[\frac{L_\Sg(s,\chi)}{s}\Big\vert_{s=0}=\sL_\chi\cdot \cL_{\Sg}^*(0,\chi)\prod_{\frakP\in\Sigma_p\bksl \stt{\frakP_1}}(1-\chi(\Pbar)) .\]
\end{thm}
When $r_\Sg(\chi)>1$, \thmref{T:main2} amounts to the statement $L'_\Sg(0,\chi)=0$. Our proof of \thmref{T:main2} is a vast generalization of the method in \cite{CH21}. One of the key ingredients is to establish a leading term formula of the anticyclotomic $p$-adic $L$-function $\cL_\Sg(s,-s,\chi)$, whose proof relies on (i) a refinement of the method of Hida and Tillouine \cite{HidaTil93ASENS} to produce non-trivial CM congruence by the $p$-adic Rankin-Selberg convolution, and (ii) the clever argument in \cite{DKV18Ann} of relating the leading term of the $U_p$-eigenvalues of $\Lam$-adic generalized Hecke eigenforms to the $\sL$-invariants. Let us sketch the idea of the proof. First note that since $\chi$ is anticyclotomic, we can choose a ray class character $\brch$ such that $\brch^{1-c}=\chi$. For a well-chosen such character $\brch$, we generalize the explicit CM congruence among elliptic modular forms in \cite{CH21} to the Hilbert case and construct an explicit ordinary $\cW\powerseries{X}$-adic Hilbert cusp forms $\sH$ in the sense of \cite[\S 1.2]{Wiles88Inv} such that
\begin{itemize}
\item $\sH$ is a generalized Hecke eigenform modulo $X^{r+1}$ with $r=r_\Sg(\chi)$;
\item $\sH\pmod{X^r}$ is congruent to the Hida family $\bftheta_{\brch^c}$ of CM forms associated with $\brch^c$.
\end{itemize} 
 After an involved calculation in the Rankin-Selberg convolution (\propref{P:3RS}), we further find an explicit relation between $U_p$-eigenvalues of $\sH$ and the following product of $p$-adic $L$-functions
 \[\frac{\cL_\Sg(s,-s,\chi)}{\cL_\Sg^*(s,\chi)}\cdot\frac{\zeta_{F,p}(1-s)}{\cL_\Sg^*(s,\bfone)},\]
 where $\zeta_{F,p}(s)$ is the $p$-adic Dekekind zeta function for $F$. We prove a special case of the $p$-adic Kronecker limit formula
 \[\frac{\cL_\Sg^*(s,\bfone)}{\zeta_{F,p}(1-s)}\Big\vert_{s=0}=2^{-d}\cdot h_{K/F}\]
in \propref{P:Kronecker}, and then apply the method of \cite{DKV18Ann} to show in \thmref{T:antiregulator} that  \[\Ord_{s=0}\cL_{\Sg}(s,-s,\chi)\geq r_\Sg(\chi)\] and prove the following leading term formula of $\cL_\Sg(s,-s,\chi)$ at $s=0$ \[ \frac{\cL_{\Sg}(s,-s,\chi)}{s^{r_\Sg(\chi)}}\Big\vert_{s=0}=\sL_\chi^{\rm ac} \cdot \cL_{\Sg}^*(0,\chi)\prod_{\frakP\in\Sigma_p \backslash \Sigma_p^{\mathrm{irr}}  }(1-\chi(\Pbar)),\] 
where $\sL_\chi^{\rm ac}$ is the anticyclotomic $\sL$-invariant (See \defref{D:Gross}). We thus deduce in \corref{C:main} the first derivative formula of $L_\Sg(s,\chi)$ from the equation \[\frac{L_\Sg(s,\chi)}{s}\Big\vert_{s=0}=2\cdot \frac{\cL_\Sg(s,0,\chi)}{s}\Big\vert_{s=0}-\frac{\cL_\Sg(s,-s,\chi)}{s}\Big\vert_{s=0}\] as well as the relation between the $\cL$-invariants $\sL_\chi$ and $\sL_\chi^{\rm ac}$.
\begin{Remark}The role of the $\Sg$-Leopoldt conjecture in the research on trivial zeros of the Katz $p$-adic $L$-functions is fundamental. In the Appendix, we use Roy's strong six exponential theorem to show that if $F$ is a real quadratic field where $p$ is split and $\chi\neq\chi^c$, then there exists a $p$-ordinary CM-type $\Sg$ such that the $\Sg$-Leopoldt conjecture for $\chi$ holds. In general, to the best of the authors' knowledge, this conjecture is only known when $H$ is an abelian extension of an imaginary quadratic field where $p$ splits. \end{Remark}

This paper is organized as follows. In \subsecref{Linvariant}, we recall the general definition of $\sL$-invariants along any $\Zp$-extension of $K$. In \subsecref{S:Hilbert} we review necessary background in the theory of Hilbert modular forms and Hilbert modular varieties. In \subsecref{S:Eisenstein}, we construct a family of special Eisenstein series $\bfE_\lam$ indexed by certain Hecke characters $\lam$ of $K$ via the adelic approach described in \cite[\S 19]{Jacquet72LNM278} and prove a Damerell's formula in \thmref{C:42}. The desired $p$-adic Eisenstein measure is constructed in \propref{P:Emeasure}. In \secref{S:RS}, we apply the $p$-adic Rankin-Selberg method to construct the $\Lam$-adic form $\sH$ in \defref{D:H.3} and state in \propref{P:3RS} the explicit formula of the Rankin-Selberg convolution. In \secref{S:Galois}, we make use of the elegant ideas in \cite{DKV18Ann} to give a relation between the $U_p$-eigenvalues of generalized Hecke eigenforms and the $\sL$-invariant in \propref{P:Ribet} and prove the leading term formula of $\cL_\Sg(s,-s,\chi)$. Finally in \secref{S:RS2} we complete the proof of \propref{P:3RS}. In the Appendix, we explain the connection between $p$-adic transcendental conjectures and the $\Sg$-Leopoldt conjecture.\\

 \noindent\emph{Acknowledgements.} The first named author (A.B.) would like to thank D. Benois, M. Dimitrov, A. Maksoud and S.C Shih  for numerous stimulating discussions. The second author (M.H.) thanks K. Buyukboduk for helpful comments on an early version of the manuscript. 

\subsection*{Notation and conventions} 
Let $\Q_p(1)$ denote the $G_\Q$ representation of dimension $1$ on which $G_{\Q}$ acts by the $p$-adic cyclotomic character. We follow the geometric convention: the Hodge-Tate weight of $\Q_p(1)$ is $-1$. 

If $L$ is a local or global field of characteristic zero, let $\frako_L$ be the ring of integers of $L$. Let $G_L$ denote the absolute Galois group of $L$ and let $C_L:=L^\times$ if $L$ is local and $C_L$ be the idele class group $\A_L^\x/L^\x$ if $L$ is global. Let $\rec_L:C_L\to G_L^{ab}$ be the \emph{geometrically normalized} reciprocity law homomorphism. 

Let $L$ be a number field. If $\frakq$ is a prime ideal of $\frako_L$ (resp. $v$ is a place of $L$), 
let $L_\frakq$ (resp. $L_v$) be the completion of $L$ at $\frakq$ (resp. $v$). 
Then $\rec_{L_\frakq}:L_\frakq^\x\to G_{L_\frakq}^{ ab}$ sends a uniformizer $\uf_\frakq$ of $\frako_{L_\frakq}$ to the corresponding geometric Frobenius $\Frob_\frakq$. If $S$ is a finite set of prime ideals of $\frako_L$, let $L(S)$ be the maximal algebraic extension of $L$ unramified outside $S$ and let $G_{L,S}=\Gal(L(S)/L)$. For a fractional ideal $\fraka$ of a global field $L$, we let $\Frob_\fraka:=\prod_{\frakq}\Frob_\frakq^{n_\frakq}$ if $\fraka$ has the prime ideal factorization $\prod_\frakq \frakq^{n_\frakq}$. 

For any finite extension of global (or local) fields $E/L$, we let $\rmN_{E/L}:E\to L$ be the norm map and let $\frakd_{E} \subset \frako_E$ denotes the absolute different ideal of $E$.
 

\section{The $\sL$-invariants}\label{Linvariant}
\subsection{The $\Sigma$-Leopoldt conjecture}
Let $S_p(K)$ be the set of primes of $K$ lying above $p$. Recall that $\Sg\subset \Hom(K,\C)$ is a fixed CM type and $\Sg_p\subset S_p(K)$ is the subset of primes induced by $\iota_p\iota_\infty^{-1}\circ\Sg$ (one has $S_p(K)=\Sigma_p \coprod \Sgbar_p$). 

If $A$ is a continuous $G_K$-module, then for any subgroup $\cL\subset\rmH^1(K_p,A)=\oplus_{w\divides p}\rmH^1(K_w,A)$, we let $\rmH^1(K,A)$ be the Selmer group of $A$ with the local condition $\cL$ defined by
\[\rmH^1_\cL(K,A)=\ker\stt{\rmH^1(K,A)\to \prod_{w\ndivides p}\rmH^1(I_w,A)\times \frac{\rmH^1(K_p,A)}{\cL}}.\]
For $\frakP\in S_p(K)$, let $\rmH^1_\emptyset(K_\frakP,A)=\rmH^1(K_\frakP,A)$ and $\rmH^1_0(K_\frakP,A)=\stt{0}$. Denote by $\rmH^1_f(K_\frakP,A)$ be the finite part of $\rmH^1(K_\frakP,A)$ in the sense of Bloch and Kato. Put\begin{align*}
\cL_{a,b}=&\prod_{\frakP\in\Sg_p}\rmH^1_a(K_\frakP,A)\oplus \rmH^1_b(K_{\ol\frakP},A)\text{ for } a,b\in \stt{0,f,\emptyset}.
\end{align*}
We set
\[\rmH^1_{(a,b)}(K,A):=\rmH^1_{\cL_{a,b}}(K,A).\]
 Let  $E$ be a finite extension of $\Qp$ and $\chi:\G_K\to E^\x$ be a non-trivial ray class character of $K$. Set $H$ be the finite extension of $K$ cut out by $\chi$ and $C=\Gal(H/K)$. The Kummer map gives rise to an isomorphism 
\beq\label{E:units} \frako_H^\times [\chi]\isoto \rmH^1_{(f,f)}(K,\chi^{-1}(1)),\eeq
where $\frako_H^\x[\chi]$ is the $\chi$-isotypic component  \[\frako_H^\x[\chi]:=\{x \in E\frako_H^\times \mid 1\otimes g (x)=\chi(g)\otimes 1 \cdot x\text{ for all } g \in \Gal(H/K)\}. \] 
Let 
\beq\label{E:irrprime}\Sg_p^{\rm irr}=\stt{\frakP\in\Sg_p\mid \chi(\Pbar)=1}\eeq
be the set of \emph{irregular} primes for $\chi$ in $\Sg_p$. Let $d:=[F:\Q]$ and $r_\Sg(\chi)=\#\Sg_p^{\rm irr}$. By Dirichlet's units Theorem and a standard calculation of local Galois cohomology groups, it is not difficult to see that 
\begin{align*}
\dim_E\rmH^1_{(f,\emptyset)}(K,\chi^{-1}(1))&=d+r_\Sg(\chi),\\
\dim_E\rmH^1_{(f,f)}(K,\chi^{-1}(1))&=\sum_{w\in\Sg_p}\dim_E\rmH^1_f(K_w,\chi^{-1}(1))=d,\\
\dim_E\rmH^1_{(0,\emptyset)}(K,\chi^{-1}(1))& \leq r_\Sg(\chi) + \dim \rmH^1_{(0,f)}(K,\chi^{-1}(1)).
\end{align*}
Note that by definition, the subspace $\rmH^1_{(0,f)}(K,\chi^{-1}(1))$ is actually the kernel of the $p$-adic regulator map 
\[\log_{\Sg,p}:\rmH^1_{(f,f)}(K,\chi^{-1}(1))\iso \frako_H^\times [\chi] \to  EY_{H,\Sg},\quad x\mapsto \sum_{\sg\in\Sg_H}\log_p(\iota_p\sg(x))\sg;\]
therefore the $\Sg$-Leopoldt conjecture (\conjref{conj2}) can be rephrased as
\[\rmH^1_{(0,f)}(K,\chi^{-1}(1))=\stt{0}.\]
\begin{lm}\label{L:21}The $\Sg$-Leopoldt conjecture for $\chi$ is equivalent to 
\begin{equation}\label{Leo2}\tag{$\Sigma$-Leo}\rmH^1_{(\emptyset,f)}(K,\chi)=\ker\stt{\rmH^1(K,\chi)\to \bigoplus_{\frakQ\in\Sgbar_p}\rmH^1(I_{\frakQ},\chi)}=\stt{0}.\end{equation}
\end{lm}
\begin{proof}
Since $\rmH^1_{(f,f)}(K,\chi)=\stt{0}$ by the finiteness of the class numbers, by the Poitou-Tate exact sequence we have
\[0\to \rmH^1_{(0,f)}(K,\chi^{-1}(1))\to \rmH^1_{(f,f)}(K,\chi^{-1}(1))\to\bigoplus_{w\in\Sg_p}\rmH^1_f(K_w,\chi^{-1}(1))\to \rmH^1_{(\emptyset,f)}(K,\chi)^\vee\to 0,\]
where $V^\vee=\Hom_E(V,E)$ for a $E$-vector space $V$. We thus find that $\dim_E\rmH^1_{(\emptyset,f)}(K,\chi)=\dim_E\rmH^1_{(0,f)}(K,\chi^{-1}(1))$; the assertion follows.
\end{proof}
\subsection{The $\sL$-invariants}
We recall the general definition of $\sL$-invariants for CM fields following \cite[Definition 5.2]{BS}. The validity of the $\Sg$-Leopoldt conjecture for $\chi$ is assumed in this subsection. So we have $\dim_E\rmH^1_{(0,\emptyset)}(K,\chi^{-1}(1))=r_\Sg(\chi)$ and the decomposition \[\rmH^1_{(f,\emptyset)}(K,\chi^{-1}(1))=\rmH^1_{(0,\emptyset)}(K,\chi^{-1}(1))\bigoplus \rmH^1_{(f,f)}(K,\chi^{-1}(1)).\]
We shall define the $\sL$-invariant for $\chi$ to be certain $p$-adic regulator of the cohomology group $\rmH^1_{(0,\emptyset)}(K,\chi^{-1}(1))$. For every $\frakP\in S_p(K)$, let $\pairing_{\frakP}:\rmH^1(K_\frakP,\chi^{-1}(1))\times\rmH^1(K_\frakP,\chi)\to E$ be the local Tate pairing. For $\frakP\in\Sg_p^{\rm irr}$, $\chi|_{G_{K_{\Pbar}}}=1$, and we denote by $\Ord_{\Pbar}\in\rmH^1(K_{\Pbar},\chi)=\Hom(K_{\Pbar}^\x,E)$ the usual $\Pbar$-adic valuation. Let $X_{\Sg_p^{\rm irr}}$ be the free abelian group generated by $\ol\frakP$ with $\frakP\in\Sg_p^{\rm irr}$ and write $EX_{\Sg_p^{\rm irr}}=E\ot_\Z X_{\Sg_p^{\rm irr}}$. \begin{align*}\mathbf{O}_p:\rmH^1_{(0,\emptyset)}(K,\chi^{-1}(1))&\to EX_{\Sg_p^{\rm irr}},\quad \mathbf{O}_p(x)=\sum_{\frakP\in\Sg_p^{\rm irr}}\pair{\loc_{\Pbar}(x)}{\Ord_{\Pbar}}_{\Pbar}\Pbar.
\end{align*}
Note that $\rmH^1_f(K_{\ol\frakQ},\chi^{-1}(1))=\rmH^1(K_{\ol\frakQ},\chi^{-1}(1))$ if $\frakQ\not\in\Sg_p^{\rm irr}$, so the $\Sg$-Leopodlt conjecture implies that $\mathbf{O}_p$ is injective and hence an isomorphism. On the other hand, for each $\psi\in \rmH^1(K,E)=\Hom(G_K,E)$, we put
\begin{align*}
\mathbf{L}_p^\psi:\rmH^1_{(0,\emptyset)}(K,\chi^{-1}(1))&\to E X_{\Sg_p^{\rm irr}},\quad \mathbf{L}_p^\psi(x)=\sum_{\frakP\in\Sg_p^{\rm irr}}\pair{\loc_{\Pbar}(\psi)}{\loc_{\Pbar}(x)}_{\Pbar}\Pbar. 
\end{align*}
\begin{defn}\label{D:Gross}The $\sL$-invariant $\sL_\chi^\psi$ associated with the ray class character $\chi$ along the additive homomorphism $\psi:G_K\to\Cp$ is defined by 
\[\sL_\chi^\psi:=(-1)^{r_\Sg(\chi)}\det (\mathbf{L}_p^\psi\circ \mathbf{O}_p^{-1}|E X_{\Sg_p^{\rm irr}}).\]
By definition, the cyclotomic $\sL$-invariant $\sL_\chi:=\sL_\chi^{\ell^{\rm cyc}}$ with $\ell^{\rm cyc}=\log_p\circ\cyc$. We define the \emph{anticyclotomic} $\sL$-invariant $\sL_\chi^{\rm ac}$ by 
\[\sL_\chi^{\rm ac}:=\sL_\chi^{\ell^{\rm ac}},\quad \ell^{\rm ac}:=\log_p\circ\e_\Sg^{1-c}:G_K\to\Cp.\] 
\end{defn}



\def\cK{K}
\def\cF{F}
\def\rmf{{\bfh}}
\def\addchar{\boldsymbol \psi}
\def\cOKv{\frako_{K_v}}
\def\OF{\frako}
\def\OFv{\frako_v}
\def\Sh{Sh}
\def\cOK{\frako_K}
\def\cOF{\frako_F}
\def\universal{\boldsymbol\Psi}
\newcommand\laurent[1]{((#1))}
\section{Review of Hilbert modular varieties and Hilbert modular forms}\label{S:Hilbert}
We will use the following notation frequently. If $M$ is an finitely generated abelian group, denote by $\wh M$ the profinite completion of $M$. If $E$ is a number field, denote by $\frako_E$ the ring of integers of $E$. Let $\wh E=\wh\frako_E\ot_{\Z}\Q$ be the group of finite ideles. If $a\in \wh E^\x$, the fractional ideal of $E$ generated by $a$ is defined to be $a\frako_E:=a\wh\frako_E\cap E$, and if $\fraka$ is an integral ideal of $E$, let $\rmN\fraka:=\#(\cO_E/\fraka)$ be the norm of $\fraka$. 
\subsection{Hilbert Shimura varieties}
\def\Ig{{\rm Ig}}
Let $F$ be a totally real field and let $\OF=\cOF$. Recall that a Hilbert-Blumenthal abelian variety by $\OF$ (or a HBAV for short) over a scheme $S$ is an abelian scheme $A$ over $S$ of relative dimension $d$ and equipped with an embedding of rings $\iota : \OF\to \End_S(A)$ such that $\Lie(A)$ is an $\OF\ot\cO_S$-module locally free of rank $1$. 

 Let $\frakc$ be a fractional ideal of $F$ and let $\frakc^+$ be the set of totally positive elements in $\frakc$. The dual abelian scheme $A^{\rm t}{}_{/S}$ of $A$ has a canonical HBAV structure. Let $\cP_A:=\Hom_{\OF}(A,A^{\rm t})_{\rm sym}$ be the rank $1$ projective $\OF$-module of symmetric $\OF$-linear homorphisms and $\cP_A^+\subset \cP_A$ be the positive cone of polarizations. A $\frakc$-polarization of $A$ is an $\OF$-linear map $\lam:\frakc\to\cP_A$ sending $\frakc_+$ to $\cP_A^+$ such that the morphism \[\lam:A\ot_{\OF}\frakc\iso A^{\rm t},\quad x\ot \al\mapsto \lam(\al)x.\] is an isomorphism of abelian schemes.

Let $N$ be a positive integer and let $\zeta_N=e^\frac{2\pi\sqrt{-1}}{N}$ be a fixed primitive $N$-th root of unity. Suppose that $N\geq 4$ is coprime to $p$ and the fractional ideal $\frakc$ is coprime to $pN$. Consider the moduli functor $\cE_{\frakc,N,p^n}$ over $\Z[\frac{1}{N\Delta_F},\zeta_N]$ classifying HBAV with $\frakc$-polarizations and level $\Gamma(N)\cap\Gamma_1(p^n)$-structures. More precisely, for a basis scheme $S$ over $\Z[\frac{1}{N\Delta_F},\zeta_N]$, $\cE_{\frakc,N,p^n}(S)$ is the set of isomorphism classes of the quadruple $(A,\lam,\iota,\eta_N,j_p)_{/S}$, where 
\begin{itemize}
\item $A$ a HBAV over $S$ and $\lam$ is a $\frakc$-polarization of $A$.
\item $\eta_N:(\OF\ot_{\Z}N^{-1}\Z/\Z)\oplus (\frakd_F^{-1}\ot_{\Z}\mu_N)\iso A[N]$ is an $\OF$-linear isomorphism
\item $j:\frakd_F^{-1}\ot_{\Z}\mu_{p^n}{}_{/S}\hookto A[p^n]$ is an $\OF$-linear closed immersion of group schemes over $S$.
\end{itemize}
The functor $\cE_{\frakc,N,p^n}$ is represented by a quasi-projective and smooth scheme $\frakM(\frakc,N,p^n)$ of relative dimension $[F:\Q]$ over $\Z[\frac{1}{N\Delta_F},\zeta_N]$ (\cite{Rap78Comp}, \cite{Chai90Ann} and \cite{DP94Comp}). We put
\[\frakM(\frakc,N,p^\infty)=\prolim_{n\to \infty}\frakM(\frakc,N,p^n).\]
\subsection{Complex uniformizaton}
Now we recall the complex analytic structure of $\frakM(\frakc,N,p^n)(\C)$. 
Let $\frakH_F\subset F\ot_{\Q}\C$ be the upper half plane defined by \[\frakH_F=\stt{\tau=x+\sqrt{-1}y\in F\ot\C\mid x\in F\ot\R,\,y\in (F\ot\R)_+}.\] 
To each $x=(\tau_x,g_x)\in \frakH_F\times \GL_2(\wh F)$, we associate the quadruple $(A_x,\lam_x,\eta_x^{(p)},j_x)$ over $\C$ as follows:  Let $V=Fe_1\oplus Fe_2$ be the two-dimensional vector space over $F$ with the symplectic pairing 
\[\pair{ae_1+be_2}{ce_1+de_2}=\Tr_{F/\Q}(bc-ad).\]
Let $\GL_2(F)$ act on $V$ from the right via
\[(xe_1+ye_2)\pMX{a}{b}{c}{d}=(xa+yc)e_1+(xb+yd)e_2.\]
For $\tau\in \frakH_F$, define the $\OF$-linear isomorphism ${\rm q}_\tau:V\ot_{\Q} \R\isoto F\ot \C$ by 
\[{\rm q}_\tau(xe_1+ye_2)=x\tau+y,\quad 
(x,y\in F\ot \R).\]
Let $\bfL$ be the lattice $\OF e_1\oplus \frakd_F^{-1}e_2$. For $g\in \GL_2(\wh F)$, put
\[\bfL_g:=(\bfL\ot_\Z\wh\Z) g^{-1}\cap V.\]
Given $(\tau_x,g_x)\in \frakH_F\times \GL_2(\wh F)$, we define the complex torus
\[A_x(\C)=(F\ot\C)/L_x,\quad L_x={\rm q}_{\tau_x}(\bfL_{g_x}).\]
Let $\frakc_x$ be the fractional ideal $\det g_x\OF$ generated by $\det g_x$.
For $a\in \frakc_x$, the pairing $E_a(z,w):=\Tr_{F/\Q}(a\Im(\ol{z}w)/\Im\tau_x)$ defines a  Riemann form on $F\ot\C$ if $a\in(\frakc_x)_+$. Therefore $A_x(\C)$ carries a structure of HBAV $A_x$ over $\C$ with the $\OF$-action induced by ${\rm q}_{\tau_x}$. For $a\in\frakc_x$, let $\lam_x(a):A_x\to A_x^{\rm t}$ be the isogeny induced by $E_a(-,-)$, and the map $\lam_x:a\mapsto \lam_x(a)$ induces a $\frakc_x$-polarization $\lam:(\frakc_x,(\frakc_x)_+)\to (\cP_{A_x},\cP_{A_x}^+)$. Define the level $\Gamma(N)\cap\Gamma_1(p^n)$-structures by 
\begin{align*}
\eta_x&: (\OF\ot N^{-1}\Z/\Z)\oplus (\frakd_F^{-1}\ot \mu_N)\iso A_x[N],\\ &\qquad(a\ot N^{-1},b\ot \zeta_N)\mapsto (\frac{a}{N}e_1g^{-1}+\frac{b}{N}e_2g^{-1})\pmod{\bfL_x},\\
j_x&: \frakd_F^{-1}\ot \mu_{p^n}\hookto A_x[p^n],\quad  a\ot \zeta_{p^n}\mapsto  a/p^n e_2 g_p^{-1}\pmod{\bfL_g}.
\end{align*}
This way $(A_x,\lam_x,,\eta_x,j_x)$ is a HBAV with $\frakc_x$-polarization $\lam_x$ and $\Gamma(N)\cap\Gamma_1(p^n)$-structures, and the isomorphism class $[(A_x,\lam_x,,\eta_x,j_x)]$ gives rise to a point in $\frakM(\frakc_x,N,p^n)(\C)$. 

Fix an idele $\bfc\in \wh F^{(pN)\x}$ such that $\bfc\OF=\frakc$ and put \[\GL_2(\wh F)^{(\bfc)}=\stt{g\in \GL_2(\wh F)\mid \det g=\bfc}.\] For any open-compact subgroup in $\GL_2(\wh F)$, consider the complex Hilbert modular variety $\Sh(\bfc,U)_{/\C}$ defined by
\[\Sh(\bfc,U)(\C):=\SL_2(F)\bksl \frakH_F\times \GL_2(\wh F)^{(\bfc)}/(U\cap \SL_2(\wh F)).\]
For any positive integer $n$, put
\[U_N(p^n)=\stt{g=\pMX{a}{b}{c}{d}\in\GL_2(\wh\OF)\mid a,d\in 1+p^nN\wh\OF,\,c\in \frakc p^n\wh\frakd_F,\, b\in N \wh\frakd_F^{-1}}.\] 
Assume $N$ is large enough such that 
\beq\label{E:neat}U_N(1)\text{ is neat and }\det (U_N(1))\cap
\OF_+^\x\subset (U_N(1)\cap \OF^\x)^2.\eeq
With the above condition, we see that the map 
\beq\label{E:shi_iso}\frakH_F\times \GL_2(\wh F)^{(\bfc)}\to \frakM(\frakc,N,p^n)(\C),\quad  x=(\tau_x,g_x)\mapsto [(A_x,\lam_x,\eta_x,j_x)]\eeq
induces an isomorphism \[\Sh(\bfc, U_N(p^n))(\C)\isoto \frakM(\frakc,N,p^n)(\C).\]

\subsection{Hilbert modular forms}

\subsubsection*{Geometric modular forms}
Let $k$ be an integer. 
Let $R_0$ be a $\Z[\frac{1}{N\Delta_F},\zeta_N]$-algebra. Let $\cM_k(\frakc,N,p^n;R_0)$ be the space of geometric Hilbert modular forms of weight $k\Sg$ and level $\Gamma(N)\cap\Gamma_1(p^n)$ defined over $R_0$. Recall that an element $f\in\cM_k(\frakc,N,p^n;R_0)$ is a rule $f$ which assigns a $\frakc$-polarized quadruple $(A,\lam,\eta,j)$ of HBAV over a $R_0$-algebra $R$  together with an $(\OF\ot R)$-basis $\om$ of $\rmH^0(A,\Omega^1_{A/R})$ (or equivalently an isomorphism $\om^*:\Lie(A_x)\iso \frakd_F^{-1}\ot R$) an element $f(A,\lam,\eta,j,\om)\in R$ such that 
\begin{itemize}
\item[(M1)] the value $f(A,\lam,\eta,j,\om)$ only depends on the isomorphism class of $(A,\lam,\eta,j,\om
)$,
\item[(M2)] $f(A,\lam,\eta,j,\al^{-1}\om)=\al^{k\Sg}\cdot f(A,\lam,\eta,j; \om)$ for any $\al\in (\OF\ot R)^\x$,
\item[(M3)] if $u:R\to R'$ is a homomorphism of $R_0$-algebras, then \[f((A,\lam,\eta,j,\om)\times_{R}R')=u(f(A,\lam,\eta,j,\om)).\]
\end{itemize}
We put 
\[\cM_k(\frakc,N;R_0)=\dirlim_n \cM_k(\frakc,N,p^n;R_0).\]

We briefly recall the $q$-expansions of geometric modular forms. If $R$ is any ring and $M$ is a lattice in $F$, let $R\powerseries{M_+}$ be the ring of all formal series \[\sum_{\beta\in M_+\cup\stt{0}}a_\beta q^\beta,\quad a_\beta\in R.\] We fix a set $\cC=\stt{l_1,\dots,l_{[F:\Q]}}$ consisting of $[F:\Q]$-linearly independent $\Q$-linear forms $l_i:F\to\Q$ such that $l_i(F_+)\subset \Q_+$. Define $M_{\cC_+}=\stt{x\in M \mid l_i(x)\geq 0,\,l_i\in \cC}$. Denote by $R\powerseries{M,\cC}$ the ring of all formal series $\sum_{\beta\in M_{\cC_+}} a_\beta q^\beta$. It is clear that $R\powerseries{M_+}\subset R\powerseries{M,\cC}$ for any choice of $\cC$. Let \[R\laurent{M,\cC}:=R\powerseries{M,\cC}[U^{-1}],\quad U=\stt{q^\beta\mid \beta\in M_+}.\]  Let $(\fraka,\frakb)$ be a pair of fractional ideals of $F$ such that $\frakc=\fraka\frakb^{-1}$. To any geometric modular form $f\in \cM_k(\frakc,N;R_0)$, we can associate a power series $f|_{(\fraka,\frakb)}(q)\in R_0\laurent{\fraka\frakb,\cC}$ obtained by evaluating $f$ at the Tate HBAV ${\rm Tate}_{\fraka,\frakb}(q)$ together with a canonical differential over $\Z[\zeta_N]\laurent{\fraka\frakb,\cC}$ (\cf\cite[1.1 and (1.2.12)]{Katz78Inv} and \cite[Chap 5, \S 2]{Goren02CRM}). This power series $f|_{(\fraka,\frakb)}$ is called the $q$-expansion of $f$ at the cusp $(\fraka,\frakb)$.
\subsubsection*{Classical modular forms}Define the automorphy factor by
\beq\label{E:J.2}J:\GL_2(F\ot \R)\times \frakH_F\to (F\ot\C)^\x,\quad J(\pMX{a}{b}{c}{d}):=c\tau+d.\eeq
For any idele $\bfc\in \wh F^{\x}$ and an open-compact subgroup $U\subset \GL_2(\wh F)$, we let $M_k(\bfc,U)$ denote the space of classical modular forms of weight $k\Sg$ and level $U$, which consists of  holomorphic functions $\bff:\frakH_F\times \GL_2(\wh F)^{(\bfc)}\to\C$ such that \beq\label{E:mf.2}\bff(\al(\tau,g)u)=J(\al,\tau)^{k\Sg}f(\tau,g)\text{ for }\al\in \SL_2(F);\,u\in U\cap\SL_2(\wh F). \eeq
Letting $\bfc$ be the idele such that $\bfc\in \wh F^{(pN)\x}$ with $\bfc\OF=\frakc$, we put \[M_k(\frakc,N,p^n;\C)=M_k(\bfc,U_N(p^n)).\]
Give a pair of prime-to-$pN$ fractional ideals $(\fraka,\frakb)$ with $\fraka\frakb^{-1}=\frakc$, choose $a,b\in \wh F^{(pN)\x}$ with $a\OF=\fraka$ and $b\OF=\frakb$. Each $\bff\in M_k(\frakc,N,p^n;\C)$ admits the Fourier expansion 
\[\bff(\tau,\pDII{b^{-1}}{a})=c_0(\bff)+\sum_{\beta\in (N^{-1}\fraka\frakb)_{+}}c_\beta(\bff)e^{2\pi\sqrt{-1}\Tr(\beta\tau)},\]
where $\Tr=\Tr_{F/\Q}:F\ot\C\to \C$. The $q$-expansion of $f$ at the cusp $(\fraka,\frakb)$ is defined to be the power series
\[\bff|_{(\fraka,\frakb)}(q)=c_0(\bff)+\sum_{\beta\in (N^{-1}\fraka\frakb)_+}c_\beta(\bff)q^\beta\in \C\powerseries{(N^{-1}\fraka\frakb)_+}.\]
If $R_0$ is a $\Z[\frac{1}{N\Delta_F},\zeta_N]$-algebra in $\C$, let
\[M_k(\frakc,N,p^n;R_0)=\stt{\bff\in M_k(\frakc,N,p^n;\C)\mid \bff\vert_{(\OF,\frakc^{-1})}(q)\in R_0\powerseries{(N^{-1}\frakc^{-1})_+}}.\]
We put 
\[M_k(\frakc,N;R_0)=\dirlim_n M_k(\frakc,N,p^n;R_0).\]
By \cite[Lemme 6.12]{Rap78Comp} (or \cite[(1.6.3)]{Katz78Inv}) and the $q$-expansion principle (\cf \cite[Theoreme 6.7]{Rap78Comp} and \cite[(1.2.16)]{Katz78Inv}), we have the isomorphism
\beq\label{E:GAGA}\cM_k(\frakc,N;R_0)\isoto M_k(\frakc,N;R_0),\quad f\mapsto f^{\rm an}(x)=f(A_x,\eta_x,j_x,2\pi\sqrt{-1}\om_x^{\rm an}).\eeq
Here $\om_x^{\rm an}$ is given by the canonical isomorphism $\Lie(A_x)\isoto F\ot\C=\frakd_F^{-1}\ot\C$. 


\subsection{The Igusa tower and $p$-adic modular forms}
Let $R_0$ be a $p$-adic $\Z[\frac{1}{N\Delta_F},\zeta_N]$-algebra, \ie $R_0$ is a $\Z[\frac{1}{N\Delta_F},\zeta_N]$-algebra and $R_0=\prolim R_0/p^m R_0$. The Igusa tower $\Ig(\frakc,N)_{/R_0}$ is the formal scheme 
\[\Ig(\frakc,N)_{/R_0}:=\dirlim_m\prolim_n \frakM(\frakc,N,p^n)_{/R_0/p^m R_0}.\]

We let  $\widehat{\frakM}^{\Ord}(\frakc,N)$ be the ordinary locus (i.e., the locus where the Hasse invariant is invertible) of the formal completion of $\frakM(\frakc,N,p^0)_{/ \Z_p}$ along the special fiber. The natural projection $\pi_{\Ig}: \Ig(\frakc,N)_{/\Z_p} \to \widehat{\frakM}^{\Ord}(\frakc,N)$ given by forgetting the $\Gamma_1(p^\infty)$-level structure is Galois with  group $(\frako \otimes \Z_p)^\times$.

The space $V(\frakc,N,R_0)$ of $p$-adic modular forms of tame level $\Gamma(N)$ is defined to be the ring of formal functions on the Igusa tower $\Ig(\frakc,N)$. In other words, $V(\frakc,N,R_0)$ is the ring of global sections of the structure sheaf of the formal scmee $\Ig(\frakc,N)_{/R_0}$. To any $p$-adic modular form $f\in V(\frakc,N,R_0)$ and $(\fraka,\frakb)$ fractional ideals with $\fraka\frakb^{-1}=\frakc$, one can also associate a power series $f|_{(\fraka,\frakb)}$ in the $p$-adic completion $R_0\wh{\laurent{\fraka\frakb,\cC}}$ by evaluating at the Tate HBAV ${\rm Tate}_{\fraka,\frakb}(q)$ over $R_0\wh{\laurent{\fraka\frakb,\cC}}$. 

Let $R$ be a $p$-adic $R_0$-algebra containing the Galois closure of $\OF$ in $\Cp$. For any point $x=[(A_x,\lam_x,\eta_x,j_x)]\in \Ig(\frakc,N)(R)$, the $\Gamma_1(p^\infty)$-level structure $j_x$ induces a canonical isomorphism $\hat j_x:\frakd_F^{-1}\ot\wh\bbG_m{}_{/R}\to \wh A_x$ of formal groups, which in turn produces the isomorphism $(\hat j_x)_*:\frakd_F^{-1}\ot R\isoto \Lie(A_x)$ of $(\OF\ot R)$-modules by passage to Lie algebras. 
Define \beq\label{E:HT}\pi_{\rm HT}(x):=(\hat j_x)_*^{-1}: \Lie(A_x)\isoto\frakd_F^{-1}\ot R\eeq
to be the inverse of $(\hat j_x)_*$ (\cf\cite[(1.10.11)]{Katz78Inv}). 

If $f\in \cM_k(\frakc,N;R_0)$ is a geometric modular form of level $\Gamma_1(p^\infty)$, we define the $p$-avatar $\wh f\in V(\frakc,N;R_0)$ is defined as follows for any $x\in \Ig(\frakc,N)(R)$, define \beq\label{E:padicavatar}\wh f(x):=f(x,\pi_{\rm HT}(x))\eeq
which preserves the $q$-expansions
(\cf\cite[Theorem (1.10.15)]{Katz78Inv}).


\subsection{CM points}\label{S:CMpt}
Let $K$ be a totally imaginary quadratic extension of $F$. Let $\frakd_{K/F}=\rmN_{K/F}(\frakD_{K/F})$ be the relative discriminant of $K/F$. Suppose that $N$ is prime to $\frakd_{K/F}$. Decompose $N\cOK=\frakN^+\frakN^-$, where $\frakN^+$ is a product of primes split in $K$ and $\frakN^-$ is a product of primes inert in $K$.
By the approximation theorem, we can choose $\skewhf\in\cK$ such that
\begin{itemize}
\item[(d1)] $\ol{\skewhf}=-\skewhf$ and
$\Im\sg(\skewhf)>0$ for all $\sg\in\Sg$,
\item[(d2)] $\frakc(\cOK):=\frakd_{\cF}^{-1}(2\skewhf\frakd_{\cK/\cF}^{-1})$ is prime to $pN\frakd_{\cK/\cF}$. \end{itemize}
Consider the isomorphism $q_\skewhf:V\isoto K$ defined by $q_\skewhf(ae_1+be_2)=a\skewhf+b$. Define the embedding $\iota:\cK\hookto \Mat_2(\cF)$ by
\[\iota(a\skewhf+b)=\pMX{b}{-\skewhf\ol{\skewhf} a}{a}{b}.\]
Note that \begin{equation}\label{propertyqv}q_\skewhf((0,1)\iota(\al))=\al \text{ 
and } q_\skewhf(x)\al=q_\skewhf(x\iota(\al)) \text{ for $\al\in \cK$ and $x\in V$}. \end{equation} 

The CM type $\Sg$ gives the isomorphism
\begin{align*}i_\Sg\colon K\ot \R\isoto F\ot\C\iso\C^\Sg\\
\al\in K\mapsto (\sg(\al))_{\sg\in\Sg}.
\end{align*}Let ${\rm q}_{\skewhf_\Sg}:=i_\Sg\circ q_\skewhf:V\ot_\Q\R\isoto  F\ot\C$ with $\skewhf_\Sg:=i_\Sg(\skewhf)\in \frakH_F$. 

For a finite non-split place $v$ of $K/F$, let $\stt{1,\bftheta_v}$ be an $\OFv$-basis of $\cOKv$. We further fix a decomposition $\frakN^+=\Csplit {\Csplit}_c$ with $(\Csplit,\Csplit_c)=1$ and $ \Csplit \subset \Csplit_c^c$. If $v\divides p\frakN^+$ with $w\divides\Sg_p\Csplit$, \ie $w\in\Sg_p$ or $w\divides\Csplit$, then we have $\cOKv=\OFv e_w\oplus \OFv e_{\wbar}$, where $e_{w}$ and $e_{\wbar}$ are the idempotents of $\cOKv$ corresponding to $w$ and $\wbar$ respectively. Then $\stt{e_w,e_{\wbar}}$ is an $\OFv$-basis of $\cOKv$. In this case, we let $\skewhf_w\in \cF_v$ such that $\skewhf=-\skewhf_we_{\wbar}+\skewhf_w e_w$.
We shall also fix a finite idele $d_\cF=(d_{\cF_v})\in \wh F$ such that $d_\cF\OF=\frakd_F$. By condition (d2), we choose \[d_{F_v}=\frac{2\skewhf}{\bftheta_v-\ol{\bftheta_v}}\text { if }v\divides \frakd_{K/F},  d_{F_v}= 2\skewhf_v \text { if }v\divides \frakN^{-}\text{ and } \quad d_{\cF_v}=-2\skewhf_w\text{ if }w\divides \Sg_p\Csplit.\]
Now we make a particular choice of a local basis $\stt{e_{1,v},e_{2,v}}$ of $K_v$ over $\Fv$ for each finite
place $v$ of $\cF$ such that
$\cOKv=\OF_v e_{1,v}\oplus \OFv^* e_{2,v}$. \begin{itemize}
\item If $v\ndivides pN$, $\stt{e_{1,v},e_{2,v}}$ is taken to be $\stt{\skewhf,1}$ for all but finitely many $v$. 
\item If $v\divides \frakN^-$ is non-split, let
$\stt{e_{1,v},e_{2,v}}=\stt{\bftheta_v,d_{\cF_v}\cdot 1}$.
\item If $v\divides p\frakN^+$, let $\stt{e_{1,v},e_{2,v}}=\stt{e_{\wbar},d_{\cF_v}\cdot e_{w}}$ with $w\divides\Sg_p\Csplit$.\end{itemize}
Let $\cmpt_v$ be
the element in $\GL_2(\cF_v)$ defined by 
\beq\label{E:1}q_\skewhf(e_i\cmpt_v^{-1})=e_{i,v}.\eeq
In view of our choices of $d_{\Fv}$, one verifies that \[\det\cmpt_v=1\iff \pair{q_\skewhf^{-1}(e_{2,v})}{q_\skewhf^{-1}(e_{1,v})}=1\text{ for }v\divides pN.\]
In addition, the matrix representation of $\cmptv_v$ according to $\stt{e_1,e_2}$ for $v\divides p\frakN^+$ is given by 
\beq\begin{aligned}\label{E:cm.N}
\cmpt_v&= \pMX{\frac{d_{\cF_v}}{2}}{-\onehalf}{\frac{-d_{\cF_v}}{2\skewhf_w}}{\frac{-1}{2\skewhf_w}}=\pMX{-\skewhf_w}{-\onehalf}{1}{\frac{-1}{2\skewhf_w}}
\text{ if }v=w\wbar\text{ and }w\divides \Sg_p\Csplit.
\end{aligned}\eeq
Let $\bfh$ be the set of finite places of $F$, and put $\cmpt_{\rmf}=\prod_{v\in\rmf}\cmpt_v$. By definition, we have
\begin{equation}\label{zhatlattice}
q_\skewhf((\sL\ot \Zhat)\cdot\cmpt_{\rmf}^{-1})=\wh\cOK.\end{equation}
\begin{defn}\label{D:CMpts}
Let $\frakA$ be a fractional ideal of $K$ coprime to $pN$ and put \[\frakc(\frakA)=\frakc(\cOK)\rmN(\frakA).\] Define the CM point \beq\label{E:CMpt}x(\frakA):=[(\skewhf_\Sg,\iota(a)\cmpt_\bfh)]=[(A_\frakA,\lam_\frakA,\eta_\frakA,j_\frakA)]\in \frakM(\frakc(\frakA),N)(\C),\eeq
where $a\in \wh K^{(pN)\x}$ is an idele with $a\cOK=\frakA$. By the theory of CM abelian varieties, the triple $(A_\frakA,\lam_\frakA,\eta_\frakA)$ is defined over a number field $L$ (\cite[18.6,\,21.1]{Shimura98ABVCM}), which in turn descends to a triple the triple $(A_\frakA,\lam_\frakA,\eta_\frakA)$ defined over $\cO_L$ by a theorem of Serre-Tate. On the other hand, since the CM type $\Sg$ is $p$-ordinary with respect to $\iota_p$, the reduction $A_\frakA\ot_{\cO_L,\iota_p}\Fpbar$ is an ordinary abelian variety (\cf\cite[5.1.27]{Katz78Inv}), and hence the $\Gamma_1(p^\infty)$-level structure $j_\frakA$ descends to a $\Gamma_1(p^\infty)$-level structure over $\wh\Z_p^{\rm ur}$. We conclude that the CM point $x(\frakA)$ can be descended to a $\cW$-point $x(\frakA)$ in the Igusa tower $\Ig(N,\frakc(\frakA))(\cW)$ for some finite extension $\cW$ of $\wh\Z_p^{\rm ur}$ (\cite[5.1.19-20]{Katz78Inv}). \end{defn}

\section{Eisenstein series and the improved Katz $p$-adic $L$-functions}\label{S:Eisenstein}

\subsection{Notation and convention}
 Denote by $\addchar_{\Q}=\prod\addchar_{\Q_v}:\A_\Q/\Q\to\C^\x$ the unique additive character with $\addchar_{\R}(x)=\exp(2\pi\sqrt{-1}x)$ and by $\Abs_{\A_Q}=\prod\Abs_{\Q_v}:\A_\Q^\x/\Q^\x\to \R_+$ the absolute value so that $\abs{\cdot }_{\R}$ is the usual absolute value on $\R$. Let $E$ be a number field. Let $\frakd_E$ be the different of $E/\Q$ and $\Delta_E=\rmN_{E/\Q}(\frakd_E)$ be the absolute discriminant of $E$. Let $\addchar_E:=\addchar_\Q\circ\Tr_{E/\Q}:\A_F\to \C^\x$. Denote by $\Abs_{\A_E}=\Abs_{\A_\Q}\circ\rmN_{E/\Q}:\A_E^\x\to\R_+$. If $a\in\A_E^\x$ is an idele and $v$ is a place of $E$, denote by $a_v\in E_v^\x$ the component of $a$ at $v$. If $\al\in E^\x$ and $S$ is a finite set of places of $E$, denote by $(\al)_S\in \A_E^\x$ the idele whose $v$-component is $\al$ if $v\in S$ and $1$ if $v\not\in S$. 

\subsubsection*{$L$-functions and epsilon factors}If $\chi:C_E\to\C^\x$ is an idele class character (or Hecke character) of $L^\x$ unramified outside $S$. If $v$ is a place of $E$, let $\chi_v:E_v^\x\to\C^\x$ be the local component of $\chi$ at $v$. Let $L(s,\chi_v)$ and $\varepsilon(s,\chi_v)=\varepsilon(s,\chi_v,\addchar_{E_v})$ be the local $L$-factor and local epsilon factor of $\chi_v$ in \cite[3.1]{Tate79Corvallis}. Let $L(s,\chi)=\prod_v L(s,\chi_v)$ be the complete $L$-function of $\chi$ and $\varepsilon(s,\chi):=\prod_v\varepsilon(s,\chi_v)$ be the global epsilon factor \cite[3.5.1-2]{Tate79Corvallis}. Denote by $L_{\rm fin}(s,\chi):=\prod_{v<\infty} L(s,\chi_v)$ be the $L$-function without archimedean factors. If $\chi=\bfone$ is the trivial character of $C_E^\x$, then we put 
\[\zeta_{E_v}(s)=L(s,\bfone_v);\quad \zeta_E(s)=L(s,\bfone). \]In particular, \[\zeta_\R(s)=\Gamma_\R(s)=\pi^{-\frac{s}{2}}\Gamma(\frac{s}{2}),\quad \zeta_\C(s)=\Gamma_\C(s)=2(2\pi)^{-s}\Gamma(s).\] and $\zeta_\Q(2)=\pi/6$ with this definition. 

\subsubsection*{Measures}Let $v$ be a place of $E$. If $v$ is non-archimedean, let $\rmd x_v$ be the self-dual Haar measure on $\frako_{E_v}$ with respect to the additive character $\addchar_{E_v}$, \ie $\rmd x_v$ is the measure such that $\vol(\frako_{E_v},\rmd x_v)=\abs{\frakd_{E}}_{E_v}^\onehalf$. If $v$ is a real place, then $\rmd x_v$ is the Lebesgue measure. If $v$ is a complex place, then $\rmd x_v$ is twice the Lebesgue measure. Let $\rmd^\x x_v$ be the Haar measure on $E_v^\x$ defined by $\rmd^\x x_v=\zeta_{F_v}(1)\abs{x_v}_{E_v}^{-1}\rmd x_v$ (so $\vol(\frako_{E_v}^\times,\rmd x^\times_v)=\abs{\frakd_{E}}_{E_v}^\onehalf$). The product measure $\rmd x=\prod_v \rmd x_v$ is the Tamagawa measure on $\A_E$. Denote by $\cS(E_v)$ the space of \BS functions on $E_v$.  For $\phi\in \cS(E_v)$, the Fourier transform $\wh\phi\in \cS(E_v)$ is defined by 
\[\wh\phi(y)=\int_{E_v^\x}\phi(x_v )\addchar_{E_v}(yx_v)\rmd x_v.\]
If $\chi:E_v\to\C^\x$ is a continuous character, the Tate integral $Z(\phi,\chi,s)$ is defined by 
\[Z(\phi,\chi,s)=\int_{E_v^\times}\phi(x_v)\chi(x_v)\abs{x_v}_{E_v}^s\rmd^\x x_v\quad (s\in \C).\]

\subsubsection*{Characters}We fix an isomorphism $\iota_p:\C\iso\Cp$ once and for all. Let ${\rm I}_E$ be the set of all embeddings $E\to\C$. Each $k\in \Z[{\rm I}_E]$ shall be regarded as a character $k: (E\ot \R)^\x\to\C^\x$ and $(E\ot\Qp)^\x\to \Cp^\x$ via $\iota_p$. 
If $\chi$ is a $p$-adic character of $G_{E,S}$, we say $\chi$ is locally algebraic of weight $k\in\Z[{\rm I}_E]$ if $\chi(\rec_{E}(x))=x^k$ for all $x\in (E\ot\Qp)^\x$ sufficiently close to $1$. We shall also view $\chi$ as an idele class character via $\rec_E$ and still denote by $\chi$ if there is no fear for confusion. In particular, a primitive ray class character $\chi$ modulo $\frakc$ shall be identified with an idele class character $\chi$ of $E$ of conductor $\frakc$. If $\chi:C_E\to\C^\x$ is an idele class character of $F$ of level $\frakn$, we put
  \beq\label{E:13}\chi(s\cO_E):=\chi(s)\text{ for }s\in\wh E^{(\frakn)\x}.\eeq 
Therefore, \[\chi(\frakq):=\chi(\Frob_\frakq)=\chi_\frakq(\uf_\frakq)\text{ for }\frakq\not\in S.\]
 To an idele class character $\chi: C_E\to \C^\x$ of infinity type $k\in\Z[{\rm I}_E]$, \ie $\chi_\infty(a_\infty)=a_\infty^k$, we can associate a $p$-adic idele class character $\wh\chi: C_E\to \Zbar_p^\x$ defined by \beq\label{E:pavatar}\wh\chi(a)=\chi(a)\iota_p(a_\infty^{-k})a_p^k.\eeq
We call $\wh\chi$ the $p$-adic avatar of $\chi$. The $p$-adic character $\wh\chi$ factors through $\rec_E:C_E\to G_E^{ab}$, and hence gives rise to a locally algebraic $p$-adic Galois character $\wh\chi:G_E\to\Cp^\x$ of weight $k$.
 
 Let $\Om:\Gal(\Q(\zeta_p)/\Q)\to\C^\x$ be the character such that $\iota_p\circ\Om$ is the $p$-adic \Teich character. Let $\Om_E=\Om|_{\Gal(E(\zeta_p)/E)}$. Viewed as an idele class character, we have $\Om_E=\Om\circ\rmN_{E/\Q}$.

\subsection{Generalities on Eisenstein series}
Let $F$ be a totally real number field as before. Denote by $\bfh$ the set of finite places of $F$. We shall identify $\Sg$ with the set of infinite places of $F$. We recall a well-known construction of adelic Eisenstein series on $\GL_2(\A_F)$ in \cite[\S 19]{Jacquet72LNM278} (\cf \cite[p. 351]{BumpGrey}). Let $(\chi_1,\chi_2)$ be a pair of Hecke characters of $F^\x$, \ie quasi-characters of $\A_F^\x\bksl F^\x$.  For each place $v$, let $I(\chi_{1,v},\chi_{2,v};s)$ be the space of smooth (and $\SO_2(\R)$-finite if $v$ is archimedean) functions $f:\GL_2(\Fv)\to\C$ such that 
\[f(\pMX{a}{b}{0}{d} g)=\chi_1(a)\chi_2(d)\abs{\frac{a}{d}}_v^{s+\onehalf}f(g).\]
Here $\Abs_v=\Abs_{\Fv}$. Let $B$ be the upper triangular subgroup of $\GL_2$. Define the adelic Eisenstein series $E_\A(-,f):\GL_2(F)\bksl \GL_2(\A_F)\to \C$ associated with $f$ by 
\beq\label{E:Eisenstein.2}E_\A(g,f):=\sum_{\gamma\in B(F)\bksl \GL_2(F)}f(\gamma g),\quad g\in\GL_2(\A_F),\eeq
provided the sum is convergent. 
For $f=\ot f_v\in I(\chi_1,\chi_2;s):=\ot'_v I(\chi_{1,v},\chi_{2,v};s)$ and $\beta\in F^\x$, define the local Whittaker function $W_\beta(-,f_v):\GL_2(F_v)\to\C$ by 
\[W_\beta(g_v,f_v)=\int_{F_v}f_v(\pMX{0}{-1}{1}{0}\pMX{1}{x_v}{0}{1}g_v)\addchar_{\Fv}(-\beta x_v)\rmd x_v.\]
 The global Whittaker function is defined by 
\[W_\beta(g,f)=\prod_v W_\beta(g_v,f_v).\]
Then we have Fourier expansion
\beq\label{E:FC3}E_\A(g,f)=f(g)+Mf(g)+\sum_{\beta\in F^\x}W_\beta(g,f),\eeq
where $Mf(g)$ is obtained by the analytic continuation of the intertwining integral
\[Mf(g)=\int_{\A_F}f(\pMX{0}{-1}{1}{0}\pMX{1}{x}{0}{1} g)\rmd x,\,g\in \GL_2(\A_F)\]
(\cf\cite[(7.15)]{BumpGrey}).
\subsection{The construction of adelic Eisenstein series}\label{SS:43}
Let $\frakC$ be an integral ideal of $K$ such that 
\[(\frakC,p\frakd_{K/F})=1.\]
Decompose $\frakC=\frakC^+\frakC^-$, where $\frakC^+=\Csplit\Csplit_c$ is a product of split primes in
$K/F$ such that $(\Csplit,\Csplit_c)=1$ and $\Csplit\subset \Csplit_c^c$, and $\frakC^-$ is a product of inert primes in $\cK/\cF$. Let $\lam$ be a Hecke character of $K$ of infinity type $k\Sg$, namely 
\[\lam_\infty(z)=\prod_{\sg\in\Sg}\sg(z)^k\text{ for }z\in K^\x.\]
Let $\frakC_\lam$ be the conductor of $\lam$. We assume that
\[\frakC_\lam=\frakC\prod_{\frakP\in\Sg_p}\frakP^{e_\frakP}.\] 
In particular, $\lam$ is unramified at primes in $\Sgbar_p$. 

For each place $v$ of $F$, we introduce some special \BS functions $\Phi_{\lam,v}\in \cS(F_v^2)$ as follows: 
\begin{enumerate}
\item If $v\mid \infty$, then $\Phi_{\lam,v}(x,y)=2^{-k}(x+\sqrt{-1}y)^ke^{-\pi(x^2+y^2)}$. 
\item If $v\ndivides p\frakC\ol{\frakC}$ is a finite place, let $\Phi_{\lam,v}(x,y)=\bbI_{\OFv}(x)\bbI_{\frakd_{F_v}^{-1}}(y)$.
\item For $v\mid \Cinert$ non-split in $K$, let
\[\Phi_{\lam,v}(x,y):=\bbI_{\cOKv^\x}(xe_{1,v}+ye_{2,v})\lam(xe_{1,v}+ye_{2,v})^{-1}\] 
\item If $v\divides p\frakC^+$ with $v=w\wbar$ and $w\divides \Sg_p\Csplit$, then 
\[\Phi_{\lam,v}(x,y)=\varphi_{\wbar}(x)\wh\varphi_w(y), \]
where $\varphi_w(y)=\bbI_{\OFv^\x}(y)\lam_w(y)$ and 
\[\varphi_{\wbar}(x)=\begin{cases}\bbI_{\OFv}(x)&\text{ if }\wbar\ndivides \Csplit_c,\\
\bbI_{\OFv^\x}(x)\lam_{\wbar}^{-1}(x)&\text{ if }\wbar\divides \Csplit_c.\end{cases}\]
\end{enumerate}
Let $\lam_+:=\lam|_{\A_F^\x}$ be a Hecke character of  $F^\x$. Define the function $f_{\Phi_{\lam,v},s}:\GL_2(\Fv)\to\C$ by 
\beq\label{E:Godement}f_{\Phi_{\lam,v},s}(g)=\abs{\det g}_v^s\int_{\Fv^\x}\Phi_{\lam,v}((0,t_v)g)\lam_+(t_v)\abs{t_v}_{v}^{2s}\rmd^\x t_v.\eeq
Let $f_{\Phi_{\lam},s}:=\ot_v f_{\Phi_{\lam,v},s} \in I(\mathbf{1},\lambda^{-1}_{+};s-\frac{1}{2})$ and define the Eisenstein series \[E_{\lam,s}:\GL_2(F)\bksl \GL_2(\A_F)\to\C,\quad E_{\lam,s}(g):=E_\A(g,f_{\Phi_\lam,s}).\]
The analytic properties and the functional equation of this Eisenstein series have been studied in \cite[\S 19]{Jacquet72LNM278}. By the choice of our sections, we have
\beq\label{E:5.2}\begin{aligned}E_{\lam,s}(\al g u_\infty z)=E_{\lam,s}(g)J(u_\infty,\sqrt{-1})^{-k\Sg}\lam_+^{-1}(z),\\
\al\in\GL_2(F),\,u_\infty\in \SO_2(F\ot\R),\,z\in \A_F^\x.\end{aligned}\eeq
Here $J$ is the automorphy factor in \eqref{E:J.2}.

\subsection{The evaluation of Eisenstein series at CM points}
Let $N$ be a sufficiently large power of $\rmN_{K/\Q}(\frakC)$ so that the section $f_{\Phi_\lam,s}$ is invariant by $U_N(p^n)$ for some $n\gg 0$. Define the classical Eisenstein series $\bfE_\lam:\frakH_F\times\GL_2(\wh F)\to\C$ by 
\begin{align*}\bfE_\lam(\tau,g_{\bfh})=&
E_{\lam,s}\left(\pMX{y}{x}{0}{1},g_{\bfh}\right)\bigg\vert_{s=0}\quad (\tau=x+\sqrt{-1}y).\end{align*}
Let $\frakc$ be an integral ideal of $F$ with $(\frakc,pN)=1$ and $\bfc\in \wh F^{(pN)\x}$ with $\bfc\OF=\frakc$.
From \eqref{E:5.2}, we have  \beq\label{E:ESc}\bfE_{\lam,\frakc}:=\bfE_\lam|_{\frakH_F\times \GL_2(\wh F)^{(\bfc)}}\in M_k(\frakc,N;\C).\eeq
We shall deduce an explicit Damerell's formula for $\bfE_{\lam,\frakc}$, which relates the CM values of Eisenstein series to Hecke $L$-values for CM fields. We begin with the period integral of adelic Eisenstein series. Let \[\cmpt_\infty=(\pMX{\Im\sg(\skewhf)}{0}{0}{1})_{\sg\in\Sg}\in\GL_2(F\ot_\Q\R);\quad \cmpt:=\cmpt_\infty\times\cmpt_{\rmf}\in\GL_2(\A_\cF).\]

Define the toric period integral of Eisenstein series by  
\[\ell_{\cK^\x}(E_{\lam,s}):=\int_{K^\x\A_F^\x\bksl \A_K^\x}E_{\lam,s}(\iota(z)\cmpt) \lambda(z) \rmd^\x \ol{z},\]
where $\rmd^\x \ol{z}$ is the quotient measure $\rmd^\x z/\rmd^\x t$.
\begin{prop}\label{P:period1.2}
\[\ell_{\cK^\x}(E_{\lam,s})=\abs{\det(\cmpt)}_{\A}^s \frac{(\sqrt{-1})^{k\Sg}2^s}{\sqrt{\Delta_\cK}}L(s,\lam)\prod_{\frakP\in\Sg_p}(1-\lam(\frakP^{-1})\rmN\frakP^{s-1})\lam_{\frakP}(2\vartheta)\prod_{\frakL\divides\Csplit}\frac{\lam_{\frakL}(2\skewhf)}{\varepsilon(s,\lam_\frakL)}.\]
\end{prop}
\begin{proof}By \cite[page 1530]{Hsieh12AJM}
\[\ell_{\cK^\x}(E_{\lam,s})=\prod_v \ell_{\cK_v^\x}(f_{\Phi_{\lam,v},s},\lam_v),\]
where $\ell_{\cK_v^\x}(f_{\Phi_{\lam,v},s},\lam_v)$ is the local zeta integral
\begin{align*}\ell_{\cK_v^\x}(f_{\Phi_{\lam,v},s},\lam_v)=&\int_{K_v^\x/F_v^\x}f_{\Phi_{\lam,v},s}(\iota(z)\cmpt_v)\lam_v(z)\rmd^\x \ol{z}\\
=&\abs{\det(\cmpt_v)}_v^s  \int_{K_v^\x}\Phi_{\lam,v}((0,1)\iota(z)\cmpt_v)\lam_v(z)\abs{z\ol{z}}_v^s\rmd^\x z.\end{align*}
In what follows, we fix a place $v$ and write $z$ for $z_v$.  
For $v\divides \infty$, 
\begin{align*}
\ell_{\cK_v^\x}(f_{\Phi_{\lam,v},s},\lam_v)
=& \abs{\det(\cmpt_v)}_v^s 2\pi^{-1}(2^{-1}\sqrt{-1})^k \int_0^{2\pi}\int_{0}^\infty r^{2(k+s)}e^{-\pi r^2}\rmd^\x r \rmd\theta\\
=& \abs{\det(\cmpt_v)}_v^s \cdot (\sqrt{-1})^k\cdot 2^s\Gamma_\C(k+s). 
\end{align*}
If $v\ndivides p\frakC\ol{\frakC}$, it follows from  \eqref{propertyqv} and \eqref{zhatlattice} that $\Phi_{\lam,v}((0,1)\iota(z)\cmpt_v)=\bbI_{\cOKv}(z)$, so we have
\begin{align*}
\ell_{\cK_v^\x}(f_{\Phi_{\lam,v},s},\lam_v)=&  \abs{\det(\cmpt_v)}_v^s \int_{K_v^\x}\bbI_{\cOKv}(z)\lam_v(z)\abs{z\ol{z}}_v^s\rmd^\x z \\
=& \abs{\det(\cmpt_v)}_v^s L(s,\lam_v)\abs{\frakd_K}_{K_v}^\onehalf.\end{align*}
If $v\divides\frakC^-$, then 
\[\ell_{\cK_v^\x}(f_{\Phi_{\lam,v},s},\lam_v)= \abs{\det(\cmpt_v)}_v^s  \int_{K_v^\x}\bbI_{\cOKv^\x}(z)\abs{z\ol{z}}_v^s\rmd^\x z= \abs{\det(\cmpt_v)}_v^s \abs{\frakd_K}_{K_v}^\onehalf\]
If $v\divides p\frakC^+$ with $v=w\wbar$ and $w\divides \Csplit\Sg_p$, then \begin{align*}\ell_{\cK_v^\x}(f_{\Phi_{\lam,v},s},\lam_v)
=& \abs{\det(\cmpt_v)}_v^s \lam_w(-2\skewhf)\cdot Z(\wh\varphi_w,\lam_w,s)Z(\vp_{\wbar},\lam_{\wbar},s)\\
=& \abs{\det(\cmpt_v) }_v^s \lam_w(2\skewhf)\frac{L(s,\lam_w)}{L(1-s,\lam_w^{-1})\varepsilon(s,\lam_w)}\cdot L(s,\lam_{\wbar})\cdot  \abs{\det(\cmpt_v) }_v^s\abs{\frakd_K}_{K_v}^{\onehalf}. \end{align*}
In the last equality, we have used the local functional equation of the Tate integrals.  Now the proposition follows from the above calculations of local zeta integrals.
\end{proof}

\medskip

Let $\Cl^-_K:=\Cl(K)/\Cl(F)$ and let $h^-:=\#(\Cl^-_K)$. We fix a set $\stt{\frakA_1,\dots,\frakA_{h^-}}$ of representatives for ideals $\frakA_i$ prime to $p\frakC\ol{\frakC} \frakd_{K/F}$. Let \[Q=[\cOK^\x:W\cOF^\x],\]
where $W$ is the torsion subgroup of $K^\x$. Note that $Q=1$ or $2$. 
\begin{cor}[Damerell's formula]\label{C:42}Let $\stt{x(\frakA_1),\dots x(\frakA_{h^-})}$ be the CM points defined in \defref{D:CMpts}. Then have
\begin{align*}&\sum_{[\frakA_i]\in\Cl^-_K}\bfE_{\lam,\frakc(\frakA_i)}(x(\frakA_i))\lam(\frakA_i)\\
= &\frac{[\cOK^\x:\cOF^\x](\sqrt{-1})^{k\Sg}}{2^{d-1}Q\sqrt{\Delta_F}}L(0,\lam)\prod_{\frakP\in\Sg_p}(1-\lam(\frakP^{-1})\rmN\frakP^{-1})\lam_{\frakP}(2\vartheta)\cdot \prod_{\frakL\divides\Csplit}\frac{\lam_{\frakL}(2\skewhf)}{\varepsilon(0,\lam_\frakL)}.\end{align*}
\end{cor}
\begin{proof} By definition, $x(\frakA_i)=[(\skewhf_\Sg,\iota(a_i)\cmpt_\bfh)]$, where $a_i\in\wh K^{(p\frakC)\x}$ is an idele with $a_i\cOK=\frakA_i$. Let $U_K:=(K\ot\R)^\x\wh\cOK^\x$. There is an canonical isomorphism
\[ K^\x\A_F^\x\bksl \A_K^\x/U_K\iso \Cl^-_K,\quad a\mapsto a\cOK.\]
We have
\begin{align*}\sum_{[\frakA_i]\in\Cl^-_K}\bfE_{\lam,\frakc(\frakA_i)}(x(\frakA_i))\lam(\frakA_i)=&\sum_{[a_i]\in K^\x\A_F^\x\bksl \A_K^\x/U_K}E_{\lam,s}(\iota(a_i)\cmpt)\lam(a_i)|_{s=0}\\
=&\ell_{K^\x}(E_{\lam,s})|_{s=0}\cdot \frac{\#(\Cl^-_K)}{\vol(K^\x\A_F^\x\bksl \A_K^\x,\rmd^\x \ol{z})}.\end{align*}
We have $\vol(K^\x\A_F^\x\bksl \A_K^\x,\rmd^\x \ol{z})=2L(1,\tau_{K/F})$.
 From the classical formulae \beq\label{E:L1}\#\Cl^-_K=\frac{2h_{K/F}}{Q};\quad
L(1,\tau_{K/F})\sqrt{\frac{\Delta_K}{\Delta_F}}=\frac{2^{d-1}h_{K/F}}{[\cOK^\x:\cOF^\x]},\eeq
where $h_{K/F}=h_K/h_F$ is the relative class number, we obtain
\beq\label{E:7.2}\frac{\#(\Cl^-_K)}{\vol(K^\x\A_F^\x\bksl \A_K^\x,\rmd^\x \ol{z})}\frac{1}{\sqrt{\Delta_K}}=\frac{\#(W)}{2^d\sqrt{\Delta_F}}=\frac{2}{Q}\cdot \frac{[\cOK^\x:\cOF^\x]}{2^d\sqrt{\Delta_F}}.\eeq
The assertion follows from \propref{P:period1.2} and \eqref{E:7.2}.
\end{proof}
\subsection{The evaluation of $p$-adic Eisenstein series at CM points}
We next consider the CM values of $p$-adic Eisenstein series constructed from $\bfE_{\lam,\bfc}$ in \eqref{E:ESc}.
\begin{prop}\label{P:FC1}The $q$-expansion of $\bfE_{\lam,\frakc}$ at the cusp $(\OF,\frakc^{-1})$ is given by 
\[\bfE_{\lam,\frakc}|_{(\OF,\frakc^{-1})}(q)=c_0(\bfE_\lam,\frakc)+\sum_{\beta\in (N^{-1}\frakc^{-1})_+}c_\beta(\bfE_\lam,\frakc)q^\beta,\]
where $c_0(\bfE_\lam,\frakc)=f_{\Phi_\lam,s}(1)|_{s=0}$ and 
\[c_\beta(\bfE_\lam,\frakc)=(\rmN\beta)^{k-1}\prod_{v\in\bfh}W_\beta\left(f_{\Phi_{\lam,v},s},\pDII{\bfc_v}{1}\right)\bigg\vert_{s=0}.\]
\end{prop}
\begin{proof}The is a standard argument. Let $f=f_{\Phi_{\lam},s}$. By \eqref{E:FC3}, $\bfE_\lam(\tau,\pDII{\bfc}{1})$ is given by the series
\begin{align*}
f\left(\pDII{\bfc}{1}\right)
|_{s=0}+Mf\left(\pDII{\bfc}{1}\right)|_{s=0}+\sum_{\beta\in F^\x}W_\beta\left(\pMX{y}{x}{0}{1}\pDII{\bfc}{1},f\right)\Big\vert_{s=0},
\end{align*}
where $\tau=x+\sqrt{-1}y$. Since $E_\lam$ is right invariant by $U_N(p^n)$, we find that $W_\beta\left(\pDII{\bfc}{1},f\right)=0$ if $\beta\not\in N^{-1}\frakc^{-1}$ by the equation $W_\beta(\pMX{1}{x}{0}{1}g,f)=\addchar_F(\beta x)W_\beta(g,f)$. By the factorization of Whittaker functions, we obtain
\[W_\beta\left(\pMX{y}{x}{0}{1}\pDII{\bfc}{1},f\right)|_{s=0}=c_\beta(\bfE_\lam,\frakc)(\rmN\beta)^{1-k}e^{2\pi\sqrt{-1}\Tr(\beta x)}\prod_{\sg\in\Sg}W_\beta\left(\pDII{y_\sg}{1},f_{\Phi_{\lam,\sg},s}\right)|_{s=0}.\]
For $\sg\in\Sg$, an elementary calculation shows that 
\begin{align*}W_\beta\left(\pDII{y_\sg}{1},f_{\Phi_{\lam,\sg},s}\right)|_{s=0}&
=\Gamma(k)(-2\pi\sqrt{-1})^{-k}y_\sg^{1-k}\int_{\R}\frac{e^{-2\pi\sqrt{-1}\sg(\beta) y_\sg z}}{(z+\sqrt{-1})^k}\rmd z\\
&=\sg(\beta)^{k-1}e^{-2\pi\sg(\beta) y_\sg}\bbI_{\R_+}(\sg(\beta))
\end{align*}
by Cauchy's integration formula. We thus obtain \[W_\beta\left(\pMX{y}{x}{0}{1}\pDII{\bfc}{1},f\right)|_{s=0}=c_\beta(\bfE_\lam,\frakc)e^{2\pi\sqrt{-1}\Tr(\beta\tau)}\bbI_{F_+}(\beta).\]
On the other hand, the constant term is given by \[c_0(\bfE_\lam,\frakc)=f_{\Phi_{\lam},s}\left(\pDII{\bfc}{1}\right)|_{s=0}+Mf_{\Phi_{\lam},s}\left(\pDII{\bfc}{1}\right)|_{s=0}.\]
If $v$ is infinite, then $f_{\Phi_{\lam,v},s}|_{s=0}$ belongs to the irreducible sub-representation in the induced representation $I(\bfone,\sgn^k;\frac{k-1}{2})\ot\abs{\det}^{-\frac{k}{2}}$, and hence $Mf_{\Phi_{\lam,v},s}|_{s=0}=0$. We thus obtain \[c_0(\bfE_\lam,\frakc)=f_{\Phi_{\lam},s}\left(\pDII{\bfc}{1}\right)|_{s=0}=f_{\Phi_{\lam},s}(1)|_{s=0}.\]
The assertion follows.
\end{proof}

Define the fudge factor
\beq\label{E:Fdg}{\rm Fdg}(\lam):=\prod_{v\divides \Csplit\ol{\Csplit} }\frac{
\abs{\frakd_F}_v^\onehalf}{L(1,\lam_{+,v}^{-1})}\prod_{v\divides\frakC^-}\frac{\lam_v(-d_{F_v}^{-1})\abs{\frakd_F}_v^\onehalf \varepsilon(0,\lam_{+,v})}{L(0,\lam_{+,v})}\begin{cases}1&\text{ if }\Csplit_c= \cOK,\\
0&\text{ if }\Csplit_c\neq \cOK.\end{cases}\eeq
By definition, ${\rm Fdg}(\lam)=1$ if $\frakC=(1)$ is the unit ideal.
\begin{lm}\label{L:constant}We have 
\[c_0(\bfE_\lam,\frakc)={\rm Fdg}(\lam)\cdot \frac{1}{2^d}L_{\rm fin}(1,\lam_+^{-1})\prod_{\frakp\divides p}(1-\lam_+(\frakp^{-1})\rmN\frakp^{-1}).\]\end{lm}
\begin{proof}If $\Csplit_c\neq \cOK$, then $\Phi_{\lam,v}(0,t)=0$ and hence $f_{\Phi_{\lam,v},s}(1)=0$ for $v\divides \Csplit\ol{\Csplit}$. Now we suppose that $\Csplit_c=\cOK$. A direct calculation shows that
\[
f_{\Phi_{\lam,v},s}(1)=\begin{cases}2^{-k}\sqrt{-1}^k\Gamma(s+k)\pi^{-(s+k)}=2^{-1}2^{s}(\sqrt{-1})^k\Gamma_\C(s+k)&\text{ if }v\divides\infty,\\
L(2s,\lam_+)\abs{\frakd_F}_v^\onehalf&\text{ if }v\ndivides p\frakC\ol{\frakC},\\
\lam_v^{-1}(d_{F_v})\abs{\frakd_F}_v^\onehalf&\text{ if }v\divides\frakC^-,\\
\lam_{+,v}(-1)\gamma(2s,\lam_{+,v})^{-1}\abs{\frakd_F}_v^\onehalf&\text{ if }v\divides p\Csplit\ol{\Csplit}.\end{cases}\]
Here $\gamma(2s,\lam_{+,v})=\frac{L(1-2s,\lam_{+,v}^{-1})}{L(2s,\lam_{+,v})}\varepsilon(2s,\lam_{+,v})$ is the $\gamma$-factor of $\lam_{+,v}$. We have
\[L(s,\lam_+)=\Gamma_\R(s+k+\ep)^{[F:\Q]}\cdot L_{\rm fin}(s,\lam_+),\]
where $\ep=0$ if $k$ is even and $\ep=1$ if $k$ is odd. By the functional equation of $L(s,\lam_+)$ and $\varepsilon(0,\lam_{+,v})=\abs{\frakd_F}_v^{-\onehalf}$ for $v\ndivides p\frakC\ol{\frakC}$, we find that $c_0(\bfE_\lam,\frakc)=f_{\Phi_{\lam},s}(1)|_{s=0}$ equals 
\begin{align*}&\left(\prod_v f_{\Phi_{\lam,v},s}(1)\right)\big\vert_{s=0}\\
=&\left(\frac{\sqrt{-1}^{k+\ep}\Gamma_\C(k)}{2\Gamma_\R(k+\ep)}\right)^d\prod_{v\divides p\frakC^+ }\frac{
\lam_{+,v}(-1)\abs{\frakd_F}_v^\onehalf}{L(1,\lam_{+,v}^{-1})}\prod_{v\divides\frakC^-}\frac{\lam_v(d_{F_v}^{-1})\abs{\frakd_F}_v^\onehalf \varepsilon(0,\lam_{+,v})}{L(0,\lam_{+,v})}\cdot \frac{L(2s,\lam_+)|_{s=0}}{\varepsilon(0,\lam_+)}\\
=&\left(\frac{\sqrt{-1}^{k+\ep}\Gamma_\C(k)\Gamma_\R(1-k+\ep)}{2\Gamma_\R(k+\ep)}\right)^d\prod_{v\divides\infty}\lam_v(-1){\rm Fdg}(\lam)\cdot L_{\rm fin}(1,\lam_+^{-1})\prod_{\frakp\divides p}(1-\lam_+^{-1}(\frakp)\rmN\frakp^{-1})
\end{align*}
Now by reflection and duplicate formula for Gamma functions, we see that 
\[\frac{\sqrt{-1}^{k+\ep}\Gamma_\C(k)\Gamma_\R(1-k+\ep)}{\Gamma_\R(k+\ep)}=(-1)^k.\]
Now the lemma follows.
\end{proof}
We write $\lam_+=\chi_+\Abs^{k}$, where $\chi_+$ is a ray class character of $F$ modulo $\frakC_\lam\cap F$. Let $F_{\chi_+}$ be the finite extension of $F$ cut out by $\chi_+$. We say $\lam_+$ is of type $W$ if $F_{\chi_+}$ is contained in the cyclotomic $\Zp$-extension $F_\infty$ of $F$. Take any prime $\frakq_0\divides p\frakC$ of $\OF$ such that its Frobenious $\Frob_{\frakq_0}$ generates $\Gal(F_\infty/F)$. Let $H_{\frakq_0}(\lam)$ be defined as $1-\lam_+(\frakq_0)=1-\chi_+(\frakq_0)\rmN\frakq_0^{-k}\neq 0$ if $\frakC=(1)$ and $\lam_+$ is of type $W$ and $\rmH_{\frakq_0}(\lam)=1$ otherwise. Note that $H_{\frakq_0}(\lam)$ is a non-zero algberiac number.
\begin{cor}We have
\[H_{\frakq_0}(\lam)\cdot \bfE_{\lam,\frakc}|_{(\OF,\frakc^{-1})}\in \Zbar_{(p)}\powerseries{(N^{-1}\frakc^{-1})_+}.\]
In particular, $\bfE_{\lam,\frakc}$ descends to a geometric modular form $\bbE_{\lam,\frakc}$ in $\cM_k(\frakc,N;\Zbar_{(p)})\ot_{\Z}\Q$.
\end{cor}
\begin{proof}
For $\beta\in F^\x$, we put 
\[\bfa_\beta^\setp(\bfE_\lam,\frakc):= \prod_{v\ndivide p\infty}W_\beta(f_{\Phi_{\lam,v},s},\pMX{\bfc_v}{}{}{1})\bigg\vert_{s=0}.\]
For each place $v\ndivides p$, we see that $\abs{\frakd_F}_{v}^{-\onehalf}\cdot f_{\Phi_{\lam,v},s}$ is the function $\phi_{\lam,s,v}$ defined in \cite[\S 4.5]{Hsieh14Crelle}. The explicit formula of $\bfa_\beta^\setp(\bfE_\lam,\frakc)$ can be easily deduced from \cite[(4.9)]{Hsieh14Crelle}, which in particular shows that $\bfa^{(p)}_\beta(\bfE_\lam,\frakc)$ belongs to $\frako_L$ for some number field $L$ unramified outside $p$. On the other hand,  for each $v\divides p$, 
\[W_\beta(f_{\Phi_{\lam,v},s},1)=\lam_{+,v}(\beta)\abs{\beta}_{\Fv}^{-1}\bbI_{\cO_{\Fv}}(\beta).\]
Let $\wh\lam_+$ be the $p$-adic avatar of $\lam_+$. So $\wh\lam_+((\beta)_p)=\lam_+((\beta)_p)\beta^{k\Sg}$ for $\beta\in F^\x$. Therefore, from \propref{P:FC1}, we deduce that 
\beq\label{E:FC2}c_\beta(\bfE_\lam,\frakc)=\bfa^\setp(\bfE_\lam,\frakc)\wh\lam_+((\beta)_p)\cdot\rmN_{F/\Q}(\beta)^{-1}\abs{\rmN_{F/\Q}(\beta)}_{\Q_p}^{-1}\bbI_{\cOF\ot\Zp}(\beta)\in\Zbar_{(p)}.\eeq

By \lmref{L:constant}, we see that $c_0(\bfE_\lam,\frakc)=f_{\Phi_{\lam,s}}(1)|_{s=0}$ is give by a product of the value $L_p(s,\chi_+^{-1})$ of the Deligne-Ribet $p$-adic $L$-function at $s=1-k$ and the Fudge factor ${\rm Fdg}(\lam)\in\Zbar_{
(p)}$. It is well-known that  $L_p(1-k,\chi_+^{-1})\in\Zbar_{(p)}$ if $\chi_+$ is not of type $W$ and in general $(1-\lam_+(\frakq))\cdot L_p(1-k,\chi_+^{-1})\in\Zbar_{(p)}$ for any $\frakq$ away from the conductor of $\chi_+$. It follows that the product ${\rm Fdg}(\lam) L_p(1-k,\chi_+^{-1})\in \Zbar_{(p)}$ unless $\frakC=(1)$ and $\chi_+$ is of type $W$, and hence $H_{\frakq_0}(\lam)\cdot c_0(\bfE_\lam,\frakc)\in \Zbar_{(p)}$. 
\end{proof}
We let \[\cE_{\lam,\frakc}:=\wh\bbE_{\lam,\frakc}\in V(\frakc,N;\cW)\ot_{\Zp}\Qp\] be the $p$-adic avatar of $\bfE_{\lam,\frakc}$ defined in \eqref{E:padicavatar}. Let $(\Omega_p,\Omega_\infty)\in (\cOF\ot \cW)^\x\times (F\ot\C)^\x$ be the canonical $p$-adic and complex period associated with the polarized CM abelian variety $(F\ot\C)/i_\Sg(\cOK)$ \cite[(4.4 a,b)]{HidaTil94Inv}. It will be convenient to use the modified complex period defined by \beq\label{E:period}\Omega:=\Omega_\infty(1\ot \pi)^{-1}(-\Im\skewhf_\Sg)\in (F\ot\C)^\x.\eeq
 \begin{prop}\label{P:period2.2}We have\begin{align*}\frac{1}{\Omega_p^{k\Sg}}\sum_{[\frakA_i]\in \Cl^-_K}\lam(\frakA_i)\cE_{\lam,\frakc(\frakA_i)}(x(\frakA_i))=&\frac{2\wh\lam((2\skewhf)_{\Sg_p})}{Q}\prod_{\frakL\divides\Csplit}\frac{\lam_\frakL(2\skewhf)}{\varepsilon(0,\lam_\frakL)}\cdot \frac{[\cOK^\x:\cOF^\x]}{2^d\sqrt{\Delta_F}} \\
&\times \frac{1}{(\sqrt{-1})^{k\Sg}}\cdot \frac{L(0,\lam)}{\Omega^{k\Sg}}\prod_{\frakP\in\Sg_p}(1-\lam(\frakP^{-1})\rmN\frakP^{-1}).\end{align*}
\end{prop}
\begin{proof}By \cite[Theorem 1.2]{HidaTil93ASENS}, we have
\[\frac{\cE_{\lam,\frakc(\frakA_i)}(x(\frakA_i))}{\Omega_p^{k\Sg}}=\frac{(2\pi\sqrt{-1})^{k\Sg}}{\Omega_\infty^{k\Sg}}\bbE_{\lam,\frakc(\frakA_i)}(x(\frakA_i))=
\frac{(2\pi\sqrt{-1})^{k\Sg}}{\Omega_\infty^{k\Sg}}\bfE_\lam(x(\frakA_i)),
\]
so we have
\begin{align*}\frac{1}{\Omega_p^{k\Sg}}\sum_{[\frakA_i]\in \Cl^-_K}\lam(\frakA_i)\cE_{\lam}(x(\frakA_i))
=&\frac{(2\pi\sqrt{-1})^{k\Sg}}{\Omega_\infty^{k\Sg}}\sum_{[\frakA_i]\in \Cl^-_K}\lam(\frakA_i)\bfE_{\lam}(x(\frakA_i)).
\end{align*}
Now the assertion follows immediately from \corref{C:42} together with the equation $\wh\lam((2\skewhf)_{\Sg_p})=\lam((2\skewhf)_{\Sg_p})(2\skewhf)^{k\Sg}$ and the definition of the modified period $\Omega$ in \eqref{E:period}. 
\end{proof}

\subsection{The improved Katz $p$-adic $L$-functions for CM fields}
We let $G_\infty(\frakC)$ denote the Galois group $\Gal(K(\frakC p^\infty)/K)$ of the ray class field $K(\frakC p^\infty)/K$ modulo $\frakC p^\infty$. Recall that Katz ($\frakC=(1)$) and Hida-Tilouine ($\frakC\neq (1)$) constructed an element $\cL_\Sg\in 
\cW\powerseries{G_\infty(\frakC)}$, which is uniquely characterized by the following interpolation formula: \beq\label{E:Katz78}\begin{aligned}\frac{\wh\lam(\cL_\Sg)}{\Omega_p^{k\Sg+2j}}=&\frac{[\cOK^\x:\cOF^\x]}{2^d\sqrt{\Delta_F}}\cdot \frac{(-1)^{k\Sg}(2\pi)^{k+2j}}{\Im(2\vartheta)^j\Omega_\infty^{k\Sg+2j}}\cdot \lam((2\vartheta)_{\Sg_p})\\
&\qquad\times{\rm E}_p(\lam) L(0,\lam)\prod_{\frakl\divides\frakC}(1-\lam(\frakl)),\end{aligned}\eeq
where 
\begin{itemize}
\item  $\wh\lam:G_\infty(\frakC)\to \cW^\x$ is a locally algebraic character of weight $k\Sg+j(1-c)$ \emph{critical} at $0$, namely either $k>0$ and $j\in\Z_{\geq 0}[\Sg]$ or $k\leq 1$ and $k\Sg+j\in\Z_{>0}[\Sg]$,
\item ${\rm E}_p(\lam)$ is the modified Euler factor defined by \[{\rm E}_p(\lam)=\prod_{\frakP\in\Sg_p}\frac{(1-\lam(\Pbar))(1-\lam(\frakP^{-1})\rmN \frakP^{-1})}{\varepsilon(0,\lam,\addchar_{K_\frakP})},\]
\item $L(s,\lam)$ is the complete Hecke $L$-function associated with $\lam$ defined by 
\[L(s,\lam):=\Gamma_\C(s+k\Sg+j)L_{\rm fin}(s,\lam)\]
\end{itemize}
(\cf \cite[Theorem II]{HidaTil93ASENS}). Put
\[\sg_{2\skewhf}=\rec_K((2\skewhf)_{\Sg_p})\in G_\infty(\frakC).\] Let $\chi$ be a ray class character of $K$ of conductor $\frakC$. We let \[\pi_\chi: \cW\powerseries{G_\infty(\frakC)}\to \cW\powerseries{G_\infty(1)}\]
be the homomorphism defined by $\gamma\mapsto \chi(\gamma)[\gamma|_{K(p^\infty)}]$ and define the $p$-adic $L$-function $\cL_\Sg(\chi)\in\cW\powerseries{G_\infty(1)}$ associated with the branch character $\chi$ by 
\[\cL_\Sg(\chi):=\pi_\chi(\cL_\Sg\cdot [\sg_{2\skewhf}]^{-1})\in \cW\powerseries{G_\infty(1)}.\]
The interpolation formula \eqref{E:Katz78} can be rephrased as
\beq\label{E:intep2}\frac{\wh\nu(\cL_\Sg(\chi))}{\Omega_p^{k\Sg+2j}}=\frac{[\cOK^\x:\cOF^\x]}{2^d\sqrt{\Delta_F}}\cdot \frac{1}{(\sqrt{-1})^{k\Sg+j}}\cdot {\rm E}_p(\chi\nu)\cdot \frac{L(0,\chi\nu)}{\Omega^{k\Sg+2j}},\eeq
where $\wh\nu:G_\infty(1)\to\cW^\x$ is a locally algebraic character of weight $k\Sg+j(1-c)$ critical at $0$ and $\Omega$ is the modified period in \eqref{E:period}. This form is in accordance with the conjectural shape of the $p$-adic $L$-functions for motives \cite[Conjecture A]{Coates89Bourbaki}. Let $\Gamma:=\Gal(K_{\Sg_p}/K)$. Put
\[\frakX^{\rm alg}:=\stt{\wh\nu:\Gamma\to\cW^\x\mid \wh\nu\text{ is locally algebraic of weight }k\Sg,\,k\geq 2}.\]
\begin{prop}\label{P:Emeasure}There exists a unique $p$-adic measure $\rmd\mu^{\rm Eis}_{\chi,\frakc}$ on $\Gamma$ valued in $V(N,\frakc;\cW)$ such that 
\[\int_{\Gamma}\wh\nu(\sg)\rmd\mu^{\rm Eis}_{\chi,\frakc}(\sg)=H_{\frakq_0}(\chi\nu)\cdot \cE_{\chi\nu,\frakc}
\]
for all $\wh\nu\in \frakX^{\rm alg}$.
\end{prop} 
\begin{proof}
This is shown by a standard argument. First we note that since $\chi$ is assumed to be unramified outside $p$, $H_{\frakq_0}(\chi\nu)=1-\chi(\frakq_0)\wh\nu(\frakq_0)$ if $\chi_+=1$ and $1$ otherwise. The interpolation of the constant term $\wh\nu\mapsto H_{\frakq_0}(\chi\nu)\cdot c_0(\bfE_{\chi\nu},\frakc)$ is a direct consequence of the existence of Deligne-Ribet $p$-adic $L$-function. For $\beta\in F^\x$, it is shown in \cite[(4.13)]{Hsieh14Crelle} that 
\begin{align*}H_{\frakq_0}(\lam)\cdot c_\beta(\bfE_\lam,\frakc)&=\sum_j b_j\cdot \wh\lam(a_j)\text{ for }b_j\in \cW,\\
&a_j\in \wh F^{(p\frakF)\x}\prod_{v\divides p\Csplit\ol{\Csplit}}\OF_v\prod_{v\divides\frakC^-}K_v^\x\quad (\lam=\chi\nu).\end{align*}
Here $a_j$ and $b_j$ depend on $\frakc$ and $\beta$ but do not rely on $\lam$. Therefore, there exists a $\cW$-valued measure on $\Gamma$ interpolating $\beta$-th Fourier coefficient $H_{\frakq_0}(\chi\nu)\cdot c_\beta(\bbE_{\chi\nu},\frakc)$. Since $\cE_{\lam,\frakc}=\wh\bbE_{\lam,\frakc}$ share the same $q$-expansion with $\bfE_{\lam,\frakc}$ at $(\OF,\frakc^{-1})$, we thus obtain a $p$-adic measure $\rmd\mu^{\rm Eis}_{\chi,\frakc}(q)$ on $\Gamma$ valued in the ring $\cW\powerseries{(N^{-1}\frakc^{-1})_+}$ such that the value of $\rmd\mu^{\rm Eis}_{\chi,\frakc}(q)$ at $\wh\nu\in\frakX^{\rm alg}$ is the $q$-expansion of $H_{\frakq_0}(\chi\nu)\cdot\cE_{\chi\nu,\frakc}$. Since the set $\frakX^{\rm alg}$ is dense in the space of continuous functions on $\Gamma$, we deduce that the masure $\rmd\mu^{\rm Eis}_{\chi,\frakc}(q)$ descends to a unique measure $\rmd\mu^{\rm Eis}_{\chi,\frakc}$ valued in $V(\frakc,N;\cW)$ such that the $q$-expansion of $\rmd\mu^{\rm Eis}_{\chi,\frakc}$ at the cusp $(\OF,\frakc^{-1})$ is $\rmd\mu^{\rm Eis}_{\chi,\frakc}(q)$ from the irreducibility of the Igusa tower $\Ig(N,\frakc)$ due to Ribet.
\end{proof}
 Let $\sV:G_F^{\rm ab}\to G_K^{\rm ab}$ be the transfer map. Let $\e_\Sg:\Gal(K_{\Sg_p}/K)\to\cW^\x$ be a $p$-adic character such that $\e_{\Sg}\circ\sV|_{G_K}=\cyc$. Let $\e_{\Sgbar_p}(\sg):=\e_{\Sg}(c\sg c^{-1})$. Define the two-variable $p$-adic $L$-functions.
 \[\cL_{\Sg}(s,t,\chi):=\e_{\Sg}^s\e_{\Sgbar}^t\left(\cL_{\Sg}(\chi)\right),\quad (s,t)\in\Zp^2.\]
Let $\psi$ be the idele class character of $K$ such that $\wh \psi=\e_\Sg$.
For $s\in\Zp$, put\[\cE_{\chi,\frakc}(s):=\int_{\Gamma}\e_\Sg(\sigma)^s\rmd\mu^{\rm Eis}_{\chi,\frakc}(\sigma)\cdot\begin{cases}\frac{1}{1-\e_\Sg(\frakq_0)^s}&\text{ if }\frakC=(1)\text{ and }\chi_+=\bfone,\\
1&\text{ otherwise.}\end{cases}.\]
Then \propref{P:period2.2} and \eqref{E:intep2} implies that
\beq\label{E:padicL1}\begin{aligned}&\sum_{[\frakA_i]\in\Cl^-_K}\cE_{\chi,\frakc(\frakA_i)}(s)(x(\frakA_i))\chi\e_\Sg^s(\frakA_i)\\
=&C_\chi\cdot \e_\Sg^s(\sg_{2\skewhf}\Fr_\Csplit^{-1})\cdot \frac{\cL_\Sg(s,0,\chi)}{\prod_{\frakP\in\Sg_p}(1-\chi\e_\Sg^s(\ol{\frakP}))},\end{aligned}\eeq
where 
\[C_\chi:=\frac{2\chi(\sg_{2\skewhf})}{Q}\prod_{\frakL\divides\Csplit}\frac{\chi_\frakL(2\skewhf)}{\varepsilon(0,\chi_{\frakL})}\in\Zbar_{(p)}^\x.\]
The following is a special case of $p$-adic Kronecker limit formula.
\begin{prop}\label{P:Kronecker}Suppose that \eqref{Leo1} holds. If   $\chi$ is unramified everywhere and $\chi_+=\bfone$, then we have \[\frac{2^d \cL_\Sg(s,0,\chi)}{\zeta_{F,p}(1-s)\prod_{\frakP\in\Sg_p}(1-\chi(\Pbar)\e_{\Sg}(\Pbar)^s)}\bigg\vert_{s=0}=\begin{cases} h_{K/F}&\text{ if }\chi=\bfone,\\
0&\text{ if }\chi\neq \bfone.\end{cases}
\]
\end{prop}
\begin{proof}We have $\frakC=(1)$ since $\chi$ is unramified everywhere. Put
\[\cE_{\chi,\frakc}^*(s):=\cE_{\chi,\frakc}(s)\cdot\frac{2^d}{\zeta_{F,p}(1-s)}.\]
Note that ${\rm Fdg}(\lam)=1$ if $\frakC=(1)$ and hence $\cE_{\chi,\frakc}^*(s)|_{s=0}=1$ is the constant function by a result of Colmez \cite{Colmez88Inv}. By \eqref{E:padicL1}, we have
\[\sum_{[\frakA_i]\in\Cl^-_K}\cE^*_{\chi,\frakc(\frakA_i)}(s)(x(\frakA_i))\e_\Sg^s(\frakA_i)\chi(\frakA_i)=\frac{2\e_\Sg(\sg_{2\skewhf})^s}{Q}\cdot\frac{2^d\cL_\Sg(s,0,\chi)}{\zeta_{F,p}(1-s)\prod_{\frakP\in\Sg_p}(1-\chi(\Pbar)\e_{\Sg}(\Pbar)^s)}\]
Evaluating at $s=0$ of the above equation, we find that
\[\frac{2}{Q}\cdot \frac{2^d \cL_\Sg(s,0,\chi)}{\zeta_{F,p}(1-s)\prod_{\frakP\in\Sg_p}(1-\chi(\Pbar)\e_{\Sg}(\Pbar)^s)}\bigg\vert_{s=0}=\sum_{[\frakA_i]\in \Cl^-_K}\chi(\frakA_i)=\begin{cases}\#(\Cl^-_K)&\text{ if }\chi=\bfone,\\
0&\text{ if }\chi\neq \bfone.\end{cases}\]
It follows from the formula $\#(\Cl^-_K)=2h_{K/F}/Q$.
\end{proof}

\begin{thm}\label{T:improved}Suppose that \eqref{Leo1} holds. If $\chi\neq \bfone$, then there exists a analytic $L$-function $\cL_\Sg^*(\chi)\in \Lam_K$ such that \[\cL_{\Sg}(s,0,\chi)=\e_\Sg^s(\cL_\Sg^*(\chi))\prod_{\frakP\in\Sg_p}(1-\chi\e_\Sg^s(\Pbar)).\]
We call $\cL_\Sg^*(s,\chi):=\e_\Sg^s(\cL_{\Sg}^*(\chi))$ the improved $p$-adic $L$-function.
\end{thm}
\begin{proof}For any $[\frakA]\in\Cl^-_K$, let $\rmd\mu_{\cL_\frakA}$ be the bounded $p$-adic measure on $\Gamma$ obtained by evaluating $\rmd\mu_{\cE_{\chi,\frakc(\frakA)}}$ at the CM point $x(\frakA)$. Namely, $\rmd\mu_{\cL_{\chi,\frakA}}$ is characterized by the interpolation property
\[\int_{\Gamma}\e_\Sg(\sigma)^k\rmd\mu_{\cL_{\chi,\frakA}}(\sigma)=\cE_{\chi\psi^k,\frakc(\frakA)}(x(\frakA)).\]
 Let $\cL_{\chi,\frakA}\in \Frac(\Lam_K)$ be the corresponding element in the Iwasawa algebra of $\Gamma$. Then $\cL_{\chi,\frakA}\in\Lam_K$ unless $\chi_+=1$ and $\frakC=(1)$. 
$(1-[\Fr_{\frakq_0}])\cL_{\chi,\frakA}\in\Lam_K$ in general. Define 
\beq\label{E:padicL2}\cL_\Sg^*(\chi):=\frac{[\sg_{2\skewhf}^{-1}\Fr_\Csplit]}{C_\chi}\sum_{[\frakA_i]\in \Cl^-_K}\cL_{\chi,\frakA_i}\cdot \chi(\frakA_i)[\Fr_{\frakA_i}]\in\Lam_K.\eeq
One verifies that $\cL_\Sg^*(\chi)$ satisfies the desired interpolation property by \eqref{E:padicL1} and that $\cL_\Sg^*(\chi)\in\Lam_K$ if $\chi\neq 1$ by \propref{P:Kronecker}.
\end{proof}
\section{$p$-adic Rankin-Selberg method for Hilbert modular forms}\label{S:RS}

\subsection{Hida families of Hilbert modular cusp forms}
Let $M_k(\frakn p,\xi)$ be the space of parallel weight $k\geq 1$  Hilbert modular forms 
of level $\Gamma_1(\frakn p)$ and Nebentypus $\xi$, where $\frakn$ is a prime to $p$ ideal of  $\cOF$ and $\xi$ is a finite order idele class character modulo $\frakn p$. We fix a set $\{\frakc_{i}\}_{i=1,\ldots,h_+}$ of the representatives of the narrow class group of $F$. For $\bff\in M_k(\frakn,\xi)$, we let
\[D(s,\bff)=\sum_{\fraka\text{: ideals of }\OF} C(\fraka,\bff)\rmN\fraka^{-s}\]
 be the associated Dirichlet series  and let $C_i(0,\bff)$ be the constant terms at the cusp $(\frako,\frakc_i^{-1})$ (\cite[(2.25)]{Shimura78Duke}). The value $C(\fraka,\bff)$ is called the $\fraka$-th Fourier coefficient of $\bff$. 

Let $F_\infty=F\Q_\infty$ be the cyclotomic $\Z_p$-extension of $F$ and let $\Lambda:=\cW\powerseries{ \Gal(F_\infty/F)}$ be the Iwasawa algebra over $\cW$ associated with $F_\infty/F$. We shall assume that $p$ is unramified in $F$ throughout. Let $\cyc:\Gal(F_\infty/F)\to \Zp^\x$ be the $p$-adic cyclotomic character given by the composition $\Gal(F_\infty/F)\to\Gal(\Q_\infty/\Q)\hookto \Gal(\Q(\zeta_{p^\infty})/\Q)\overset{\e_{\rm cyc}}\iso\Zp^\x$ and for any integer $k$, let $\nu_k:\Lambda \to \Z_p$ be the natural morphism induced by $\cyc^k:\Gal(F_\infty/F)\to \Zp^\x$. Let $h_
+$ be the narrow class number of $F$. A  $\Lambda$-adic modular form $\mathcal{F}$ of level $N$  and Nebentypus $\xi$ is a collection of elements \[ \{C_i(0,\calF),\,i=1,\dots,h_+;\quad C(\fraka,\mathcal{F}), \text{ $\fraka \subset \OF$ non-zero ideals}\} \subset \Lambda \] such that  for all but finitely many $k \geq 1$, there is a unique $\nu_{k}(\calF) \in M_{k+1}(\frakn p,\xi \Om_F^{k}) $ whose associated Dirichlet series is 
\[D(s,\nu_k(\calF))=\sum_{\fraka}\nu_{k}(C(\fraka,\mathcal{F}))\rmN\fraka^{-s}\] and whose constant terms are given by $C_i(0,\nu_k(\calF))=\nu_k(C_i(0,\calF))$. We denote by ${\bf M}(\frakn,\xi) $ the space of $\Lambda$-adic modular forms of level $\frakn$ and Nebentypus $\xi$ and by $\bfS(\frakn,\xi)$ the subspace of $\Lam$-adic cusp forms. Namely, $\bfS(\frakn,\xi)$ is the subspace of $\bfM(\frakn,\xi)$ consisting of $\calF\in \bfM(\frakn,\xi)$ with $\nu_k(\calF)$ is always a cusp form for all but finitely many $k\geq 1$. In particular, if $\calF\in \bfS(\frakn,\xi)$, then $C_i(0,\calF)=0$ for all $i=1,\dots,h_+$. 

We have a natural action of the Hecke operators $\{T_{\frakq}, U_{\mathfrak{l}} \}_{ \frakq \nmid \frakn p, \mathfrak{l} \mid \frakn p} $ on ${\bf S}(\frakn,\xi)$, which are compatible with the Hecke action on $S_{k+1}(\frakn p,\xi \Om_F^{k})$ under the specialization ${\bf S}(\frakn,\xi) \to S_{k+1}(\frakn p,\xi \Om_F^{k})$ along $\nu_k$. For any integral ideal $\frakb$ and $\mathcal F\in\bfS(\frakn,\xi)$, define $V_\frakb \mathcal F\in \bfS(\frakn\frakb,\xi)$ by $C(\fraka,V_\frakb\mathcal F)=C(\fraka\frakb^{-1},\mathcal F)$.  Hida defined an ordinary projector $e^{\Ord}:=\underset{r \to \infty}{\lim} (\prod_{\frakp \mid p} U_{\frakp})^{r!}$ on ${\bf S}(\frakn,\xi)$. The $\Lambda$-module ${\bf S}^{\Ord}(\frakn,\xi)=e^{\Ord}({\bf S}(\frakn,\xi))$ is the subspace of ordinary $\Lam$-adic forms of tame level $\frakn$ and Nebentypus $\xi$. For any finite flat $\Lambda$-algebra $R$, we denote by \[{\bf S}^{\Ord}(\frakn,\xi,R):={\bf S}^{\Ord}(\frakn,\xi) \otimes_{\Lambda} R\] and let ${\bf T}(\frakn,\xi,R):=R[T_{\frakl}, U_{\frakq}]_{ \frakl\nmid \frakn p, \frakq\mid \frakn p} \subset \mathrm{End}_{R}\left({\bf S}^{\Ord}(\frakn,\xi,R)\right)$ be Hida's big ordinary cuspidal Hecke algebra acting faithfully on ${\bf S}^{\Ord}(\frakn,\xi, R)$. Hida's control theorem \cite{Hida88Ann} (see also \cite{Wiles88Inv}) yields:
\begin{itemize}
\item ${\bf S}^{\Ord}(\frakn,\xi)$ is a finite free $\Lambda$-module.
\item The specialization $\nu_k: {\bf S}^{\Ord}(\frakn,\xi) \to S^{\Ord}_{k+1}(\frakn p,\xi \Om_F^{k}):=e_\Ord(S_{k+1}(\frakn p,\xi \Om_F^{k}))$ is surjective.
\item $\bfT(\frakn,\xi,\Lam)$ is finite and flat over $\Lam$. 
\end{itemize}

\subsection{Hida families of theta series}\label{CMfamiliy} Let $K$ be a totally imaginary quadratic extension of $F$ as before and let $\tau_{K/F}:C_F\to\stt{\pm 1}$ be the quadratic character associated with $K/F$. To an idele class character  $\psi: C_K \to \C^\times$ of infinity type $k\Sigma$ and conductor $\frakC_\psi$ prime to $p$, we can associate a theta series $\theta_\psi \in S_{k+1}(\mathrm{N}_{K/F}(\mathfrak{C}_\psi)\frakd_{K/F},\xi )$ with $\xi:=\mid \cdot \mid_{\mathbf{A}_F}^{-k}  \psi_+ \tau_{K/F}$. The Fourier coefficients of $\theta_\psi$ are given by
\[C(\fraka,\theta_\psi)=\sum_{\substack{\frakA:\text{ ideals of }\cOK,\\
\frakA\ol{\frakA}=\fraka,\,(\frakA,\frakC_\psi)=1}}\psi(\frakA).\]
The $p$-stabilized newform $\theta_\psi^{(\Sg_p)}$ associated with $\theta_{\psi}$ is defined by 
\[C(\fraka,\theta_\psi^{(\Sg_p)})=\sum_{\substack{\frakA:\text{ ideals of }\cOK,\\
\frakA\ol{\frakA}=\fraka,\,(\frakA,\frakC_\psi\Sg_p)=1}}\psi(\frakA).\]
In particular, if $\frakp$ is a prime above $p$, then
\[C(\frakp,\theta_\psi^{(\Sg_p)})=\psi(\Pbar)\quad \text{ if }\frakP\in\Sg_p\text{ and }\frakp\cOK=\frakP\Pbar.\]
Let $\Lambda_K=\cW \powerseries{\Gal(K_{\Sg_p}/K)}$ be the Iwasawa algebra associated with the $\Zp$-extension $K_
{\Sg_p}/K$. Then $\Lam_K$ is regarded as a $\Lam$-algebra via the transfer map $\sV:\Gal(F_\infty/F)\to \Gal(K_{\Sg_p}/K)$. Write $\sg\mapsto [\sg]$ for the inclusion of group-like elements $\Gal(K_{\Sg_p}/K)\to \Lam_K^\x$. Define the universal character $\universal: G_K \to \Lambda_K^{\times}$ by
\beq\label{E:univ}\universal(\sg)=[\sg|_{K_{\Sg_p}}]^{-1}.\eeq
By definition, $\universal$ is unramifiedd outside $\Sg_p$. For any prime ideal $\frakQ\not\in\Sg_p$, we put
\[\universal(\frakQ)=\universal(\Fr_\frakQ).\]
Let $\brch:C_K\to \cW^\x$ be a ray class character of conductor $\frakC_\brch$ coprime to $p$. We define the ordinary $\Lam_K$-adic cusp from $\bftheta_\phi\in \bfS^\Ord(\frakn,\phi_+\tau_{K/F},\Lam_K)$ with $\frakn=\rmN_{K/F}(\frakC_\brch)\frakd_{K/F}$ by 
\[C(\fraka,\bftheta_\brch)=\sum_{\substack{\frakA:\text{ ideals of }\cOK,\\
\frakA\ol{\frakA}=\fraka,\,(\frakA,\frakC_\brch\Sg_p)=1}}\brch(\frakA)\cdot \universal(\Fr_\frakA)\in \Lam_K.\]
Let $\psi$ be the Hecke character of $K$ such that $\wh\psi=\e_\Sg$. By definition,
\beq\label{E:CM3}\nu_k(\bftheta_\brch)=\e_\Sg^k(\bftheta_\brch)=\iota_p(\theta_{\brch\psi^{-k}}^{(\Sg_p)})\in  S^\Ord_{k+1}(\frakn,\phi_+\tau_{K/F}\Om_F^k).\eeq
It follows that $\bftheta_\phi$ is a $\Lam_K$-adic newform of tame conductor $\frakn$ in the sense of \cite[Definition 1.5.1]{Wiles88Inv}. For each prime $\frakl$ of $\frako_F$, we have
\[
T_\frakl \bftheta_\brch=C(\frakl,\bftheta_\brch)\bftheta_{\brch}\text{ if }(\frakl,p\frakn)=1,\quad U_\frakl \bftheta_\brch=C(\frakl,\bftheta_\brch)\bftheta_{\brch}\text{ if }\frakl\divides p\frakn,\]
where \beq\label{E:CM4}
\begin{aligned}
C(\frakl,\bftheta_\brch)&=\begin{cases}
\brch\universal(\frakL)+\tau_{K/F}(\frakl)\brch\universal(\ol\frakL)&\text{ if }(\frakl,p\frakn)=1\text{ and }\frakL\divides \frakl\OK,\\
\brch\universal(\ol\frakL)&\text{ if }\frakl\OK=\frakL\ol\frakL\text{ with }\frakl\divides p\frakn\text{ and }\frakL\divides\Sg_p\frakC_\brch,\\
0&\text{ otherwise}.
\end{cases}
\end{aligned}
\eeq

 \subsection{A $\Lam$-adic form of Eisenstein series}
Let $\zeta_{F,p}(s)$ be the Deligne--Ribet $p$-adic $L$-function satisfying the interpolation property that for all positive integer $k$, \begin{equation}\label{eq:interpolation_p-adic_L-fcn}\zeta_{F,p}(1-k)=\left(\prod_{\frakp|p} (1-\Om_F(\frakp)^{k} \rmN\frakp^{k-1})\right) L(1-k,\Om_F^{k}).
\end{equation}
Let $\boldsymbol\zeta^\vee_{F,p} \in \Frac\Lam$ be the element in Iwasawa algebra such that \[\cyc^k(\boldsymbol\zeta^\vee_{F,p})=\zeta_{F,p}(1-k). \]Let $\frakq_0\ndivides p$ be a prime of $\OF$ such that $\Fr_{\frakq_0}$ generates the Galois $\Gal(F_\infty/F)$. Put
\[H_{\frakq_0}=1-[\Frob_{\frakq_0}]\in \Lam.\]
For each integral ideal $\frakm$ of $F$ with $(\frakm,p)=1$,
we let $\cE(\frakm,1)$ be the $\Lam$-adic Eisenstein series with the $\fraka$-th Fourier coefficient given by
\begin{equation}\label{Eisensteincoefficients}
	C(\fraka,\cE(\frakm,1))=\sum_{\frakm^{-1}\fraka\subset \frakb\subset \OF,\,(p,\frakb)=1} \rmN\frakb^{-1}[\Frob_{\frakb}]^{-1}\in\Lam,
\end{equation}  
and the constant terms are given by
\[
C_{i}(0,\cE(\frakm,1))= 2^{-[F:\Q]} \boldsymbol\zeta^\vee_{F,p}
\]
for $i=1,\ldots, h_+$. This way we see that $H_{\frakq_0}\cdot \cE(\frakm,1)\in \bfM(\frakm,\Om)$. Recall that if $\frakl$ is a prime factor of $\frakm$, then $U_\frakl\in \End\bfM(\frakm,\xi)$ is defined by
\begin{align*}
C(\fraka,U_\frakl\mathcal F)=C(\frakl\fraka,\mathcal F),\quad C_i(0,U_\frakl\mathcal F)=C_i(0,\mathcal F).
\end{align*}
For any square-free divisor $\frakm_0\divides \frakm$, we define
\beq\label{E:pES}\cE(\frakm,\frakm_0)=\prod_{\frakl\divides\frakm_0}\left(1-\rmN\frakl^{-1}U_\frakl\right )\cE(\frakm,1).\eeq
In particular, 
\beq\label{E:57}C_i(0,\cE(\frakm,\frakm_0))=2^{-[F:\Q]} \boldsymbol\zeta^\vee_{F,p}\prod_{\frakl\divides\frakm_0}(1-\rmN\frakl^{-1})\text{ for }i=1,\dots,h_+.\eeq

 \subsection{The Rankin-Selberg convolution}\label{SS:RS} We call a ray class character $\chi$ \emph{anticyclotomic} if $\chi(c\sg c)=\chi(\sg^{-1})$, or equivalently $\chi\circ \sV=\bfone$. When regarded as a Hecke character, $\chi$ is anticyclotomic if $\chi|_{\A_F^\x}=\bfone$. Let $\frakf$ be an integral ideal of $\OF$. Let $\chi$ be a primitive anticyclotomic character of $K$ modulo $\frakf\OK$ with value in $\cW^\x$. Suppose that  \[(\frakf,p\frakd_{K/F})=1. \]There exists a ray class character $\brch$ unramified at $p\frakd_{K/F}$ such that \[\chi=\brch^{1-c}\] (\cf\cite[Lemma 5.31]{HidaDeepBlue}). We are allowed to take a twist by a character of the form $\xi_1\circ\rmN_{K/F}$ with a ray class character $\xi_1$ of $F$. We say a character $\eta: K_v^\x\to\C^\x$ is \emph{minimal} if the conductor of $\eta$ is minimal among $\eta\cdot \xi_1\circ \rmN_{K_v/\Fv}$, where $\xi_1$ runs over ray class characters of $F$. Now we fix an auxiliary prime $\frakl_0\ndivides p\frakf\Delta_F$ of $\OF$ which is split in $K$. With the help of  \lmref{L:54} below, we may assume $\brch$ satisfies the following minimal condition
\beqcd{min}\text{$\brch_v$ is minimal for all $v\in\bfh$ except for $v=\frakl_0$}\eeqcd  
by replacing $\brch$ by $\brch\cdot \xi_1\circ\rmN_{K/F}$ for a suitable ray class character $\xi_1$ of $F$.
    \begin{lm}\label{L:54}Let $S$ be a finite set of primes of $\OF$. Let $\xi_S$ be a finite order character of $\OF_S^\x:=\prod_{v\in S}\OFv^\x$. For any prime $\frakl\not\in S$, there exists a finite order character $\xi:C_F\to \C^\x$ such that $\xi|_{\OF_S^\x}=\xi_S$ and $\xi$ is unramified outside $S\cup\stt{\frakl}$.
 \end{lm}
 \begin{proof}Let $U_F:=(F\ot\R)^\x\prod_{v\in\bfh}\OFv^\x$ and let $U_F^{(\frakl)}=\stt{x\in U_F\mid x_\frakl=1}$  be a closed subgroup of $U_F$. We first extend $\xi_S$ to a character $\xi$ of $U_F^{(\frakl)}$ so that $\xi$ is trivial on the $v$-component of $U_F^{(\frakl)}$ for each $v\not\in S$. Since $\OF^\x\cap U_F^{(\frakl)}=\stt{1}$, we can further extend $\xi$ to the product $\OF^\x U_F^{(\frakl)}$ by requiring $\xi|_{\OF^\x}=1$. Since the product $\OF^\x U_F^{(\frakl)}$ is a closed subgroup in $U_F$, by \pont duality  $\xi$ can be further extended to $U_F$ which is trivial on $\OF^\x$. The image of $U_F$ has finite index in $C_F=\A_F^\x/F^\x$, so we can extend $\xi$ to $C_F$ and obtain desired idele class character of finite order. 
 \end{proof}

The conductor $\frakC_\brch$ of $\brch$ can be written as $\frakC_\brch=\frakC_0\frakl_0^m$ with $(\frakC_0,\frakl_0)=1$. The minimal condition for $\brch$ implies that $\frakC_0$ has a decomposition \[\frakC_0=\frakC^+\frakC^-,\quad (\frakC_0,p\frakd_{K/F}\frakl_0)=1;\,(\frakC^+,\ol{\frakC^+})=1,\] where $\frakC^-$ is only divisible by primes inert in $K$ and $\frakC^+$ is only divisible by primes split in $K$. Put \beq\label{E:nm2}\frakn:=\frakd_{K/F}\frakc^+(\frakc^-)^2\frakl_0^{m_0};\quad \frakm=\frakd_{K/F} \frakc^+\frakc^-\frakl_0^{m_0}\divides\frakn,\eeq where $\frakc^+=\frakC^+\cap \OF$ and $\frakc^-=\frakC^-\cap \OF$. With the transfer map $\sV:\Lam\to \Lam_K$, we define 
 \beq\label{E:31}\cG(\frakm,\frakl_0):=\sV\left(\frac{2^d\zeta_{F_{\frakl_0}}(1)}{\boldsymbol\zeta^\vee_{F,p}}\cdot \cE(\frakm,\frakl_0)\right).\eeq
By definition, we have \[C_i(0,\cG(\frakm,\frakl_0))=1\text{ for all }i=1,\dots,h_+; \quad C(\fraka,\cG(\frakm,\frakl_0))\in \Lam_P,\] where $\Lam_P$ is the localization of $\Lam_K$ at the augmentation ideal $P$. 
Let $\theta_\brch^\circ\in S_1(\frakn,\brch_+\tau_{K/F})$ be the theta series of weight one defined by 
\beq\label{E:theta2}C(\fraka,\theta_\brch^\circ)=\sum_{\substack{(\frakA,\frakC_\brch \Sgbar_p^{\rm reg})=1,\\\rmN\mathfrak A=\fraka}}\phi(\mathfrak A).\eeq
In particular, if $\frakp$ is a prime of $\cOF$ above $p$ with $\frakp \cOK=\frakP \ol\frakP$ and $\frakP\in\Sg_p$, then 
\[C(\frakp,\theta_\brch^\circ)=\begin{cases}\phi(\frakP)&\text{ if }\frakP\not\in\Sg_p^{\rm irr},\\
\phi(\frakP)+\phi(\Pbar)&\text{ if }\frakP\in \Sg_p^{\rm irr}. \end{cases}\]

We shall consider the spectral decomposition of the product  \[e_\Ord(\theta_\brch^\circ\cdot\cG(\frakm,\frakl_0))\in \mathbf{S}^\Ord(\frakn,\brch_+\tau_{F/F}, \Lambda_P).\] To be more precise, we denote for simplicity $\mathbf{S}$ for $\mathbf{S}^{\Ord}(\frakn,\brch_+\tau_{F/F}, \Lambda_K)$ and let  $\mathbf{K}$ be the field of fractions of $\Lambda_K$. Let $\bfK'$ be a finite extension of $\bfK$ containing all Fourier coefficients of normalized newforms in $\bfS^{\Ord}\ot_{\Lam_K}\ol\bfK$. Let $\mathbf{S}^{\perp}_{\bfK'}$ be the subspace of $\bfS\ot_{\Lam_K}\bfK'$ generated by 
\[\stt{V_\frakb\calF\mid \calF\text{ a new form of level }\frakm\text{ in  }\bfS\ot_{\Lam_K}\bfK'\text{ and }\frakm\frakb\divides \frakn,\, \calF\neq\bftheta_\brch,\,\bftheta_{\brch^c}}\]
and let $\bfS^{\perp}_{\bfK}=\bfS^{\perp}_{\bfK'}\cap \bfS\ot_{\Lam_K}\bfK$. Then it follows from the theory of newforms \cite[Prop. 1.5.2]{Wiles88Inv} that we have a spectral decomposition (a direct sum as Hecke modules) \begin{equation}\label{spectraldec}\mathbf{S} \otimes_{\Lambda_K} \mathbf{K}:= \mathbf{K} \bftheta_{\phi}\oplus\mathbf{K} \bftheta_{\phi^c} \oplus \mathbf{S}^{\perp}_{\bfK}.\end{equation}
Note that $\frakn$ is the tame conductor of $\bftheta_\brch$ and $\bftheta_{\brch^c}$, so $\bfS^\perp_\bfK$ is the space interpolating the orthogonal complement of the subspace spanned by classical specializations of $ \bftheta_{\phi}$ and $\bftheta_{\phi^c}$ under the Petersson inner product. \begin{defn}\label{D:H.3}Let $\aA$ and $\bB$ be the unique elements in $\bfK$ such that
 \beq\label{E:32}\sH:=e_\Ord(\theta_\brch^\circ \cG(\frakm,\frakl_0))-\aA\cdot \bftheta_\brch-\bB\cdot\bftheta_{\brch^c}\in \bfS^\perp_\bfK.\eeq\end{defn}
 \begin{prop}\label{P:3RS}With the ray class character $\brch$ and the ideal $\frakm$ as above, we have 
 \begin{align*}\e_\Sg^s(\bB)=&\frac{1}{\Dmd{\rmN\frakm}^s}\cdot \frac{2^d\cL_{\Sg}(s,0,\bfone)\cL_{\Sg}(s,0,\chi)}{h_{K/F}\cdot \cL_{\Sg}(s,-s,\chi)\zeta_{F,p}(1-s)}\\
 &\times \frac{\e_\Sg^{-s}(\frakf\frakd_K^2)}{\prod_{\frakP\in\Sg_p^{\rm irr}}(1-\e_\Sg^s(\Frob_{\Pbar}))\prod_{\frakP\in\Sg_p}(1-\e_{\Sg}^s(\Frob_{\Pbar}))} \cdot {\rm Ex}_{\frakl_0}(s),
  \end{align*}
  where
  \[{\rm Ex}_{\frakl_0}(s)=\frac{(1-\rmN\frakl_0^{-1}\e_\Sg^s(\frakL_0^{-1}))(1-\rmN\frakl_0^{-1}\e_\Sg^s(\ol{\frakL}_0^{-1}))(1-\rmN\frakl_0^{-1}\chi\e_\Sg^s(\frakL_0^{-1}))(1-\rmN\frakl_0^{-1}\chi\e_\Sg^s(\ol{\frakL}_0^{-1}))}{(1-\rmN\frakl_0^{-1})^2(1-\rmN\frakl_0^{-1}\chi(\frakL_0)\e_\Sg^s(\frakL_0\ol{\frakL}_0^{-1}))(1-\rmN\frakl_0^{-1}\chi(\ol{\frakL})\e_\Sg^s(\ol{\frakL}_0\frakL_0^{-1}))}.\]
  In particular, $\cC(\brch^c,\brch)$ and $\cC(\brch,\brch^c)$ are both non-zero.
  \end{prop}
The proof of \propref{P:3RS} will be given in \secref{S:RS2}.

 \def\Tdag{\bfT^\dagger}
\def\bfal{\boldsymbol\alpha}
\def\char{{\rm char}}
\def\rB{r_B}
\def\rA{r_A}
\def\kOrd{o}
\def\kLog{\ell^{\rm cyc}}
\def\varep{\varepsilon}
\section{Modular construction of cohomology classes and the applications to $p$-adic $L$-functions}\label{S:Galois}

\subsection{Ribet's construction of cohomology classes}
 Let $\bfS^\perp:=\bfS^\perp_{\bfK}\cap (\bfS\ot_{\Lam_K}{\Lam_P})$. Let ${\bf T^\perp} \subset \mathrm{End}_{\Lambda_K}(\mathbf{S}^{\perp})$ be the image of the Hecke algebra $\bfT=\bfT(\frakn,\brch_+\tau_{K/\Q},\Lam_K)$ restricted to $\bfS^\perp$. The computation of \cite[\S 5]{DKV18Ann} provides a general method to compute the $\sL$-invariant $\sL_\chi^\psi$  of $\chi$ along an additive homomorphism $\psi:G_{K,S}\to\Cp$ introduced in \defref{D:Gross} if one can construct an explicit Hecke eigensystem of $\bfT^\perp$ valued in some artinian local ring $W$. 

Fix a topological generator $\gamma_0\in \Gal(K_{\Sg_p}/K)\iso\Zp$. We have $\Lam_K=\cW\powerseries{X}$ with the variable \beq\label{E:thevariable}X:=([\gamma_0]-1)/\log_p(\e_\Sg(\gamma_0))\in \Lam_K.\eeq
Let $S$ be a set of prime divisors of $p\frakn$. For any finite extension $L/F$, denote by $G_{L,S}$ the Galois group of the algebraic extension of $L$ unramified outside primes above $S$. Let $r:=r_\Sg(\chi)$. Suppose that $r>0$. Let $\Sg_p^{\rm irr}=\stt{\frakP_1,\dots,\frakP_r}$ and put
\[\wtd\Lam:=\Lam_P\left[Y,\varep_1,\varep_2,\dots,\varep_r\right].\]
For $n\geq m\geq 1$, $z\in \Lam_P$, and $b\in\Qp$, let $I_{n,m,z,b}$ be the ideal of $\wtd\Lam$ defined by 
\[I_{n,m,z,b}:=(Y^{n+1},Y(X-Y),X^m-zY^m,X\varep_i,Y\varep_i,\varep_i^2,\varep_1\dots\varep_r+bY^n).\]
We consider the artinian local ring $W=\wtd\Lam/I_{n,m,z,b}$ with the maximal ideal $\frakm_W=(X,Y,\varep_i)$. The following proposition is extracted from \cite[\S 4 and \S 5]{DKV18Ann} with some modifications to our setting. Recall that $\universal:G_{K,\Sg_p}\to\Lam_K^\x$ is the universal character defined in \eqref{E:univ}. \begin{prop}\label{P:Ribet} Let $\lam:\bfT^\perp\to W$ be a surjective $\Lam_K$-algebra homomorphism with $n\geq 0$ such that for primes $\frakl\not\in S$, \begin{align*}\lam(T_\frakl)&\con \brch(\frakL)+\brch(\ol{\frakL})\pmod{\frakm_W} \text{ if }\frakl=\frakL\ol{\frakL}\text{ split in }K,\\
\lam(T_\frakl)&\con 0\pmod{\frakm_W}\text{ if $\frakl$ is inert in $K$},
\intertext{ and for $\frakp=\frakP\Pbar$ with $\frakP\in\Sg_p$,}
\lam(U_\frakp)&=\brch^c\universal(\Pbar)\text{ if }\frakP\not\in\Sg_p^{\rm irr},\\
 \lam(U_{\frakp})&= \brch\universal(\Pbar_i)(1+\varep_i)\text{ if }\frakP=\frakP_i\in\Sg_p^{\rm irr}.\end{align*}  
Suppose that there exists a character $\wtd\Psi:G_{K,S}\to W^\x$ and an element $Z\in (X,Y)$ such that \begin{itemize} \item[(i)] $\wtd\Psi\con \universal+Z\psi\pmod{Z\frakm_W}$ for some additive character $\psi:G_{K,S}\to \cW$, \item[(ii)]
 $\lam(T_\frakl)=\brch\wtd\Psi(\Frob_\frakL)+\brch\wtd\Psi(\Frob_{\ol{\frakL}})$ for $\frakl=\frakL\ol{\frakL}$ split in $K$ with $\frakl\not\in S$
.\end{itemize}  
If \eqref{Leo2} holds, then we have the equation \[-bY^n+\sL_\chi^\psi Z^r\con 0\pmod{Z^r\frakm_W}.\]
 \end{prop}
 \begin{proof}
Let $\Lam^\dagger\supset\Lam_P$ be the local ring of rigid analytic functions around $X=0$, \ie
\[\Lam^\dagger=\stt{\sum_{n=0}^\infty a_n X^n\in \Qbar_p\powerseries{X} \, \middle| \text{ there exist $t>0$ such that }\lim_{n\to\infty}\abs{a_n}t^n=0}.\]
 Let $\Tdag=\bfT^\perp\ot_{\Lam_K}\Lam^\dagger$ be a finite $\Lam^\dagger$-algebra. Let $\wtd\Lam^\dagger=\wtd\Lam\ot_{\Lam_P}\Lam^\dagger$ and $I=I_{n,m,z,b}\wtd\Lam^\dagger$. Let $\bfI=\Ker\lam$ be the kernel of the map $\lam:\Tdag\to \wtd\Lam^\dagger/I$. Note that $\Tdag/\bfI\iso \wtd\Lam^\dagger/I$ is a local ring. Let $\frakm_R$ be the maximal ideal of $\Tdag$ containing $\bfI$ and denote by $R=\Tdag_{(\frakm_R)}$  the localization of $\Tdag$ at $\frakm_R$. Then $R$ is a finite flat $\Lam^\dagger$-algebra. In addition, $R$ is reduced for $\frakn$ is the tame conductor of $\theta_\brch$. The theory of pseudo-characters produces a continuous irreducible Galois representation $\rho_\lam: G_{F,S}\to \GL_2(R)$ such that $\Tr\rho_\lam(\Frob_{\frakl})=T_\frakl\in R$ for any prime $\frakl\not\in S$ and the assumption (ii) implies that \beq\label{E:81}\Tr\rho_\lam|_{G_{K,S}}\con \phi\wtd\Psi+\phi^c\wtd\Psi^c\pmod{\bfI R}.\eeq
The fractional field $\Frac R$ is isomorphic to a product of fields \[\Frac R=\prod_{i=1}^t L_{\sH_i},\] and each field $L_{\sH_i}$ is a finite extension of $\Lam^\dagger$ and corresponds to a cuspidal Hida family $\sH_i$. For $i=1,\dots,t$, let $\pi_i:\Frac R\to L_{\sH_i}$ be the natural projection map. Then $\rho_{\sH_i}:=\rho_\lam\circ\pi_i:G_{F,S}\to \GL_2(L_{\sH_i})$ is the Galois representation associated with $\sH_i$. We claim the restriction $\rho_{\sH_i}|_{G_{K,S}}$ is still irreducible. Otherwise $\sH_i$ would be the Hida family $\bftheta_{\brch_1}$ of CM forms in $\Spec\bfT^\perp$ for some ray class character $\brch_1\neq \brch$ or $\brch^c$ whose specialization at some arithmetic point $P'$ above $P$ agree with $\theta_\brch^{(\Sg_p)}$, which in turns suggests that $\brch_1+\brch_1^c=\brch+\brch^c$, and $\brch_1=\brch$ or $\brch^c$, a contradiction. 

By \cite[Theorem 6.12]{HidaTil94Inv}, for each $\frakp\divides p$, the subspace $L_{\sH_i}^2$ fixed by the inertia group $\rho_{\sH_i}(I_{\frakp})$ is one-dimensional over $L_{\sH_i}$. Let $v_{\sH_i,\frakp}$ be a basis. Since $\sH_i$ are not Hida families of CM forms, the Galois representations $\rho_{\sH_i}|_{G_{K,S}}$ are irreducible. Combining the $\Sigma$-Leopoldt's conjecture \eqref{Leo2} and $\brch\neq\brch^c$, we can apply the argument \cite[Lemma 4.3]{DKV18Ann} shows that there exists $\sg_0\in G_{K,S}$ such that $\brch(\sg_0)\neq\brch(c\sg_0c)$ and $v_{\sH_i,\frakp}$ is not an eigenvector for all $\sH_i$ and $\frakp$. Choosing a basis $\stt{v_1,v_2}$ consisting of eigenvectors of $\rho_\lam(\sg_0)$ for the representation $\rho_\lam$, we may assume \beq\label{E:82}\rho_\lam(\sg_0)=\pDII{\nu_1}{\nu_2},\quad \nu_1\con \brch(\sg_0)\pmod{\frakm_R};\quad\nu_2\con\brch(c\sigma_0c)\pmod{\frakm_R}.\eeq
The image $\rho_\lam(R[G_{K,S}])$ of the group ring $R[G_{K,S}]$ is of the form
\[\rho_\lam(R[G_{K,S}])=\pMX{R}{R_{12}}{R_{21}}{R}\subset \Mat_2(R),\]
where $R_{ij}$ are ideals in $R$; this is a generalized matrix algebra in the sense of \cite[Theorem 1.4.4]{BC09Ast324}. Note that $R_{12}$ is a faithful $R$-module since $\rho_{\sH_i}|_{G_{K,S}}$ is irreducible for all $i$. Writing 
 \[\rho_\lam(\sg)=\pMX{a(\sg)}{b(\sg)}{c(\sg)}{d(\sg)}\text{ for }\sg\in G_{K,S},\] it follows from \eqref{E:81} and \eqref{E:82} that \[a(\sg)\con \brch\wtd\Psi(\sg)\pmod{\bfI},\quad d(\sg)\con\brch\wtd\Psi(c\sg c)\pmod{\bfI};\quad R_{12}R_{21}\subset \bfI R.\] 

For each $\frakp\divides p$, the $\rho_\lam(I_\frakp)$-fixed vector $v_{\frakp}:=(v_{\sH_i,\frakp})_{i=1,\dots,t}$ is an eigenvector of $\rho_\lam(G_{K_{\Pbar}})$ such that $\rho_\lam(\sg)v_\frakp=\bfal_{\frakp}(\sg)v_\frakp$, where 
 $\bfal_\frakp:G_{K_{\Pbar}}\to \bfT^\perp$ is the unramified character with $\bfal_\frakp(\Frob_{\Pbar})=U_{\frakp}$. Writing $v_\frakp=A_\frakp v_1+C_\frakp v_2$
with $(C_\frakp,D_\frakp)\in (\Frac R)^2$, we obtain an invertible matrix $\pMX{A_\frakp}{B_\frakp}{C_\frakp}{D_\frakq}\in\GL_2(\Frac R)$ such that 
 \[\pMX{a(\sg)}{b(\sg)}{c(\sg)}{d(\sg)}\pMX{A_\frakp}{B_\frakp}{C_\frakp}{D_\frakp}=\pMX{A_\frakp}{B_\frakp}{C_\frakp}{D_\frakp}\pMX{\bfal_\frakp(\sg)}{*}{0}{*} \text{ for all }\sg\in G_{K_{\ol{\frakP}}}.\]
It follows that
 \beq\label{E:41}C_\frakp\cdot b(\sg)=A_\frakp\cdot (\boldsymbol\alpha_\frakp(\sg)-a(\sg))\text{ for }\sg\in G_{K_{\Pbar}}.\eeq
Note that $A_\frakp$ and $C_\frakp$ both belong to $(\Frac R)^\x$ since $v_{\sH_i,\frakp}$ is not an eigenvector of $\rho_\lam(\sg_0)$ for each $i$. Define the function \[\sK:G_{K,S}\to R_{12},\quad \sK(\sg)=b(\sg)/d(\sg). \]For any $R$-submodule $R'\supset \frakm_R R_{12}$ of $R_{12}$,  the reduction of $\sK$ modulo $R'$  \[\ol{\sK}:=b/d\pmod{R'}=\phi^{-c} b\pmod{R'}:G_{K,S}\to R_{12}/R'\] 
 is a continuous one-cocycle in $Z^1(G_{K,S},\chi\ot R_{12}/R')$. We claim that if the class $[\ol{\sK}]\in \rmH^1(K,\chi\ot R_{12}/R')$ represented by $\ol{\sK}$ is zero, then $R_{12}=R'$. We can write $b(\sg)\pmod{R'}=(\phi^c(\sg)-\phi(\sg))z$ for some $z\in R_{12}/R'$. 
 Evaluating $b(\sg)$ at $\sg=\sg_0$, we immediately see that $z=0$ from \eqref{E:82} and hence $b(\sg)\pmod{R'}$ is zero. Since $R_{12}$ is the $R$-module generated by $\stt{b(\sg)}_{\sg\in G_{K,S}}$, we conclude $R_{12}=R'$. In particular, this shows that $\sK\pmod{\frakm_R R_{12}}$ represents a non-zero class \beq\label{E:kappa}\kappa_\lam:=[\ol{\sK}]\in\rmH^1(K,\chi\ot R_{12}/\frakm_R). \eeq
 
Let $\frakp$ be a prime of $\OF$ above $p$ and write $\frakp\OF=\frakP\Pbar$ with $\frakP\in \Sg_p$. Suppose that $\frakP\not\in\Sg_p^{\rm irr}=\stt{\frakP_1,\dots,\frakP_r}$, \ie  $\brch\neq \brch^c$ on $G_{K_{\Pbar}}$. Then there exists $\sg_\frakp$ with $\bfal_{\frakp}(\sg_\frakp)-a(\sg_\frakp)\in R^\times$. We thus obtain
 \[\sK(\sg)=\frac{A_\frakp}{C_\frakp}\cdot\frac{(\bfal_{\frakp}(\sg)-a(\sg))}{d(\sg)},\]
 and $\ol{\sK}(\sg)=\frac{A_\frakp}{C_\frakp}(1-\chi(\sg))\pmod{\frakm_RR_{12}}$ for all $\sg\in G_{K_{\ol{\frakP}}}$. This shows that the class $[\ol{\sK}]$ is locally trivial at $\ol{\frakP}$ for any $\frakP\in\Sg_p\bksl \Sg_p^{\rm irr}$. Let $R'_{12}$ be the submodule of $R_{12}$ generated by $\stt{\sK(\sg)}_{\sg\in G_{K_{\Pbar_i}}}$ with $\frakP_i\in \Sg_p^{\rm irr}$. Then $\ol{\sK}:G_K\to R_{12}/(\frakm_RR_{12}+R'_{12})$ is a cocycle whose class $[\ol{\sK}]$ is locally trivial outside $\Sg_p$, and we see that $[\ol{\sK}]=0$ by the $\Sg$-Leopodlt conjecture \eqref{Leo2}. From the above claim, we obtain that $\frakm_R R_{12}+R'_{12}=R_{12}$ and hence $R'_{12}=R_{12}$ by Nakayama's lemma. Let $\wtd\varep_i=\brch\universal(\Fr_{\Pbar_i}^{-1})U_p-1\in R$ and $\wtd Z\in R$ be liftings of $\varep_i$ and $Z$ respectively. Then the assumption (i) implies that $\bfal_{\frakp_i}(\sg)-a(\sg)\in \wtd\varep_i R+\wtd Z R +\bfI$ for $\sg\in G_{K_{\Pbar_i}}$. Therefore, we conclude that \beq\label{E:65.R}R_{12}\subset \sum_{i=1}^r\frac{A_{\frakp_i}}{C_{\frakp_i}}(\wtd\varep_iR+\wtd ZR+\bfI).\eeq
For each integral prime $\frakQ$ of $K$, denote by \[\pairing_{\frakQ}:\rmH^1(K_{\frakQ},\chi\ot R_{12}/\frakm_R)\times \rmH^1(K_{\frakQ},\chi^{-1}(1))\to R_{12}/\frakm_R\]the local Tate pairing. We choose a basis $\stt{u_1,u_2,\dots,u_r}$ of $\rmH^1_{(0,\emptyset)}(K,\chi^{-1}(1))$. 
Since the class $\kappa_\lam$ is unramified outside $\Sg_p\cup \Sgbar_p^{\rm irr}$, we obtain
 \beq\label{E:PT.R}\sum_{i=1}^r\pair{\loc_{\Pbar_i}(\kappa_\lam)}{\loc_{\Pbar_i}(u_j)}_{\Pbar_i}=0\text{ in }  R_{12}/\frakm_R.\eeq
For $\sg\in G_{K_{\Pbar_i}}$, we have \begin{align*}\brch^c\wtd\Psi^{c}(\sg^{-1})(\bfal_{\frakp}(\sg)-a(\sg))&\con \universal(\sg)(1+\varep_i)^{\kOrd_i(\sg)}-(\universal(\sg)+\psi(\sg)Z)\\
 &\con \kOrd_i(\sg)\varep_i-\psi(\sg)Z\pmod{\bfI+\wtd Z \frakm_R}.\end{align*}
Put $o_{ij}=\pair{\Ord_{\Pbar_i}}{\loc_{\Pbar_i}(u_j)}_{\Pbar_i}$ and $\psi_{ij}=\pair{\loc_{\Pbar_i}(\psi)}{\loc_{K_{\Pbar_i}}(u_j)}_{\Pbar_i}$. According to \defref{D:Gross}, 
\[\sL_\chi^\psi:=(-1)^r\frac{\det\left((\psi_{ij})\right)}{\det\left((o_{ij})\right)}.\]
It follows from \eqref{E:PT.R} that
 \[\sum_{i=1}^{r}\frac{A_{\frakp_i}}{C_{\frakp_i}}(o_{ij}\wtd\ep_i-\psi_{ij}\wtd Z+m_{ij}')\con 0\pmod{\frakm_R R_{12}}\text{ for some } m'_{ij}\in \bfI+\wtd Z\frakm_R.\]
Combining \eqref{E:65.R}, we obtain
 \[\sum_{i=1}^r\frac{A_{\frakp_i}}{C_{\frakp_i}}(\kOrd_{ij}\wtd\varep_i-\psi_{ij}\wtd Z+m_{ij})=0\text{ for some }m_{ij}\in \bfI+\wtd Z\frakm_R+\wtd \varep_i \frakm_R.\]
 Hence, 
 \[\det(o_{ij}\wtd\varep_i-\psi_{ij}\wtd Z+m_{ij})=0.\]
 Applying $\lam$ on both sides, we obtain 
 \[\det(o_{ij}\varep_i-\psi_{ij}Z+\lam(m_{ij}))=0,\quad \lam(m_{ij})\in (Z,\varep_i)\frakm_W.\]
 Since $\varep_iZ=\varep_i\frakm_W^r=0$, we find that that 
 \begin{align*}\det(o_{ij}\varep_i-\psi_{ij}Z+\lam(m_{ij}))&\con \det(o_{ij})\varep_1\dots\varep_r+(-1)^r\det(\psi_{ij})Z^r\pmod{Z^r\frakm_W}\\
 &\con \det(o_{ij})\left(-bY^n+\sL_\chi^\psi Z^r\right)\pmod{Z^r\frakm_W}.
 \end{align*}
 The proposition follows immediately.  
 \end{proof}
 
\subsection{Applications to the Katz $p$-adic $L$-functions}
Let $\chi$ be a non-trivial anticyclotoic character of conductor $\frakf\OK$ as in \subsecref{SS:RS}. Recall that in the $\Lam$-adic cusp form $\sH$ in \defref{D:H.3} is given by
\[\sH=e_\Ord(\cG(\frakm,\frakl_0)\cdot \theta^\circ_\brch)-\frac{1}{B}\bftheta_{\brch^c}-\frac{1}{A}\bftheta_{\brch}\in \bfS^\perp_\bfK.\]
with $A=\cC(\brch,\brch^c)^{-1}$ and $B=\cC(\brch^c,\brch)^{-1}$ in $\bfK^\x=(\Frac\Lam_P)^\x$. We shall assume that the Leopoldt conjecture for $F$ holds, which is, by a result of Colmez \cite{Colmez88Inv}, equivalent to the following
\begin{equation}\label{Leo1}\tag{$F$-Leo}\zeta_{F,p}(s)\text{ has a simple pole at $s=1$}\iff \frac{1}{\boldsymbol\zeta_{F,p}^\vee}\in X\Lam_P .\end{equation}
By the constant term formula \eqref{E:57} and the definition of the $\Lam$-adic Eisenstein series $\cG(\frakm,\frakl_0)$ in \eqref{E:31}, we find that \eqref{Leo1} implies that
\beq\label{E:ELeo}\cG(\frakm,\frakl_0)\con 1\pmod{X}.\eeq
We continue to assume that $r=r_\Sg(\chi)>0$ in this subsection. For $i=1,\dots,r$, we put 
\[
\wtd\varep_i:=\brch\universal(\Fr_{\Pbar_i}^{-1})U_{\frakp_i}-1\in \bfT.
\]
\begin{prop}\label{P:orderB}  Assume that \eqref{Leo1} and \eqref{Leo2} hold. Then we have \[B\in X\Lam_P.\]
\end{prop}
\begin{proof} 
Suppose that $B\not\in X\Lam_P$. Namely, $B^{-1}\in \Lam_P$. There are two cases:

Case (i): $A^{-1}\in \Lam_P$. 
Let $\wtd \Lam$ act on $\bfS^\perp$ by  
\beq\label{E:W1}\varep_i\cdot \sF=\wtd\varep_i\sF;\quad Y\cdot \sF=\varep_1\dots\varep_r\cdot \sF.\eeq Define the ideal $I$ of $\wtd\Lam$ by 
\[I:=I_{1,1,0,-1}=(X,\varep_i^2,\varep_1\dots\varep_r-Y)\]
and $W:=\wtd\Lam/I$.
Then one verifies immediately that $\sF_0:=\sH\pmod{X}$ is annihilated and $W\cdot \sF_0$ is a free of $W$-module of rank one. We have a homomorphism $\lam:\bfT^\perp\to W$ with 
\begin{align*}
\lam(T_\frakl)&=\brch(\frakL)+\brch(\ol{\frakL})\text{ if }\frakl=\frakL\ol{\frakL}\text{ is split in $K$},\\
\lam(T_\frakl)&=0\text{ if $\frakl$ is inert in $K$},\\
\lam(U_{\frakp})&=\phi(\frakP)\text{ if }\frakp=\frakP\Pbar\text{ with }\frakP\in\Sg_p\bksl \Sg_p^{\rm irr},\\
\lam(U_{\frakp_i})&=\brch(\Pbar_i)(1+\varep_i)\text{ if }\frakp_i=\frakP_i\Pbar_i\text{ with }\frakP_i \in\Sg_p^{\rm irr}.
\end{align*}
Applying \propref{P:Ribet} to $\lam$ with $\wtd\Psi=1$, $Z=X$ and $\psi=0$, we find that 
\[Y\in X^r\frakm_W=\stt{0}.\]
This is a contradiction.

Case (ii): $A\in P\Lam_P$. Let $\rA=\Ord_P(A)\geq 1$ and $a^*=A/X^{\rA}|_{X=0}\in\Cp^\x$. We let $\wtd\Lam$ act on $\bfS^\perp/(X^{\rA+1})$ by 
\[Y\cdot \sF=X\sF,\quad \varep_i\cdot \sF=\wtd\varep_i\sF. \]
Put $\sH':=-A\sH\in \bfS^\perp$. Then by \eqref{E:ELeo},
\[\sF_0':=\sH'\pmod{X^{\rA+1}}=\bftheta_{\brch}+\frac{A}{B}\theta^{(\Sg_p)}_\brch-A\theta^\circ_\brch\pmod{X^{\rA+1}}.\]Define the ideal $I$ by \[I:=I_{\rA,1,1,a^*}=(Y^{\rA+1},X-Y,Y\varep_i,X\varep_i,\varep_i^2,\varep_1\dots\varep_r+a^*Y^{\rA})\]
and $W=\wtd\Lam/I$.
Then one verifies that $\sF_0'$ is annihilated by the ideal $I$. Note that $\stt{Y^i\sF_0'}_{i=0}^{\rA-1}$ and $\stt{\varep_J\cdot \sF_0'}_{J\neq\emptyset}$ is a basis of $W\cdot \sF_0'$ and that $W$ is generated by the $\Lam_P$-algebra $W$ is generated by $\stt{Y^i}_{i=0}^{\rA-1}$ and the products $\varep_J$ with $J\neq \emptyset$ over $\Cp$. We thus conclude that $W\cdot \sF_0'$ is a free $W$-module of rank one. Using \eqref{E:CM4}, we verify easily hat $W\cdot \sF_0'$ is an invariant $\bfT^\perp$-submodule with
\begin{align*}
T_\frakl\sF_0'&=(\phi\universal(\Fr_{\frakL})+\phi\universal(\Fr_{\ol{\frakL}}))\sF_0'\text{ if $\frakl=\frakL\ol\frakL$ is split},\quad T_\frakl\sF_0'=0\text{ if $\frakl$ is inert}
;\\
U_{\frakp_i}\sF_0'&=\brch\universal(\Pbar_i)(1+\varep_i)\sF_0'.
\end{align*}
Applying \propref{P:Ribet} again with $\wtd\Psi=\universal$, $Z=Y^{\rA}$ and $\psi=0$, we get
\[Y^{\rA}\con 0\pmod{Y^{\rA}\frakm_W}.\]
This is a contradiction. 
\end{proof}
 For $\sg\in G_K$, the universal character $\universal(\sg)\pmod{X^2}$ in \eqref{E:univ} can be written as
\beq\label{E:83}\universal(\sg)=1+\eta_\Sg(\sg)X\pmod{X^2}\eeq
for some homomorphism $\eta_\Sg\in\Hom(G_{K,S},\cW)$. By definition and the choice of $X$ in \eqref{E:thevariable}, we have
\[\eta_\Sg(\sg)=-\log_p(\e_{\Sg}(\sg|_{K_{\Sg_p}})).\]
Recall that in \defref{D:Gross}  we have introduced the anticyclotomic $\sL$-invariant 
\beq\label{E:antiLinv}\sL_\chi^{\rm ac}=\sL_\chi^{\ell^{\rm ac}}\eeq 
along the \emph{anticyclotomic logarithm} $\ell^{\rm ac}\in \Hom(G_{K,S},\Cp)$ given by 
\[\ell^{\rm ac}=\log_p\circ\e_\Sg^{1-c}=\eta_{\Sgbar}-\eta_{\Sg}.\]
We next apply \propref{P:Ribet} to relate the anticyclotomic $\sL$-invariant $\sL_\chi^{\rm ac}$ to the higher order derivatives of $B$.

\begin{prop}\label{P:Linvariant}
With the Leopoldt hypotheses \eqref{Leo1} and \eqref{Leo2}, we have 
\[\sL_\chi^{\rm ac}=\frac{B}{X^r}\bigg\vert_{X=0}.\]
\end{prop}
\begin{proof} 
Write $\cG=\cG(\frakm,\frakl_0)$ and consider \[ \sH=e_\Ord(\cG\theta_\phi^\circ)- \frac{1}{B} \bftheta_{\brch^c}- \frac{1}{A}\bftheta_{\brch}\in \bfS^\perp_\bfK.\] By \propref{P:orderB}, $B= \bB^{-1} \in X\Lambda_P$. We set $\rA=\Ord_P(A)$ and $\rB=\Ord_P(B)\geq 1$. Put \[b^*:=\frac{B}{X^{\rB}}\Big\vert_{X=0}\in \Cp^\x.\] 
In what follows, $\frakl$ denotes an integral prime of $F$ with $(\frakl,p\frakm)=1$ and $\frakL$ is a prime of $K$ above $\frakl$. Let $\tau_{K/F}$ be the quadratic ideal character of $F$ associated with $K/F$. Choose a prime-to-$p$ prime ideal $\frakq_0=\frakQ_0\ol\frakQ_0$ split in $K$ such that  \[\brch(\frakQ_0)-\brch(\ol\frakQ_0)\neq 0\text{ and }\ell^{\rm ac}(\Frob_{\ol\frakQ_0})\neq 0.\]
This implies that $C(\frakq_0,\bftheta_\brch)-C(\frakq_0,\bftheta_{\brch^c})\in X\Lambda^\x_P$. Define the Hecke operators $Y_\brch$ and $Y_{\brch^c}$ in $\bfT_P$ by
\beq\label{E:Y.R}\begin{aligned}Y_\brch&:=\frac{(T_{\frakq_0}-C(\frakq_0,\bftheta_\brch))X}{C(\frakq_0,\bftheta_{\brch^c})-C(\frakq_0,\bftheta_{\brch})},\\
Y_{\brch^c}&:=\frac{(T_{\frakq_0}-C(\frakq_0,\bftheta_{\brch^c}))X}{C(\frakq_0,\bftheta_{\brch})-C(\frakq_0,\bftheta_{\brch^c})}.\end{aligned}\eeq
\subsection*{Case (I): $\Sigma_p^{\rm irr}\neq\Sigma_p $}  
 Let $\wtd\Lam$ act on $\bfS^\perp/(X^{\rB+1})$ by  \beq\label{E:MKV2}Y\cdot \sF=Y_\brch\sF;\quad \varep_i\cdot \sF=\wtd\varep_i\sF.\eeq
Let $\frakP_0\in \Sg_p-\Sg_p^{\rm irr}$ and $\frakp_0\OK=\frakP_0\Pbar_0$. Put 
\begin{align*}\sH_1:=&\frac{-B}{\brch(\frakP_0)-\brch(\ol{\frakP}_0)}(U_{\frakp_0}-\brch(\ol{\frakP}_0))\mathscr{H}=\bftheta_{\brch^c} -B\cdot e_\Ord(\cG\theta_\phi^\circ) \in \bfS^\perp;\\
\sF_1:=&\sH_1\pmod{X^{\rB+1}}.\end{align*} By \eqref{E:ELeo}, we have 
\[\sF_1=\bftheta_{\brch^c}-b^*\theta_\brch^\circ X^{\rB}\pmod{X^{\rB+1}}.\]
One verifies that 
\begin{align*}&Y\cdot \sF_1=X\sF_1=X\bftheta_{\brch^c}\pmod{X^{\rB+1}},\\
& \varep_1\dots\varep_r\cdot \sF_1=- b^*X^{\rB}\sF_1,\\
&\varep_i^2\cdot \sF_1=B\brch(\frakP_i^{-2})(U_{\frakp_i}-\brch(\frakP_i))^2\theta_\brch^\circ\pmod{X^{\rB+1}}=0.
\end{align*}
Thus $\sF_1$ is annihilated by the ideal
\[I:=I_{\rB,1,1,b^*}=(X^{\rB +1},Y-X,\varep_i^2,X\varep_i ,\varep_1\varep_2\dots\varep_r+b^*X^{\rB})\subset \wtd\Lam.\]
Let $W=\wtd\Lam/I$. For each subset $J$ of $\stt{1,2\dots,r}$, put $\varep_J:=\prod_{i\in J} \varep_i$. It is easy to see that $W$ is generated by the $\Lam_P$-algebra $W$ is generated by $\stt{X^i}_{i=0}^{\rB-1}$ and the products $\varep_J$ with $J\neq \emptyset$ over $\Cp$. We claim that $W\cdot \sF_1$ is a free $W$-module of rank one. To see it, it suffices to show that $\stt{X^i\sF_1}_{i=0}^{\rB-1}$ and $\stt{\varep_J \sF_1}_{J\neq \emptyset}$ are linearly independent. Suppose that there exists a polynomial $f(X)$ with $\deg f<\rB$ and $\al_J\in\Cp$ with $J\subset \stt{1,\dots,r}$ such that 
\[f(X)\sF_1+\sum_{J\subset\stt{1,\dots,r},\,J\neq\emptyset}\al_J \varep_J\cdot \sF_1=0.\] 
We must show that $f(X)=0$ and $\al_J=0$. Applying $Y$ on both sides, we get
\[Xf(X)\bftheta_{\brch^c}\con 0\pmod{X^{\rB+1}}.\]
Noting that $C(\frako,\bftheta_{\brch^c})=1$, we obtain $Xf(X)\con 0\pmod{X^{\rB+1}}$. This implies that $f(X)=0$ and \[0=\sum_{J\subset\stt{1,\dots,r},\,J\neq\emptyset}\al_J \varep_J\cdot \sF_1=X^{\rB}\left(\sum_{J\neq \emptyset}\al_J \varep_J\cdot \theta_\brch^\circ\right)\pmod{X^{\rB+1}}.\]
Now the claim follows from the linear independence of the set $\stt{\varep_J\cdot \theta_\brch^\circ}_{J}$ of modular forms.
From \eqref{E:CM4},  we deduce that
 \begin{align*}
  T_{\frakl}\sF_1&= \left(\phi(\frakL) \universal(\Fr_{\ol\frakL}) +  \tau_{K/F}(\frakl)\phi(\ol\frakL)\universal(\Fr_{\frakL})\right)\sF_1;\\
U_{\frakp}\sF_1&=\brch(\frakP)\universal(\Fr_{\Pbar})\sF_1,\quad \frakp=\frakP\Pbar\text{ with }\frakP\in \Sg_p\bksl \Sg_p^{\rm irr},\\
 U_{\frakp_i}\sF_1&=\brch\universal(\Fr_{\Pbar_i})(1+\varep_i)\cdot \sF_1\text{ for }i=1,\dots,r.
\end{align*}
This shows that the free $W$-module $W\cdot \sF_1$ is a $\bfT^\perp$-invariant submodule of $\bfS^\perp/(X^{\rB+1})$. This induces a surjective $\Lam_P$-alegbra homomorphism $\lam:\bfT^\perp\to W$ such that 
$t\sF_1=\lam(t)\cdot \sF_1$. This $\lam$ satisfies the assumptions in \propref{P:Ribet} with $\wtd\Psi=\universal^c$, $Z=X$ and $\psi=\ell^{\rm ac}$. Applying \propref{P:Ribet}, we obtain 
\[-b^*X^{\rB}+\sL_\chi^{\rm ac}\cdot X^r\con 0\pmod{X^r\frakm_W}.\]
This implies that $\rB\geq r$ and that $\sL_\chi^{\rm ac}=0$ if $\rB>r$ and $\sL_\chi^{\rm ac}=b^*$ if $\rB=r$.
 \subsection*{Case (II): $\Sigma_p =\Sigma_p^{\rm irr}$ and $\rB\geq \rA$.} 

 Define the $\Lambda_P$-adic cuspform
 \[\sH_2:=(-B)\cdot \sH= \frac{B}{A}\cdot \bftheta_\brch+\bftheta_{\brch^c}-B\cdot e_\Ord(\cG\theta_\phi^\circ).\]
 Let $\sF_2:=\sH_2\pmod{X^{\rB+1}}$. Let $Y$ and $\varep_i$ act on $\wtd\Lam$ as in \eqref{E:MKV2} and set $z_2=1+\frac{B}{A}\in \Lam_P$. Then one verifies that
 \begin{align*}
 Y\cdot \sF_2&=X\bftheta_{\brch^c},\quad Y\cdot \bftheta_{\brch^c}=X\bftheta_{\brch^c},\\
 X^{\rA}\sF_2&=z_2 Y^{\rA}\cdot \sF_2,\\
 \varep_1\dots\varep_r\cdot \sF_2&=-b^*X^{\rB}\theta_\brch^{(\Sg_p)}=-b^* Y^{\rB}\cdot\sF_2.
 \end{align*}
 This shows that $\sF_2$ is annihilated by the ideal \[I:=I_{\rB,\rA,z_2,b^*}=(Y^{\rB +1},\varep_i^2,X\varep_i, Y\varep_i,Y(X-Y),X^{\rA}-z_2Y^{\rA},\varep_1\dots\varep_r+b^*Y^{\rB}).\]
Let $W=\wtd \Lam/I$. We claim that $W\cdot \sF_2$ is a free $W$-module. To see it, since $W/I$ is generated by $\stt{X^i}_{i=0}^{\rA-1}$ $\stt{Y^j}_{j=1}^{\rB-1}$ and $\stt{\varep_J}_{J\neq \emptyset}$, it suffices to show that \[\stt{X^i\sF_2}_{i=0}^{\rA-1}\cup \stt{Y^j\cdot \sF_2}_{j=1}^{\rB-1}\cup \stt{\varep_J \cdot\sF_2}_{J\subset \Sg_p,\,J\neq \emptyset, \Sg_p}\]
are linearly independent over $\Cp$.  Suppose that we have a linear equation
\[f(X)\sH_2+g(Y)Y\cdot \sH_2+\sum_{J\neq \emptyset}\al_J \varep_J\cdot \sH_2\con 0\pmod{X^{\rB+1}}\]
for some polynomials $f(X),g(X)\in \Cp[X]$ with $\deg f< \rA$ and $\deg g<\rB-1$. Applying $Y$ on both sides, we obtain 
\[(f(X)+g(X)X)X\bftheta_{\brch^c}\con 0\pmod{X^{\rB+1}},\]
and hence $f(X)+g(X)X\con 0\pmod{X^{\rB}}$ as $C(\OF,\bftheta_{\brch^c})=1$. It follows that \beq\label{E:MKV1}\begin{aligned}&f(X)+g(X)X=0;\\ 
&f(X)X^{\rB-\rA}u\bftheta_\brch+\sum_{J\neq \emptyset} \al_J b^* X^{\rB}\varep_J\theta_\brch^\circ=0,\quad u=\frac{B}{AX^{\rB-\rA}}\in \Lam_P^\x.\end{aligned}\eeq
We this find that $f(X)X^{\rB-\rA}\con 0\pmod{X^{\rB}}$, which in turn shows that $f(X)=0$ for $\deg f<\rA$. In view of \eqref{E:MKV1}, we thus get $g(X)=0$ and $\al_J=0$ for $J\neq\emptyset$. On the other hand, a direct calculation shows that
 \[T_\frakl\sF_2=(\brch\Psi_2(\Frob_{\frakL})+\tau_{K/F}(\frakl)\brch\Psi_2(\Frob_{\ol\frakL}))\sF_2,\]
 where $\Psi_2:G_{K,S}\to W^\x$ is the character defined by 
 \[\Psi_2(\sg)=\universal(\sg)+\frac{\universal^c(\sg)-\universal(\sg)}{X}Y.\]
This shows that $W\cdot \sF_2$ is an invariant $\bfT^\perp$-submodule. We therefore obtain the surjective homomorphism $\lam:\bfT^\perp\to W$ such that $t\sF_2=\lam(t)\cdot \sF_2$.
Applying \propref{P:Ribet} to $\lam$ with $\wtd\Psi=\Psi_2$, $Z=Y$ and $\psi=\ell^{\rm ac}$, we find that 
\[-b^*Y^{\rB}+\sL_\chi^{\rm ac}Y^r\con 0\pmod{Y^r\frakm_W}.\]
This implies that $\rB\geq r$ and that $\sL_\chi^{\rm ac}=0$ if $\rB>r$ and $\sL_\chi^{\rm ac}=b^*$ if $\rB=r$. 
\subsection*{Case (III): $\Sg_p=\Sg_p^{\rm irr}$ and $\rA>\rB>0$}
 Let $a^*=\frac{A}{X^{\rA}}|_{X=0}\in \Cp^\x$.
Put \[\sH_3:=(-A)\sH=\bftheta_\brch+\frac{A}{B}\cdot \bftheta_{\brch^c}-A\cdot e_\Ord(\cG\theta_\brch^\circ).\]
Let $W$ act $\bfS^\perp/(X^{\rA+1})$ by  
 \[Y\cdot \sF:=\frac{(T_{\frakq_0}-C(\frakq_0,\bftheta_{\brch^c}))X}{C(\frakq_0,\bftheta_{\brch})-C(\frakq_0,\bftheta_{\brch^c})}\sF;\quad \varep_i\cdot\sF=(\brch\universal(\Fr_{\Pbar_i}^{-1})U_p-1)\sF.\]
 Put $\sF_3=\sH_3\pmod{X^{\rA+1}}$. One verifies easily that
 \begin{align*}
 Y\cdot \sF_3&=Y\cdot \bftheta_\brch\con X\bftheta_\brch\pmod{X^{\rA+1}};\\
\varep_1\dots\varep_r\cdot \sF_3&\con -A\cdot \theta_\brch^{(\Sg_p)}\con -a^*\cdot X^{\rA}\bftheta_\brch\pmod{X^{\rA+1}}.
 \end{align*}
One verifies that 
 \[X^{\rB}\sF_3=(1+\frac{A}{B})Y^{\rB}\cdot \sF_3,\text{ and }\varep_1\dots\varep_r\cdot \sF_3=-a^*Y^{\rA}\cdot \sF_3.\]
Setting $z_3=1+\frac{A}{B}\in \Lam_P$, it follows that $\sH_3$ is annihilated by the ideal
 \beq\label{E:I3}\begin{aligned} I_3&:=I_{\rA,\rB,z_3,a^*}\\
 &=(Y^{\rA +1},\varep_i^2,X\varep_i, Y\varep_i,Y(X-Y),X^{\rB}-z_3Y^{\rB},\varep_1\dots\varep_r+a^*Y^{\rA}).\end{aligned}\eeq
By a similar argument as in the case (II), one can show that $\stt{X^i\sF_3}_{i=0}^{\rB-1}$, $\stt{Y^j\cdot \sF_3}_{j=1}^{\rA-1}$ and $\stt{\varep_J\cdot \sF_3}_{J\neq\emptyset}$ are linearly independent and conclude that $W\cdot \sF_3$ is free of rank one over $W$. By \eqref{E:CM4}, we also verify that \begin{align*}T_\frakl\sF_3=&(\brch\Psi_3(\Frob_\frakL)+ \tau_{K/F}(\frakl)\brch\Psi_3(\Frob_{\ol\frakL}))\cdot \sH_3,
\intertext{ where $\Psi_3:G_{K,S}\to W^\x$ is defined by }
\Psi_3(\sg)=&\universal^c(\sg)+\frac{\universal(\sg)-\universal^c(\sg)}{X}Y.
\end{align*}
Applying \propref{P:Ribet} to $\lam$ with $\wtd\Psi=\Psi_3$, $Z=X-Y$ and $\psi=\ell^{\rm ac}$, we obtain 
\[-a^*Y^{\rA}+\sL_\chi^{\rm ac}(X-Y)^r\con 0\pmod{(X-Y)^r\frakm_W}.\]
Since $\rA>\rB$, we have $Z^{\rB}=(X-Y)^{\rB}=\frac{a^*}{b^*}Y^{\rA}\in W$ .
This implies that $\rB\geq r$ and that $\sL_\chi^{\rm ac}=0$ if $\rB>r$ and $\sL_\chi^{\rm ac}=b^*$ if $\rB=r$. 
This finishes the proof in all cases.
\end{proof}

\begin{thm}\label{T:antiregulator}Suppose that \eqref{Leo1} and \eqref{Leo2} hold. Then we have \[\Ord_{s=0}\cL_\Sg(s,-s,\chi)\geq r\]and 
\[\lim_{s\to 0}\frac{\cL_\Sg(s,-s,\chi)}{s^r}=\sL_\chi^{\rm ac}\cdot \cL^*_\Sg(0,\chi)\prod_{\frakP\in \Sg_p\bksl \Sg_p^{\rm irr}}(1-\chi(\Pbar)).\]
\end{thm}
\begin{proof}Let $B(s):=\e_\Sg^s(B)$ be the meromorphic function on $\Zp$. Then $B(s)/s^r|_{s=0}=B/X^r|_{X=0}$ based on the choice of the uniformizer $X$. By \propref{P:3RS},
\beq\label{E:310} \begin{aligned}B(s)&=\frac{1}{\Dmd{\rmN\frakm}^s}\cdot\frac{\cL_\Sg(s,-s,\chi)}{\cL_\Sg^{*}(s,\chi)\prod_{\frakP
\in\Sg_p\bksl \Sg_p^{\rm irr}}(1-\chi(\Pbar))}\cdot h_{K/F}\\
&\quad\times\frac{\zeta_{F,p}(1-s)\prod_{\frakP\in\Sg_p}(1-\e_\Sg^s(\Frob_{\Pbar}))}{\cL_\Sg(s,0,\bfone)\Dmd{\Delta_K\rmN\frakC_\chi}^s}\cdot {\rm Ex}_{\frakl_0}(s). \end{aligned}\eeq
To simply the notation, we put
 \[\cL_\Sg^{**}(s,\chi):=\cL_\Sg^*(s,\chi)\prod_{\frakP
\in\Sg_p\bksl \Sg_p^{\rm irr}}(1-\chi(\Pbar)).\]
Combined with \propref{P:Linvariant}, \propref{P:Kronecker} and ${\rm Ex}_{\frakl_0}(0)=1$, we obtain
\[\sL_\chi^{\rm ac}=\frac{B(s)}{s^r}\Big\vert_{s=0}=\frac{\cL_\Sg(s,-s,\chi)}{s^r\cL_\Sg^{**}(s,\chi)}\Big\vert_{s=0}.\]
The theorem follows.
\end{proof}
\begin{cor}\label{C:main}Suppose that the Leopoldt hypotheses \eqref{Leo1} and \eqref{Leo2} hold and that $(\frakf ,p\frakd_{K/F})=1$. Suppose that $r_\Sg(\chi)>0$ and $\frakP_1\in \Sg_p^{\rm irr}$. Then we have $L_\Sg(0,\chi)=0$ and 
\[\frac{L_\Sg(s,\chi)}{s}\Big\vert_{s=0}=\sL_\chi\cdot \cL^*_\Sg(0,\chi)\prod_{\frakP\in \Sg_p\bksl \stt{\frakP_1}}(1-\chi(\Pbar)).\]
 \end{cor}
 \begin{proof} 
Let $r=r_\Sg(\chi)$. Since $\eta_\Sg(\Fr_{\Pbar})=-\log_p(\e_\Sg(\Fr_{\Pbar}))$ for $\frakP\in\Sg_p$,
by \thmref{T:improved} we see that $\Ord_{s=0}(\cL_\Sg(s,0,\chi))=r$ and if $r=1$ and $\Sg_p^{\rm irr}=\stt{\frakP_1}$, then
\[\frac{\cL_\Sg(s,0,\chi)}{s}\Big\vert_{s=0}=\eta_\Sg(\Fr_{\Pbar_1})\cdot \cL_\Sg^{**}(0,\chi).\]
This implies that $\frac{L_\Sg(s,\chi)}{s}|_{s=0}=0$ if $r>1$ in view of \thmref{T:antiregulator}. Now suppose that $r=1$. Note that the cyclotomic logarithm $\ell^{\rm cyc}=-\eta_\Sg-\eta_{\Sgbar}$, so if $r=1$ and $\Sg_p^{\rm irr}=\stt{\frakP_1}$, then \[\sL_\chi^{\rm ac}=2\eta_\Sg(\Fr_{\Pbar_1})-\sL_\chi.\]
It follows that the cyclotomic derivative $\frac{L_\Sg(s,\chi)}{s}\Big\vert_{s=0}$ equals
\begin{align*}
\frac{\cL_\Sg(s,s,\chi)}{s}\Big\vert_{s=0}
&=\frac{2\cL_\Sg(s,0,\chi)-\cL_\Sg(s,-s,\chi)}{s}\Big\vert_{s=0}\\
&=2\eta_\Sg(\Fr_{\Pbar_1})\cL_\Sg^{**}(0,\chi)-\sL_\chi^{\rm ac}\cL_\Sg^{**}(s,\chi)\\
&=\sL_\chi\cdot \cL_\Sg^{**}(0,\chi).
\end{align*}
This completes the proof of our main theorem.
 \end{proof}
\subsection{Non-vanishing of $\cL^*_\Sg(0,\chi)$}
We study the non-vanishing of $\cL^*_\Sg(0,\chi)$ via the one-sided divisibility of Iwasawa main conjecture for CM fields in \cite{Hsieh14JAMS}. 
\begin{prop}\label{P:nonvanishing}Let $\chi$ be a non-trivial ray class character of prime-to-$p$ order. Suppose that  $(\chi,K,p)$ satisfies the assumptions \eqref{Iw1} and \eqref{Iw2} in the introduction. If the $\Sg$-Leopoldt conjecture \eqref{Leo2} is valid, then $\cL^*_\Sg(0,\chi)\neq 0$.
\end{prop}
\begin{proof}
\def\bfchi{\boldsymbol\chi}
This can be deduced from the one-sided divisibility result in Iwasawa main conjecture for the improved Katz $p$-adic $L$-function $\cL_\Sg^*(\chi)$. To begin with, let $K_\infty$ be the $\Zp^{d+1}$-extension of $K$ and for any $K_\infty/L/K$, put $\Lam(L):=\cW\powerseries{\Gal(L/K)}$. If $L_2\subset L_1$, then let $\Lam(L_1)\surjto\Lam(L_2)$ be the surjection induced by $\Gal(L_1/K)\to \Gal(L_2/K)$. Let $S$ be the union of the set of finite places of $K$ where $\chi$ is ramified and the set of places above $p$. Define the universal character $\bfchi: G_{K,S}\to \Lam(K_\infty)^\x,\,\, \sg\mapsto \chi(\sg)[\sg]$. Let $\Lam(L)^*=\Hom_{cts}(\Lam(L),\Qp/\Zp)$ be the \pont dual of $\Lam(L)$ endowed with $\Lam(L)$-module structure given by $\lam\cdot f(x)=f(\lam x)$ for $\lam,x\in \Lam(L)$. Let $\bfchi\ot \Lam(L)^*$ be the discrete $\Lam(L)$-module $\Lam(L)^*$ with the Galois action via the character $\bfchi$. Define the Selmer group
\[\Sel_{L}(\chi):=\ker\stt{\rmH^1(K_S/K,\bfchi\ot\Lam(L)^*)\to\prod_{w\in S\bksl \Sg_p}\rmH^1(I_w,\bfchi\ot\Lam(L)^*)^{D_w}},\]
where $I_w$ and $D_w$ denote the inertia and decomposition groups of $w$. Then $\Sel_{L}(\chi)$ are discrete and co-finitely generated $\Lam(L)$-modules. Let $\char_{\Lam(L)}\Sel_L(\chi)$ be the characteristic ideal of the \pont dual $\Sel_L(\chi)^*$ in $\Lam(L)$. 

Let $L$ be the compositum of the cyclotomic $\Zp$-extension and $K_{\Sg_p}$. Since  $\rmH^0(K,\bfchi\ot\Lam(K_\infty)^*)=\stt{0}$ and $L$ contains the cyclotomic $\Zp$-extension, we have \beq\label{E:control1}\Sel_{K_\infty}(\chi)^*\ot_{\Lam(K_\infty)}\Lam(L)\sim \Sel_{L}(\chi)^*\eeq is pseudo-isomorphic of finitely generated $\Lam(L)$-modules in view of  \cite[Remark 8.4]{Hsieh14JAMS}. Let $I$ be the kernel of the natural quotient map $\Lam(L)\to \Lam(K_{\Sg_p})=\Lam_K$. We have the exact sequence 
\[0\to \Sel_{K_{\Sg}}(\chi)\to \Sel_{L}(\chi)[I]\overset{\loc_{\Sgbar_p}}\longto \prod_{w\in\Sgbar_p}\left(\rmH^0(I_w,\bfchi\ot \Lam(L)^*)\ot_{\Lam(L)}\Lam_K\right)^{D_w}.\]
Taking \pont dual of the above exact sequence, we find that 
\beq\label{E:control2}\prod_{\frakP\in\Sg_p}\frac{\Lam_K}{(1-\chi(\Pbar)[\Fr_{\Pbar})}\to \Sel_L(\chi)^*\ot_{\Lam(L)}\Lam_K\to \Sel_{K_{\Sg}}(\chi)^*\to 0.\eeq
Recall that $[\Fr_{\Pbar}]\in\Lam_K$ is the image of the Frobenius at $\Pbar$ in $\Gal(K_{\Sg_p}/K)\hookto\Lam_K$ and $[\Fr_{\Pbar}]\neq 1$ by the ordinary assumption. The equation \eqref{E:control2} implies that
\[\prod_{\frakP\in\Sg_p}(1-\chi(\ol{\frakP})[\Fr_{\ol\frakP}])\cdot\char_{\Lam_K}( \Sel_{K_{\Sg_p}}(\chi)^*)\subset \char_{\Lam_K}(\Sel_L(\chi)^*\ot_{\Lam(K_\infty)}\Lam_K).\]
Combined with \eqref{E:control1}, we get 
\begin{align*}\prod_{\frakP\in\Sg_p}(1-\chi(\ol{\frakP})[\Fr_{\ol\frakP}])\cdot\char_{\Lam_K}( \Sel_{K_{\Sg_p}}(\chi)^*)\subset &\char_{\Lam(K_\infty)}(\Sel_{K_\infty}(\chi)^*\ot_{\Lam(K_\infty)}\Lam_K)\\
\subset &\char_{\Lam(K_\infty)}(\Sel_{K_\infty}(\chi)^*)\ot_{\Lam(K_\infty)}\Lam_K.
\end{align*}
On the other hand, by \cite[Theorem 2]{Hsieh14JAMS}, 
\[\char_{\Lam(K_\infty)}(\Sel_{K_\infty}(\chi)^*)\subset \cL_\Sg(\chi)\Lam(K_\infty).\]
By \thmref{T:improved}, it follows that 
\[\prod_{\frakP\in\Sg_p}(1-\chi(\ol{\frakP})[\Fr_{\ol\frakP}])\cdot\char_{\Lam_K}( \Sel_{K_{\Sg_p}}(\chi)^*)\subset \cL_\Sg(\chi)\Lam_K=\prod_{\frakP\in\Sg_p}(1-\chi(\ol\frakP)[\Fr_{\ol\frakP}])\cdot \cL^*_\Sg(\chi)\Lam_K,\]
and hence
\beq\label{E:IMCv}\char_{\Lam_K}( \Sel_{K_{\Sg_p}}(\chi)^*)\subset \cL^*_\Sg(\chi)\Lam_K.\eeq
Note that 
\[\rmH^0(I_w,\bfchi\ot\Lam_K^*)\ot\Lam_K/(X)=\stt{0}\text{ for all }w\in S\bksl \Sg_p.\]
By \cite[Lemma 8.3]{Hsieh14JAMS}, this implies that
\beq\label{E:control3} \Sel_{K_{\Sg_p}}(\chi)^*\ot_{\Lam_K}\Lam_K/(X)\iso \rmH^1_{(\emptyset,f)}(K,\chi\ot_{\Zp}\Qp/\Zp)^*\eeq
By the Poitou-Tate duality, $\rmH^1_{(0,f)}(K,\chi^{-1}(1))$ and $\rmH^1_{(\emptyset,f)}(K,\chi\ot\Qp/\Zp)^*$ have the same rank, so the $\Sg$-Leopoldt conjecture implies that $\rmH^1_{(\emptyset,f)}(K,\chi\ot\Qp/\Zp)$ is a finite group. Therefore, from \eqref{E:control3}, we deduce that the characteristic ideal of  $\Sel_{K_{\Sg_p}}(\chi)^*$ is not divisible by $X$, and by \eqref{E:IMCv}, \[\cL^*_\Sg(\chi)\pmod{X}=\cL_\Sg^*(0,\chi)\neq 0.\] This completes the proof.
\end{proof}

The following is an immediate consequence of \corref{C:main} and \propref{P:nonvanishing}. 
\begin{cor}Let $\chi$ be a non-trivial anticyclotomic ray class character of $K$ such that \begin{enumerate} 
\item $(\chi,K,p)$ satisfies the assumptions \eqref{Iw1} and \eqref{Iw2}.
\item \eqref{Leo1} and \eqref{Leo2} both hold.\end{enumerate}
Then
\[\Ord_{s=0}L_\Sg(s,\chi)=1\text{ if and only if } r_\Sg(\chi)=1 \text{ and } \sL_\chi\neq 0.\]
\end{cor}
\begin{Remark}One can also deduce the one side implication that 
if $r_\Sg(\chi)=1 \text{ and } \sL_\chi\neq 0$, then $\Ord_{s=0}L_\Sg(s,\chi)=1$ from \cite[Theorem 3.27]{BS21IMRN} combined with \cite[Theorem 2]{Hsieh14JAMS} and the existence of the trivial zeros 
\end{Remark}
  
\section{Explicit calculations in the Rankin-Selberg convolution}\label{S:RS2}
This section is devoted to the proof of \thmref{P:3RS}.

\subsection{Adelic Hilbert modular forms}Let $k$ be any integer. Let $\om$ be a Hecke character of $F$ of finite order. Let $\cA_k(\om)$ denote the space of automorphic forms of weight $k$ and central character $\om$, consisting of smooth and slowly increasing functions $\varphi:\GL_2(F)\bksl \GL_2(\A_\cF)\to\C$ such that 
\[\varphi(g u_\infty z)=\om(z)\varphi(g)J(u_\infty,\sqrt{-1})^{-k\Sg},\quad z\in \A_F^\x,\, u_\infty\in\SO_2(F\ot_\Q\R).\]
Let $\cA^0_{\ulk}(\om)\subset \cA_{\ulk}(\om)$ be the subspace of cusp forms. For any integral ideal $\frakn$ of $F$, put
\[U_1(\frakn)=\stt{g=\pMX{a}{b}{c}{d}\in\GL_2(\wh\OF)\mid c\in\frakn\wh\OF,\,d-1\in \frakn\wh\OF}.\]
Let $\cA_k(\frakn,\om)$ be the subspace of $\cA_k(\om)$ right invariant by $U_1(\frakn)$. For any $\varphi\in\cA_k(\frakn,\om)$, the Whittaker function $W_\varphi:\GL_2(\A_F)\to\C$ of $\varphi$ is defined by 
\[W_{\varphi}(g):=\int_{F\bksl \A_F}\varphi(\pMX{1}{x}{0}{1}g)\addchar_F(-x)\rmd x.\]
To any Hilbert modular form $\bff\in M_k(\frakn,\om)$, we can associate a unique automorphic form in $\varphi_{\bff}\in \cA_k(\frakn,\om)$ such that 
for every $y\in (F\ot\R)^\x$ and $a\in \wh F^\x$, we have 
\beq\label{E:FC_W}W_{\varphi_{\bff}}\left(\pDII{y a}{1}\right)=C(a\OF,\bff)\cdot \abs{ya}^\frac{k}{2}e^{-2\pi\Tr(y)}\bbI_{(F\ot \R)_+}(y).\eeq
On the other hand, a Hilbert modular form $\bff$ is given by an $h$-tuple 
\[\bff=(f_{\frakc_1},\dots f_{\frakc_{h}}),\quad f_{\frakc_i}\in M_k(\bfc_i,U_1(\frakn)).\] Then we have 
\[f_{\frakc_i}(\tau)=y^{-\frac{k\Sg}{2}}\varphi_{\bff}\left(\pMX{y}{x}{0}{1}\pDII{\bfc_i}{1}\right)(\rmN\frakc_i)^\frac{k}{2}\quad (\tau=x+\sqrt{-1}y)\in\frakH_F).\]
Then \[C(\beta\frakc_i,\bff)=c_\beta(f_{\frakc_i}),\quad C_i(0,\bff)=c_0(f_{\frakc_i}).\]

If $\bff\in S_k(\frakn,\om)$ is a Hilbert cuspidal newform of conductor $\frakn$, then denote by $\pi_{\bff}$ the unitary and irreducible automorphic representation of $\GL_2(\A_F)$ generated by $\varphi_{\bff}$. 
This representation $\pi_{\bff}$ is discrete series at all infinite places and has conductor $\frakn$ and central character $\xi$. 
Moreover, $\varphi_{\bff}$ is the new vector in $\pi_{\bff}$ and 
\beq\label{E:Lfcn}L(s,\pi_{\bff})=\Gamma_\C(s+k-1)^{[F:\Q]}D(s+\frac{k-1}{2},\bff),\eeq
where $L(s,\pi_{\bff})$ is the automorphic $L$-function associated with $\pi_{\bff}$ by Jacquet and Langlands. 
\subsection{Eisenstein series} 
The aim of this subsection is to realize the weight $k$ specialization $\nu_k(\cE(\frakm,\frakl))$ of the $\Lam$-adic Eisenstein series $\cE(\frakm,\frakl)$ as adelic Eisenstein series in \cite[Chapter V]{Jacquet72LNM278}. For each place $v$, let $\Om_{\Fv}$ be the local component of the \Teich character $\Om_F$ at $v$.
Let $\cD$ be the pair 
\[\cD=(k,\frakm,\frakm_0),\]
where $k$ is a positive integer, and $\frakm\subset \frakm_0$ are integral ideals of $F$ such that $(\frakm,p\OF)=1$ and $(\frakm_0,p\frakd_F)=1$. Define 
$\Phi_{\cD}=\ot_{v\in\Sg}\Phi_{\cD,v}\ot_{v\in\bfh}'\Phi_{\cD,v}\in \cS(\A_F^2)$ by 
\begin{itemize}
\item $\Phi_{\cD,v}(x,y)=2^{-k}(x+\sqrt{-1}y)^ke^{-\pi(x^2+y^2)}$ if $v\divides\infty$, 
\item $\Phi_{\cD,v}(x,y)=\bbI_{\frakm\OFv}(x)\bbI_{\frakd_{\Fv}^{-1}}(y)$ if $v\ndivides p\frakm_0\infty$,
\item $\Phi_{\cD,v}(x,y)=\bbI_{\frakm\OFv}(x)\bbI_{\OFv^\x}(y)$ if $v\divides\frakm_0$,
\item $\Phi_{\cD,v}(x,y)=
\Om_{\Fv}^{-k}(x)\bbI_{\OFv^\x}(x)\bbI_{\OFv}(y)$ if $v\divides p$.\end{itemize}
Let $f_{\Phi_\cD,s}=\ot_v f_{\Phi_{\cD,v},s}\in I(\Om_F^k,\bfone,s)$, where
\beq\label{E:Godement2}f_{\Phi_{\cD,v},s}(g)=\Om_{\Fv}^k(\det g)\abs{\det g}_v^{s+\onehalf}\int_{\Fv^\x}\Phi_{\cD,v}((0,t_v)g)\Om_{\Fv}^k(t_v)\abs{t_v}_v^{2s+1}\rmd^\x t_v.\eeq
Consider the adelic Eisenstein series $E_\A(g,f_{\Phi_{\cD,s}})$ defined by \eqref{E:Eisenstein.2}. It is easy to verify that $E_\A(g,f_{\Phi_{\cD,s}})$ is an automorphic form in $\cA_k(\frakm,\Om_F^k)$ whenever it is defined. 
For $\tau=x+\sqrt{-1}y\in \frakH_F$ and $\frakc=\bfc\OF$ for $\bfc\in\wh F^{(p\frakm)\x}$, put 
 \[E_\frakc(\tau):=y^{-\frac{k\Sg}{2}}E_\A(\pMX{y}{x}{0}{1}\pDII{\bfc}{1},f_{\Phi_\cD,s})|_{s=\frac{1-k}{2}}\cdot\abs{\bfc}^{-\frac{k}{2}}.\]
Let $\bfE_k(\frakm,\frakm_0)=(E_{\frakc_1},\dots E_{\frakc_{h_F^+}})\in M_k(\frakm,\Om_F^k)$. This way $E_\A(-,f_{\Phi_{\cD,s}})|_{s=\frac{1-k}{2}}$ is the automorphic form associated with $ \bfE_k(\frakm,\frakm_0)$ via the relation \eqref{E:FC_W}. For any integer $n$ coprime to $p$, put $\Dmd{n}:=n\cdot \Om((n))$. 

\begin{prop}\label{P:3FE}For $k\geq 2$, we have \[\nu_k( \cE(\frakm,\frakm_0))= \frac{\rmN\frakm}{\Dmd{\rmN\frakm}^k}\cdot \bfE_k(\frakm,\frakm_0).\]
\end{prop}
\begin{proof}As in the proof of \propref{P:FC1}, the $\beta$-th Fourier coefficient of $E_\frakc$ is given by 
\[c_\beta(E_\frakc)=\prod_{v\in\bfh}W_\beta(\pDII{\bfc_v}{1},f_{\Phi_{\cD,v},s})\abs{\bfc_v}_{v}^{-\frac{k}{2}}.\]
Let $v$ be a prime $\frakl$. Put $\bfal_\frakl=\rmN\frakl^{-1}\Dmd{\rmN\frakl}^k$. If $\frakl\ndivides\frakm_0$, then it is easy to see that
\begin{align*}W_\beta(\pDII{\bfc_v}{1},f_{\Phi_{\cD,v},s})|_{s=\frac{1-k}{2}}&=\abs{\bfc_v}_{v}^{\frac{k}{2}}\sum^{v(\beta \bfc_v)}_{i=v(\frakm)}\bfal_\frakl^i\text{ if }\frakl\ndivides p\frakm_0,\\
W_\beta(1,f_{\Phi_{\cD,v},s})&=\bbI_{\OFv}(\beta)\text{ if }\frakl\divides p.\end{align*}
For $\frakl\divides\frakm_0$, we have
\begin{align*}
W_\beta(1,f_{\Phi_{\cD,v},s})=&\int_{\Fv^\x}\bbI_{\frakm\OFv}(t)\wh{\bbI_{\OFv^\x}}(-\beta t^{-1})\Om^k(t)\abs{t}_v^{2s}\rmd^\x t.
\end{align*}
Since $\wh{\bbI_{\OFv^\x}}=\bbI_{\OFv}-\abs{\uf_v}_v\bbI_{\uf_v^{-1}\OFv}$, we find that 
\begin{align*}W_\beta(1,f_{\Phi_{\cD,v},s})=&\sum^{v(\beta)}_{i=v(\frakm)}\Om_F^k(\uf_v^i)\abs{\uf_v^i}_v^{1-k}-\abs{\uf_v}_v\sum^{v(\beta)+1}_{i=v(\frakm)}\abs{\uf_v}_v\Om^k(\uf_v^i)\abs{\uf_v^i}_v^{1-k}\\
=&\sum^{v(\beta)}_{i=v(\frakm)}\bfal_\frakl^i-\rmN\frakl^{-1}\sum^{v(\beta\frakl)}_{i=v(\frakm)}\bfal_\frakl^i.
\end{align*}
Combined the above formulas, we deduce  that 
\begin{align*}C(\fraka,\bfE_k(\frakm,1))&=\sum_{\fraka\subset \frakb\subset \frakm,\,(p,\frakb)=1}\rmN\frakb^{-1}\Dmd{\frakb}^k=\frac{\Dmd{\rmN\frakm}^k}{\rmN\frakm}C(\fraka,\nu_k(\cE(\frakm,1)),\\
C(\fraka,\bfE_k(\frakm,\frakm_0))&=C(\fraka,\prod_{\frakl\divides\frakm_0}(1-\rmN\frakl^{-1}U_\frakl)\bfE_k(\frakm,1)).
\end{align*}
We next investigate the constant terms. An elementary calculation shows that 
\begin{align*}
f_{\Phi_{\cD},s}(\pDII{\bfc}{1})=&0,\\
Mf_{\Phi_{\cD},s}(\pDII{\bfc}{1})\bigg\vert_{s=\frac{1-k}{2}}=&\frac{\Dmd{\rmN\frakm}^k}{2^d\rmN\frakm}\cdot \zeta_F(1-k)\prod_{\frakp\divides p} (1-\rmN\frakp^{k-1})\prod_{\frakl\divides\frakm_0}\frac{1}{\zeta_{F_\frakl}(1)}.
\end{align*}
We obtain $C_i(0,\bfE_k(\frakm,\frakm_0))$ equals
\[C(\fraka,\prod_{\frakl\divides\frakm_0}(1-\rmN\frakl^{-1}U_\frakl)\bfE_k(\frakm,1))=\zeta_{F,p}(1-k)\cdot \frac{\Dmd{\rmN\frakm}^k}{2^d\rmN\frakm}\prod_{\frakl\divides\frakm_0}\frac{1}{\zeta_{F_\frakl}(1)}.\]
In particular, $C_i(0,\bfE_k(\frakm,1))=\frac{\Dmd{\rmN\frakm}^k}{\rmN\frakm}C_i(0,\nu_k(\cE(\frakm,1))$.
Combining everything together, we conclude that 
\begin{align*}\bfE_k(\frakm,1))&=\frac{\Dmd{\rmN\frakm}^k}{\rmN\frakm}\nu_k(\cE(\frakm,1)),\\
\bfE_k(\frakm,\frakm_0)&=\prod_{\frakl\divides\frakm_0}\left(1-\rmN\frakl^{-1}U_\frakl\right)\bfE_k(\frakm,1).\end{align*}
The proposition follows. 
\end{proof}

\subsection{Whittaker functions}
Let $\pi$ be a irreducible automorphic representation of $\GL_2(\A_F)$ with the conductor $\frakn$. 
For each place $v$, denote by $\sW(\pi_v)=\sW(\pi_v,\addchar_{\Fv})$ the Whittaker model of $\pi_v$ with respect to $\addchar_{\Fv}$.  Recall that this is a space of smooth functions $W$ on $\GL_2(\Fv)$ satisfying the following properties:
\begin{itemize} \item $\sW(\pi_v)$ is invariant for the right translation $\rho$ of $\GL_2(\Fv)$ and the resulting representation is isomorphic to $\pi_v$;
\item For each $W\in\sW(\pi_v)$, 
\[W(\pMX{1}{x}{0}{1}g)=\addchar_{\Fv}(x)W(g)\quad x\in \Fv, \,g\in \GL_2(\Fv).\]
\end{itemize}
If $v\in\Sg$ and $\pi_v$ is a discrete series of weight $k_v$, let $W_{\pi_v}\in \sW(\pi_v)$ be the Whittaker function of minimal weight $k_v$, \ie the unique element in $\sW(\pi_v)$ such that  \[W_{\pi_v}(\pDII{y}{1})=y^\frac{k_v}{2}e^{-2\pi y}\bbI_{\R_+}(y).\] If $v$ is a finite place of $F$, let $W_{\pi_v}$ be the new Whittaker function in $\sW(\pi_v,\addchar_{\Fv})$, \ie $W_{\pi_v}$ is the unique Whittaker function right invariant by $U_1(\frakn)$ with $W_{\pi_v}(1)=1$. If $\pi_v$ is a principal series and $\al: \Fv^\x\to \C^\x$ is a character, we let $W_{\pi_v}^{(\al)}$ be the unique Whittaker function in $\sW(\pi_v)$ such that 
\[W_{\pi_v}^{(\al)}(\pDII{a}{1})=\al(a)\abs{a}_{v}^\onehalf\bbI_{\OFv}(a).\]

\subsection{The bilinear pairing}
Let $\bff\in S_{k_1}(\frakn p,\om)$ be a $p$-stabilized cuspidal Hilbert newform of weight $k_1\Sg$. Let $\pi_{\bff}$ be the cuspidal automorphic representation generated by $\varphi_{\bff}$. Let $\frakn$ be the conductor of $\pi_{\bff}$. Assume that 
\[ (p,\frakn\Delta_F)=1.\] 
For each $v\divides p$, let $\pi_{\bff,v}=\pi(\al_{\bff,v},\beta_{\bff,v})$ be a principal series associated with unramified characters $\al_{\bff,v},\beta_{\bff,v}:\Fv^\x\to\C^\x$. The Whiattker function $W_{\varphi_\bff}$ has a factorization
\beq\label{E:W1}W_{\varphi_{\bff}}=\prod_{v\divides p}W^{(\al_{\bff,v})}_{\pi_v}\prod_{v\divides \infty}W_{\pi_v}\prod_{v\ndivides p\infty} W_{\pi_v}\quad (\pi_v=\pi_{\bff,v}). \eeq
Put
\begin{align*}
S_{\rm vr}=&\stt{v:\text{places dividing }\frakn\mid L(s,\pi_{\bff,v})=1}.
\end{align*}
Denote by $\pi_{\bff}^\vee$ the contragredient of $\pi_{\bff}$. Let $\varphi_{\breve\bff}$ be the automorphic form in $\pi_{\bff}^\vee$ with the Whittaker function given by 
\beq\label{E:W2}W_{\varphi_{\breve\bff}}=\prod_{v\divides p}W^{(\beta_{\bff,v}^{-1})}_{\pi^\vee_v}\cdot \prod_{v\divides \infty}W_{\pi^\vee_v}\prod_{v\ndivides p\infty} W_{\pi^\vee_v}\quad(\pi^\vee_v=\pi_{\bff,v}^\vee).\eeq
Note that $W^{(\beta_{\bff,v}^{-1})}_{\pi^\vee_v}=W^{(\al_{\bff,v})}_{\pi_v}\ot\om_v^{-1}$. Define the $\C$-linear pairing $\pairing:\cA^0_{-k_1}(\om)\times\cA_{k_1}(\om^{-1})\to\C$ by 
  \beq\label{E:pairing}\pair{\varphi_1}{\varphi_2}=\int\limits_{\PGL_2(F)\bksl \PGL_2(\A_F)}\varphi_1(g)\varphi_2(g)\rmd^{\rmt} g.\eeq
Here $\rmd^\rmt g$ is the Tamagawa measure of $\PGL_2(\A_F)$. Let $n$ be a positive integer. Put
\[\cJ_\infty=\pDII{-1}{1}\in \GL_2(F\ot\R);\quad t_n=\pMX{0}{p^{-n}}{-p^{n}}{0}\in \GL_2(F\ot\Qp).\]
\begin{prop}\label{P:Petnorm} We have
\begin{align*}\pair{\rho(\cJ_\infty t_n)\varphi_{\bff}}{\varphi_{\breve\bff}}=&\frac{2L(1,\pi_\bff,\Ad)}{\sqrt{\Delta
_F}\zeta_F(2)}\cdot \frac{1}{2^{(k_1+1)[F:\Q]}}\prod_{v\divides\frakn}\frac{\zeta_{\Fv}(2)}{\zeta_{\Fv}(1)}\\
&\times\prod_{v\in S_{\rm vr}}\frac{1}{L(1,\pi_v,\Ad)}\prod_{v\divides p}\frac{\om_v^{-1}\al_{\bff,v}^2\Abs_v(p^{n})\zeta_{\Fv}(2)}{\zeta_{\Fv}(1)}\cdot{\rm E}_v(\Ad\bff),\end{align*}
where ${\rm E}_p(\Ad\bff)$ is the modified $p$-Euler factor for the adjoint of $\bff$ given by \begin{align*}{\rm E}_v(\Ad\bff)=&L(1,\al_{\bff,v}^{-1}\beta_{\bff,v})^{-1}L(0,\al_{\bff,v}^{-1}\beta_{\bff,v})^{-1}\\
=&(1-\al_{\bff,v}^{-1}\beta_{\bff,v}(p)\abs{p}_{v})(1-\al_{\bff,v}^{-1}\beta_{\bff,v}(p)).
\end{align*} 
\end{prop}
\begin{proof} For each place $v$, let $\pair{-}{-}:\sW(\pi_v)\times \sW(\pi_v^\vee)\to\C$ be he pairing defined by 
\[\pair{W_1}{W_2}=\int_{\Fv^\x}W_1(\pDII{a_v}{1})W_2(\pDII{-a_v}{1})\rmd^\x a_v^{\rm st},\]
where $\rmd^\x a_v^{\rm st}$ is the Haar measure on $\Fv^\x$ with $\vol(\OFv^\x,\rmd^\x a_v^{\rm st})=1$. It follows from \cite[Proposition 6]{Waldsupurger85Comp} that
\begin{align*}\pair{\rho(\cJ_\infty t_n)\varphi_{\bff}}{\varphi_{\breve\bff}}&=\frac{2L(1,\pi,\Ad)}{\sqrt{\Delta_F}\zeta_F(2)}\cdot 2^{-(k_1+1)[F:\Q]}\prod_{v\divides\frakn}\frac{\zeta_{\Fv}(2)\pair{W_{\pi_v}}{ W_{\pi^\vee_v}}}{\zeta_{\Fv}(1)L(1,\pi_v,\Ad)}\\\
&\times \prod_{v\divides p}\frac{\zeta_{\Fv}(2)\pair{\rho(t_n)W^{(\al_{\bff,v})}_{\pi_v}}{W^{(\al_{\bff,v})}_{\pi_v}\ot\om_v^{-1}}}{\zeta_{\Fv}(1)L(1,\pi_v,\Ad)}.\end{align*}
For $v\divides \frakn$, it is well-known that 
\[\frac{\pair{W_{\pi_v}}{ W_{\pi^\vee_v}}}{L(1,\pi_v,\Ad)}=\begin{cases}1&\text{ if }v\not \in S_{\rm vr},\\
\frac{1}{L(1,\pi_v,\Ad)}&\text{ if }v\in S_{\rm vr}.
\end{cases}\]
The local integral at primes above $p$ is computed in \cite[Lemma 2.8]{Hsieh21AJM}
\[\pair{\rho(t_n)W^{(\al_{\bff,v})}_{\pi_v}}{W^{(\al_{\bff,v})}_{\pi_v}\ot\om_v^{-1}}=\om_v^{-1}\al_{\bff,v}^2\Abs_v(p^{n})\cdot \gamma(0,\al_{\bff,v}^{-1}\beta_{\bff,v})\zeta_{\Fv}(1),\]
where 
\[\gamma(0,\al_{\bff,v}^{-1}\beta_{\bff,v})=\frac{\abs{\frakd_F}_v^{-\onehalf}L(1,\al_{\bff,v}\beta_{\bff,v}^{-1})}{L(0,\al_{\bff,v}^{-1}\beta_{\bff,v})}.\]
This finishes the proof.
\end{proof}
\subsection{Rankin-Selberg convolution}We keep the notation in the previous subsection. Let $\pi_1=\pi_{\bff}$. Let $\bfg^\circ\in S_{k_2}(\frakn p,\om^{-1})$ be a Hilbert cuspform of parallell weight $k_2<k_1$. Let $\pi_2$ be the cuspidal automorphic representation generated by $\varphi_{\bfg^\circ}$. Assume that 
\begin{itemize}
\item[(RS1)] $k_1-k_2$ is divisible by $p-1$;
\item[(RS2)] for every $v\divides \frakn$, $\pi_{2,v}$ is isomorphic to $\pi_{1,v}^\vee$ up to unramified twists.
\end{itemize}
We put\begin{align*}
S_0=&\stt{v:\text{places dividing }\frakn\mid \pi_{1,v}\text{ is minimal}},\\
S_{\rm dis}=&\stt{v\in S_0\mid \pi_{1,v}\text{ is a discrete series }}.
\end{align*} 
Here recall that we say $\pi_{1,v}$ is minimal if the conductor of $\pi_{1,v}$ is minimal among that of its twists by a character. We define
\beq\label{E:nmm}\frakn=\prod_{\frakq}\frakq^{e_\frakq},\quad  \frakm=\frakl_0^r\prod_{\frakq\not\in S_{\rm dis}}\frakq^{e_\frakq}\prod_{\frakq\in S_{\rm dis}}\frakq^{\lceil\frac{e_\frakq}{2}\rceil};\quad \frakm_0=\prod_{\frakq\not\in S_0}\frakq.\eeq
Let $S_p^{\rm reg}$ be a subset of $S_p$ the set of prime factors of $p$. Suppose that the automorphic form $\varphi_{\bfg^\circ}$ associated with $\bfg^\circ$ admits the Whittaker function
\beq\label{E:W3}W_{\varphi_{\bfg^\circ}}=\prod_{v\not\in S_p^{\rm reg}}W_{\pi_{2,v}}\prod_{v\in S_p^{\rm reg}}W_{\pi_{2,v}}^{(\al_{\bfg,v})}.\eeq
\begin{prop}\label{P:RS.app}With the assumptions (RS1-2) and $\cD=(k_1-k_2,\frakm,\frakm_0)$, we have
\begin{align*}&\frac{\pair{\rho(\cJ_\infty t_n)\varphi_{\bff}}{\varphi_{\bfg^\circ}\cdot E_\A(-,f_{\Phi_\cD,s-1/2})}}{\pair{\rho(\cJ_\infty t_n)\varphi_{\bff}}{\varphi_{\breve\bff}}}=\frac{(\sqrt{-1})^{(k_1-k_2)[F:\Q]}L(s,\pi_1\times\pi_2)}{2L(1,\pi_1,\Ad)}\\
&\quad\times\frac{1}{\rmN\frakm}\prod_{v\divides\frakm_0}\frac{L(1,\pi_{1,v},\Ad)}{L(s,\pi_{1,v}\times\pi_{2,v})}\prod_{v\in S_p} \frac{1}{{\rm E}_v(\ad\bff)L(s,\pi_2\ot\beta_{\bff,v})}\prod_{v\in S_p^{\rm reg}}
\frac{1}{L(1-s,\al_{\bfg,v}^{-1}\al_{\bff,v}^{-1})}.\end{align*}
\end{prop}
\begin{proof}Let $N$ be the unipotent upper triangular subgroup of $\GL_2$. For each place $v$, define the trilinear form $\Psi:\sW(\pi_{1,v})\times \sW(\pi_{2,v})\times I(\bfone_v,\bfone_v,s)\to \C$ by 
\[\Psi(W_1,W_2,f)=\int_{N(\Fv)\bksl \PGL_2(\Fv)}W(g_v)W_2(\pDII{-1}{1} g_v)f(g_v)\rmd g_v,\]
where $\rmd g_v$ is the Haar measure given by 
\[\int_{N(\Fv)\bksl \PGL_2(\Fv)}F(g_v)\rmd g_v=\int_{\GL_2(\OFv)}\int_{\Fv^\x}F(\pDII{a_v}{1}u_v)\abs{a_v}^{-1}\rmd^\x a_v^{\rm st}\rmd u_v\] with $\vol(\GL_2(\OFv),\rmd u_v)=1$. According to \cite[Chapter V]{Jacquet72LNM278} (\cf\cite[(5.11)]{HsiehChen20}), we have the basic identity\begin{align*}&\pair{\rho(\cJ_\infty t_n)\varphi_{\bff}}{\varphi_{\bfg^\circ}\cdot E_\A(-,f_{\Phi_\cD,s-\onehalf})}=\int\limits_{\PGL_2(F)\bksl \PGL_2(\A_F)}\varphi_{\bff}(g\cJ_\infty t_n)\varphi_{\bfg^\circ}(g)E_\A(g,f_{\Phi_\cD,s-\onehalf})\rmd^{\rmt}g\\
&\quad =\frac{1}{\sqrt{\Delta_F}\zeta_F(2)}\prod_{v\divides \infty}\Psi(\rho(\cJ_\infty)W_{\pi_{1,v}},W_{\pi_{2,v}},f_{\Phi_{\cD,v}s-\onehalf})\prod_{v\ndivides p\infty}\Psi(W_{\pi_{1,v}},W_{\pi_{2,v}},f_{\Phi_{\cD,v},s-\onehalf})\\
&\quad \times \prod_{v\divides p}\Psi(W_{\pi_{1,v}}^\Ord,W_{\pi_{2,v}},f_{\Phi_{\cD,v}s-\onehalf}).
\end{align*}
For each place $v$, let $L(s,\pi_{1,v}\times\pi_{2,v})$ be the local Rankin-Selberg $L$-function for $\pi_{1,v}\times\pi_{2,v}$. If $v\in \bfh$ and $v\ndivides p\frakn$, 
\[\Psi(W_{\pi_{1,v}},W_{\pi_{2,v}},f_{\Phi_{\cD,v},s-\onehalf})=L(s,\pi_{1,v}\times\pi_{2,v})\]
by \cite[Proposition 15.9]{Jacquet72LNM278}, and if $v\divides\infty$, it is computed in \cite[Proposition 5.3]{HsiehChen20} that
\begin{align*}
\Psi(\rho(\cJ_\infty)W_{\pi_{1,v}},W_{\pi_{2,v}},f_{\Phi_{\cD,v}s-\onehalf})=&(\sqrt{-1})^{k_1-k_2}2^{-k_1-1}L(s,\pi_{1,v}\times\pi_{2,\infty}).
\end{align*}
Therefore, we find that
\[\pair{\rho(\cJ_\infty t_n)\varphi_{\bff}}{\varphi_{\bfg^\circ}E_\A(-,f_{\Phi_\cD,s-\onehalf})}=\frac{L(s,\pi_1\times\pi_2)}{\sqrt{\Delta_F}\zeta_F(2)}\left(\frac{(\sqrt{-1})^{(k_1-k_2)}}{2^{k_1+1}}\right)^{[F:\Q]}\prod_{v\divides p\frakn}\Psi_v(s),\]
where $\Psi_v(s)$ are the local zeta integrals defined by \begin{align*}\Psi_v(s)&=\frac{\Psi(W_{\pi_{1,v}},W_{\pi_{2,v}},f_{\Phi_{\cD,v},s-1/2})}{L(s,\pi_{1,v}\times\pi_{2,v})}\text{ if }v\divides \frakn,\\
\Psi_v(s)&=\frac{\Psi ( \rho (t_n)W_{\pi_{1,v}}^{(\al_{\bff,v})},W_{\pi_{2,v}},f_{\Phi_{\cD,v},s-1/2})}{L(s,\pi_{1,v}\times\pi_{2,v})}\text{ if }v\divides p;\,v\not\in S_p^{\rm reg},\\
\Psi_v(s)&=\frac{\Psi ( \rho (t_n)W_{\pi_{1,v}}^{(\al_{\bff,v})},W^{(\al_{\bfg,v})}_{\pi_2,v},f_{\Phi_{\cD,v},s-1/2})}{L(s,\pi_{1,v}\times\pi_{2,v})}\text{ if }v\in S_p^{\rm reg}.
\end{align*}
We first consider the integrals $\Psi_v(s)$ for $v\divides\frakn$. Put 
\[S_{\rm i}=\stt{v\in S_{\rm dis}\mid \pi_{1,v}\text{ is supercuspidal and }\pi_{1,v}\iso\pi_{1,v}\ot\tau_{F_{v^2}/\Fv}}.\]
The integrals $\Psi_v(s)$ for $v\in S_0$ have been evaluated in \cite[Lemma 6.3, 6.5]{HsiehChen20}, and we obtain
\[\Psi_v(s)=\frac{\zeta_{\Fv}(2)\abs{\frakm}_{v}}{\zeta_{\Fv}(1)}\begin{cases}
1&\text{ if }v\not\in S_{\rm i},\\
(1+\rmN v^{-1})&\text{ if }v\in S_{\rm i}.
\end{cases}\]
If $v\not\in S_0\iff v\divides\frakm_0$, then $\pi_{1,v}$ is not minimal and $\abs{\frakn}_{v}=\abs{\frakm}_{v}$. In this case, $W_{\pi_{i,v}}(\pDII{a}{1})=\bbI_{\OFv^\x}(a)\text{ for }i=1,2$ and \[f_{\Phi_{\cD,v},s}(u)=\bbI_{U_0(\frakn)_v}(u)\text{ for }u\in\GL_2(\OFv).\] 
A direct calculation shows that
\begin{align*}\Psi(W_{\pi_{1,v}},W_{\pi_{2,v}},f_{\Phi_{\cD,v},s-1/2})&=\vol(U_0(\frakn)_v,\rmd g_v)=\frac{\abs{\frakn}_{v}\zeta_{\Fv}(2)}{\zeta_{\Fv}(1)};\\
\Psi_v(s)&=\frac{\abs{\frakm}_v\zeta_{\Fv}(2)}{\zeta_{\Fv}(1)L(s,\pi_{1,v}\times\pi_{2,v})}.\end{align*}
We next compute the local zeta integral $\Psi_v(s)$ for $v\in S_p$ by a similar calculation in \cite[Lemma 6.1]{HsiehChen20}. Put $W_1=W_{\pi_{1,v}}^{(\al_{\bff,v})}$ and $W_2=W_{\pi_{2,v}}$ if $v\not\in S_p^{\rm reg}$ or $W_{\pi_{2,v}}^{(\al_{\bfg,v})}$ if $v\in S_p^{\rm reg}$. Then $\Psi ( \rho (t_n)W_1,W_2,f_{\Phi_{\cD,v},s-1/2})$ equals 
\begin{align*}
&\frac{\zeta_{\Fv}(2)}{\zeta_{\Fv}(1)}\int_{\Fv^\x}\int_{\Fv}W_1(\pDII{y}{1}\pMX{0}{-1}{1}{x}t_n)W_2(\pDII{-y}{1}\pMX{0}{-1}{1}{x})
\abs{y}_v^{s-1}\\
&\times f_{\Phi_{\cD,v},s-\onehalf}(\pMX{0}{-1}{1}{x})\rmd x\rmd^\x y\\
=&\frac{\zeta_{\Fv}(2)}{\zeta_{\Fv}(1)}\int_{\Fv}\int_{\Fv^\x}W_1(\pDII{yp^n}{p^{-n}}\pMX{1}{0}{-p^{2n}x}{1})W_2(\pDII{-y}{1}\pMX{0}{-1}{1}{x})\abs{y}_v^{s-1}\\
&\times \bbI_{\OFv}(x)\rmd^\x y\rmd x\\
=&\frac{\zeta_{\Fv}(2)\al_{\bff,v}\Abs_v^\onehalf(p^{2n})\om_{v}^{-1}(p^n)}{\zeta_{\Fv}(1)}\frac{1}{\gamma(s,\pi_2\ot\al_{\bff,v})}\int_{\Fv^\x}W_2(\pDII{-y}{1})\om_2^{-1}\al_{\bff,v}^{-1}\Abs_v^{\onehalf-s}(y)\rmd^\x y\\
=&\frac{\om_v^{-1}\al_{\bff,v}^2\Abs_v(p^{n})\zeta_{\Fv}(2)}{\zeta_{\Fv}(1)}\cdot L(s,\pi_{2,p}\ot\al_{\bff,v})\begin{cases}
\frac{1}{L(1-s,\al_{\bfg,v}^{-1}\al_{\bff,v}^{-1})}&\text{ if }v\in\Sg^{\rm reg}_p,\\
1&\text{ otherwise.}
\end{cases}
\end{align*}
Here we have used the local functional equation for $\GL(2)$ in the second equality. Combining these formulas, we finally obtain
\begin{align*}&\pair{\rho(\cJ_\infty t_n)\varphi_{\bff}}{\varphi_{\bfg^\circ}\cdot E_\A(-,f_{\Phi_\cD,s-1/2})}=\frac{L(s,\pi_1\times\pi_2)}{\sqrt{\Delta_F}\zeta_F(2)}\left(\frac{(\sqrt{-1})^{k}}{2^{(k_1+1)}}\right)^{[F:\Q]}(\sqrt{-1})^{k_1-k_2}\\
&\times\frac{1}{\rmN\frakm}\prod_{v\divides \frakn}\frac{\zeta_{\Fv}(2)}{\zeta_{\Fv}(1)}\prod_{v\in S_{\rm i}}(1+\rmN v^{-1})\prod_{v\divides\frakm_0}\frac{1}{L(s,\pi_{1,v}\times\pi_{2,v})}\\&\times\prod_{v\divides p}\frac{\om_v^{-1}\al_{\bff,v}^2\Abs_v(p^{n})\zeta_{\Fv}(2)}{\zeta_{\Fv}(1)}\cdot\frac{1}{L(s,\pi_{2,p}\ot\beta_{\bff,v})}\prod_{v\in S_{\rm reg}}\frac{1}{L(1-s,\al_{\bfg,v}^{-1}\al_{\bff,v}^{-1})}.\end{align*}
The result follows from \propref{P:Petnorm} and the fact that 
if $v\in S_{\rm vr}\cap S_0$, then $v$ must be supercuspidal, and hence 
\[L(1,\pi_v,\Ad)=\begin{cases}(1+\rmN v^{-1})^{-1}&\text{ if }v\in S_{\rm i}\\
1&\text{ otherwise.}\end{cases}\]
by \cite[Corollary (1.3)]{GJ78ASENS}.
\end{proof}
\subsection{Proof of \propref{P:3RS}}
Now we are in a position to to prove \propref{P:3RS} by Hida's $p$-adic Rankin-Selberg method. We shall use the representation theoretic approach in \cite{HsiehChen20}. Let $d=[F:\Q]$. It suffices to show that for all but finitely many positive integer $k$ with $k\con 0\pmod{p^d-1}$, 
  \beq\label{E:34}\begin{aligned}\e_\Sg^k(\bB)=&\frac{1}{\Dmd{\rmN\frakm}^k}\cdot \frac{2^d\cL_{\Sg}(k,0,\bfone)\cL_{\Sg}(k,0,\chi)}{h_{K/F}\cdot \cL_{\Sg}(k,-k,\chi)\zeta_{F,p}(1-s)}\\
  &\times\frac{\psi(\frakf\frakd_K^2)^{-k}}{\prod_{\frakP\in\Sg_p^{\rm irr}}(1-\psi^k(\Pbar))^2\prod_{\frakP\in\Sg_p^{\rm reg}}(1-\psi^k(\Pbar))} \cdot{\rm Ex}_{\frakl_0}(k) . \end{aligned}\eeq
 Here recall that $\chi=\brch^{1-c}$ is an anticyclotomic character of conductor $\frakf\OK$ and $\psi$ is the idele class character of $K$ corresponding to $\e_\Sg$ with $\psi_\infty(z)=z$. Note that $\psi^k$ is unramified everywhere for $k$ is divisible by $p^d-1$. To evaluate $\e_\Sg^k(\bB)$, we consider the spectral decomposition \beq\label{E:36}\begin{aligned}
&e_\Ord(\theta_\brch^\circ\nu_k(\cG(\frakm,\frakl_0)))=\frac{2^d\rmN\frakm\cdot \zeta_{F_{\frakl_0}}(1)}{\Dmd{\rmN\frakm}^k\zeta_{F,p}(1-k)}\cdot e_\Ord(\theta_\brch^\circ \bfE_k(\frakm,\frakl_0))\\
 &=\cC_k(\brch,\brch^c)\cdot\theta_{\brch\psi^{-k}}^{(\Sg_p)}+\cC_k(\brch^c,\brch)\cdot \theta_{\brch^c\psi^{-k}}^{(\Sg_p)}+\sH_{k}\in S_{k+1}(\frakn p,\brch_+^{-1}\tau_{K/F}),\end{aligned}\eeq
  where $\sH_{k}$ is orthogonal to the space generated by $\stt{V_{\fraka}\theta_{\brch^{-1}\psi^{-k}}, V_\fraka\theta_{\brch^{-c}\psi^{-k}}}_{\fraka\divides p}$ under the Petersson inner product. Since $\e_\Sg^k(\bftheta_\brch)=\theta_{\brch\psi^{-k}}^{(\Sg_p)}$ is a $p$-stabilized newform of weight $(k+1)\Sg$, the decomposition \eqref{E:36} is indeed obtained by the image of \eqref{E:32} under the map $\e_\Sg^k$, and hence \[\e_\Sg^k(\bB)=\iota_p^{-1}(\cC_k(\brch^c,\brch)).\] 
  
 Now we use the adelic Rankin-Selberg method to compute the value $\cC_k(\brch^c,\brch)$. 
 Let $\bff^{\rm new}=\theta_{\brch^{-1}\psi^{-k}}$ and $\bfg^{\rm new}=\theta_\brch$ be the Hilbert newforms associated with Hecke characters $\brch^{-1}\psi^{-1}$ and $\brch$. Let $\pi_1$ and $\pi_2$ be the cuspidal automorphic representation of $\GL_2(\A_F)$ associated with $\varphi_{\bff^{\rm new}}$ and $\varphi_{\bfg^{\rm new}}$. Then $\pi_1$ and $\pi_2$ are the automorphic inductions of the idele class characters $\brch^{-1}\psi^{-k}\Abs_{\A_K}^{\frac{k}{2}}$ and $\brch$, so  we see that $\pi_1$ and $\pi_2$ satisfy (RS1) and (RS2).
In addition,  the automorphic $L$-functions of $\pi_1$ and $\pi_2$ are given by
 \begin{align*}L(s,\pi_1)=&\Gamma_\C(s+k)^dD(s+\frac{k}{2},\bff^{\rm new})=L(s+\frac{k}{2},\brch^{-1}\psi^{-k});\\
 L(s,\pi_2)=&\Gamma_\C(s)^dD(s,\bfg^{\rm new})=L(s,\brch).\end{align*}
   
 Let $\bff=\theta_{\brch^{-1}\psi^{-k}}^{(\Sg_p)}$ be the $p$-stabilized newform associated with $\bff^{\rm new}$ and let $\breve \bff:=\theta_{\brch^c\psi^{-k}}^{(\Sg_p)}$ be the specialization of $\e_\Sg^s(\bftheta_{\brch^c})$ at $s=k$. Let $\bfg^\circ=\theta_{\brch}^\circ$ be the theta series defined in \eqref{E:theta2}. 
Let $v\divides p$ with $v=\frakP\Pbar$ and $\frakP\in\Sg_p$. Then $\pi_{1,v}=\pi(\al_{\bff,v},\beta_{\bff,v})$, where $\al_{\bff,v},\beta_{\bff,v}:\Fv^\x\to\C^\x$ are the unramified characters such that 
 \[\al_{\bff,v}(p)=\brch^{-1}\psi^{-k}(\Pbar)\rmN\frakP^{-\frac{k}{2}};\quad \beta_{\bff,v}(p)=\brch^{-1}\psi^{-k}(\frakP)\rmN\frakP^{-\frac{k}{2}}.\]
 Then one verifies that the Whitaker functions of $\varphi_{\bff}$ and $\varphi_{\breve\bff}$ are given by \eqref{E:W1} and \eqref{E:W2}. Likewise, $\pi_{2,v}\iso\pi(\al_{\bfg,v},\beta_{\bfg,v})$, where $\al_{\bfg,v}$ and $\beta_{\bfg,v}$ are unramified characters of $\Fv^\x$ with $\al_{\bfg,v}(p)=\brch(\frakP)$ and $\beta_{\bfg,v}(p)=\brch(\Pbar)$. Then one verifies that the Whittaker function of $\varphi_{\bfg^\circ}$ is precisely given by \eqref{E:W3}. In the notation of previous two subsections, 
 \begin{itemize}
 \item $k_1=k+1$ and $k_2=1$.
\item $S_{\rm dis}$ is the set of prime factors of $\frakc^-$ and $e_\frakq$ is even for $\frakq\in S_{\rm dis}$.
 \item $S_0$ is the set of prime factors of $\frakn$ except for $\frakl$, and $\frakm_0=\frakl$ in view of \eqref{min}.
 \item The ideal $\frakm$ in \eqref{E:nm2} is the one introduced in \eqref{E:nmm}.
 \end{itemize}
By \cite[Proposition 5.2]{HsiehChen20}, for $n\gg 0$ large enough, we have
 \begin{align*}\cC_k(\brch^c,\brch)=&\frac{\pair{\rho(\cJ_\infty t_n)\varphi_{\bff}}{\varphi_{\bfg^\circ}\cdot E_\A(-,f_{\Phi_\cD,s-1/2})}|_{s=1-\frac{k}{2}}}{\pair{\rho(\cJ_\infty t_n)\varphi_{\bff}}{\varphi_{\breve \bff}}}\cdot \frac{2^d\rmN\frakm\zeta_{F_\frakl}(1)}{\zeta_{F,p}(1-k)\Dmd{\rmN\frakm}^{k}}.
 \end{align*}
Note that by definition
\[{\rm Ex}_\frakl(k)=
\frac{\zeta_{F_\frakl}(1)^2L(1,\brch^{1-c}\psi^{(1-c)k})}{L(1,\psi^{-k})L(1,\brch^{c-1}\psi^{-k})}=\frac{\zeta_{F_\frakl}(1)L(1,\pi_{1,\frakl},\Ad)}{L(1-\frac{k}{2},\pi_{1,\frakl}\times\pi_{2,\frakl})}.\]
Put $\psi_-=\psi^{1-c}$. From \propref{P:Petnorm} and \propref{P:RS.app} we deduce that
 \beq\label{E:37}\begin{aligned}&\cC_k(\brch^c,\brch)=\frac{1}{\Dmd{\rmN\frakm}^k}\cdot\frac{2^d(\sqrt{-1})^{k\Sg}L(1,\psi^{-k})L(1,\brch^{c-1}\psi^{-k})}{2L(1,\pi_1,\Ad)\zeta_{F,p}(1-k)}\\
 &\times \frac{ \prod_{\frakP\in\Sg_p}(1-\psi^{-k}(\frakP)\rmN\frakP^{-1})(1-\brch^{c-1}\psi^{-k}(\frakP)\rmN\frakP^{-1})\prod_{\frakP\in\Sg_p^{\rm reg}}(1-\phi^{1-c}\psi^k(\ol{\frakP}))}{\prod_{\frakP\in\Sg_p}(1-\chi\psi_-^k(\frakP)^{-1}p^{-1})(1-\chi\psi_-^k(\ol{\frakP}))}\\
 &\times {\rm Ex}_\frakl(k).
 \end{aligned}\eeq
 By the functional equations of $L$-functions, one has
 \begin{align*}
 L(1,\psi^{-k})L(1,\brch^{c-1}\psi^{-k})&=\varepsilon(1,\psi^{-k})\varepsilon(1,\chi^{-1}\psi^{-k})L(0,\psi^k)L(0,\chi\psi^k),\\
 L(1,\pi_1,\Ad)&=L(1,\tau_{K/F})L(1,\brch^{1-c}\psi^{(1-c)k})\\
 &=\varepsilon(1,\chi\psi_-^k)\cdot L(1,\tau_{K/\Q})L(0,\chi\psi_-^k).
 \end{align*}
Note that $\frakf\OK$ is the conductor of $\chi$ and $\psi^k$ is unramified everywhere with $\psi_-|_{\A^\x_F}=1$, so we have \begin{align*}
 \varepsilon(1,\psi^{-k})&=\sqrt{\Delta_K}^{-1}\psi(\frakd_K)^{-k},\quad
  \varepsilon(1,\chi\psi^{-k})=\varepsilon(1,\chi)\psi(\frakf\frakd_K)^{-k},\\
\varepsilon(1,\chi\psi_-^k)& =\varepsilon(1,\chi).\end{align*}
It follows that
\beq\label{E:38}\frac{(\sqrt{-1})^{k\Sg}L(1,\psi^{-k})L(1,\brch^{c-1}\psi^{-k})}{2L(1,\pi_1,\Ad)}=\frac{(\sqrt{-1})^{k\Sg}L(0,\psi^k)L(0,\chi\psi^k)}{2\sqrt{\Delta_K}L(1,\tau_{K/\Q})L(0,\chi\psi_-^k)}\cdot \psi(\frakf\frakd_K^2)^{-k}.\eeq
 By the interpolation formulae of the Katz $p$-adic $L$-function in \eqref{E:intep2}, we find that
\beq\label{E:39}\begin{aligned}\frac{\cL_\Sg(k,0,\chi)\cL_\Sg(k,0,\bfone)}{\cL_\Sg(k,-k,\chi)}=&\frac{[\cOK^\x:\cOF^\x]}{2^d\sqrt{\Delta_F}(\sqrt{-1})^{k\Sg}}\frac{L(0,\chi\psi^{k})L(0,\psi^{k})}{L(0,\chi\psi_-^k) }\\
&\times \prod_{\frakP\in\Sg_p}\frac{ (1-\psi^{k}(\frakP^{-1})\rmN\frakP^{-1})(1-\chi\psi^{k}(\frakP^{-1})\rmN\frakP^{-1})(1-\psi^k(\ol{\frakP}))^2}{(1-\chi\psi_-^k(\frakP^{-1})p^{-1})(1-\chi\psi_-^k(\ol{\frakP}))}.\end{aligned}\eeq
Combining \eqref{E:37}, \eqref{E:38}, \eqref{E:39} and \eqref{E:L1},
we obtain \eqref{E:34}.  This finishes the proof.


\appendix 
\addcontentsline{toc}{chapter}{APPENDICES}
\section{Transcendental number theory and non-vanishing of the $p$-adic regulators}
\subsection{The $p$-adic transcendence conjectures}\label{section: Schanuel}
The non-vanishing of the regulators are usually tied with well-known conjectures in $p$-adic transcendental number theory. We first recall some standard results and conjectures in  $p$-adic transcendence. 
As before, $\log_p: \overline{\Q}{}_p^{\times}\to \overline{\Q}_p$ is the Iwasawa $p$-adic logarithm with $\log_p(p)=0$, and we use the same notation for its composition with $\iota_p:\overline{\Q} \hookrightarrow \overline{\Q}_p$. Denote $\wtd{\mathcal{L}}\subset \overline{\Q}_p$ the $\overline{\Q}$-vector space generated by $1$ and  elements of 
$\log_p(\overline{\Q}{}^\times)$. We recall the statements of two conjectures at the top of a Hierarchy of conjectures and theorems concerning the arithmetic nature of  certain values of $\log_p$.
\begin{conj}[Strong Four Exponentials Conjecture]\label{conj:4exp} 
Let $M$ be a $2\times 2$ matrix
with entries in $\wtd\cL$. Assume the rows and the columns of $M$ are linearly independent over $\Qbar$ respectively. Then $M$ has rank $2$. 
\end{conj}
\begin{conj}[$p$-adic Schanuel's Conjecture]\label{conj:Schanuel}
Let $\alpha_1, \dots, \alpha_n \in \Qbar^\times$ such that \[\log_p(\alpha_1),\dots, \log_p(\alpha_n)\]are linearly independent over $\Q$. Then $\Q(\log_p(\alpha_1),\dots, \log_p(\alpha_n))$ has transcendence degree $n$ over $\Q$.
\end{conj}
 
 \begin{Remark}The significance of thse $p$-adic transcendental conjectures perhaps stems from their direct applications to the non-vanishing of $p$-adic regulators. Below we give some examples without details. \begin{enumerate} 
 \item If $F$ is a real quadratic field, the Strong Four Exponentials conjecture implies that $\mathscr{R}_{\Sigma}(\chi) \ne 0$, \ie the $\Sg$-Leopoldt conjecture for $\chi$ holds.  \item The $p$-adic
Schanuel conjecture implies the $\Sg$-Leopoldt conjecture (\cf \cite[Lemma 1.2.1]{HidaTil94Inv}) and the non-vanishing of the $\sL$-invariants $\sL_\chi$ and $\sL_\chi^{\rm ac}$.
 \end{enumerate}
 \end{Remark}
\subsection{Examples of the  $\Sg$-Leopoldt conjecture}
Notation is as in the introduction. Recall that $\chi$ is a ray class character of a CM field $K$ and $H$ is the abelian extension over $K$ cut out by $\chi$. We give examples of $(\chi,K,\Sg,p)$ such that  $\Sg$-Leopoldt conjecture for $\chi $ holds, {\it i.e.,} \conjref{conj2} (or equivalently  \eqref{Leo2} holds). The following famous Baker-Brumer Theorem is the best result towards these conjectures in the current literature.\begin{thm}[Baker--Brumer \cite{brumer}]\label{thm:BB} 
Let $\lambda_1,\dots, \lambda_n \in \mathcal{L}$ are linearly independent over $\Q$, then they are linearly independent over $\overline{\Q}$.
\end{thm}
\begin{cor}\label{C:A3}Suppose that $K = F M$ for an imaginary quadratic field $M$ in which $p$ splits and  $H/M$ is abelian.  Suppose also that  the CM type  $\Sigma$ of $K$ is obtained by extending the CM type of $M$, \ie a choice of embeddings $\iota:M \hookrightarrow \C$. Then the $\Sg$-Leopoldt conjecture for $\chi:\Gal(H/K)\to\Qbar_p^\x$ holds. 
\end{cor}
\begin{proof}By the Dirichlet unit theorem (\cf\cite[Prop. 8.7.2, p.503]{Neuk08}), there is an isomorphism 
 \[\Q\oplus \Q\frako_H^\x\iso \Q[\Gal(H/M)]\]
 as $\Gal(H/M)$-modules, so there exists a \emph{Minkowski} unit $u\in \frako_H^\x$, \ie the set $\stt{\sg(u)u^{-1}}_{\sg\in \Gal(H/M)\bksl \stt{1}}$ is a basis of $\Q\frako_H^\x$. 
Moreover, since $H/M$ is abelian and $\chi\neq 1$,  we find that
 \[\rmH^1_{(f,f)}(K,\chi^{-1}(1))\iso \bigoplus_{i=1}^d\frako_H^\x[\chi_i]\text{ as $\Gal(H/M)$-modules},\]
 where $\chi_i$ are characters of $\Gal(H/M)$ extending $\chi$ and that $\dim_E\frako_H^\x[\chi_i]=1$. For any character $\chi_i\neq\bfone$,
\[u_{\chi_i}:=\sum_{\tau\in \Gal(H/M)} \chi_i(\tau^{-1})\ot \tau(u)u^{-1}\in \frako_H^\x[\chi_i]. \]
Then $\log_p(u_{\chi_i})\neq 0$ by the Baker-Brumer Theorem. It follows that $\log_p:\frako_H^\x[\chi_i]\iso \Cp$ is an isomorphism and $\log_{\Sg,p}:\rmH^1_{(f,f)}(K,\chi^{-1}(1))\iso \Cp^d$.
\end{proof}
In the above example, $H$ is abelian over an imaginary quadratic field $M$, so the Katz $p$-adic $L$-function $L_\Sg(s,\chi)$ is a product of $p$-adic $L$-functions over $M$, and the trivial zeros in this case have been studied extensively in \cite{BS} via the Euler system of elliptic units. To obtain more interesting examples beyond this case, we will make use of the following strong version of the Six Exponentials Theorem due to Roy \cite[Corollary 2]{Roy}\begin{thm}[Strong Six Exponentials Theorem]\label{thm:6exp} 
Let $M$ be a $2\times 3$ matrix
with entries in $\wtd\cL$. Assume the rows and the columns of $M$ are linearly independent over $\Qbar$ respectively. Then $M$ has rank $2$. 
\end{thm}
\begin{prop}\label{P:examplesSigmaLeopoldt}\
Assume that $F$ is a quadratic real field in which $p$ splits and $\chi\neq\chi^c$. Then there exists a $p$-adic CM type $\Sigma$ of $K$ such that the $\Sigma$-Leopoldt conjecture holds.\end{prop}
\begin{proof} 
We let $L$ be the Galois closure of $H$. Choose a $p$-adic CM type $\Sg=\stt{\sg_1,\sg_2}$ of $K/F$ and let $\Sg_{L,i}$ be the subset of $\sg\in\Gal(L/\Q)$ such that $\sg|_{K}=\sg_i$. Let  $G_\infty=\stt{1,c}\subset \Gal(L/\Q)$ be the decomposition at the archimdean place induced by $\iota_\infty:\Qbar\hookto\C$. By the Dirichlet unit Theorem, there is an isomorphism as $\Gal(L/\Q)$-modules
\[ \Q\oplus \Q\frako_H^\x\iso \Ind_{G_\infty}^{\Gal(L/\Q)}\Q.\]
It follows that there exists $\bfu\in \frako_L^\x$ such that 
$\stt{\tau(\bfu)\bfu^{-1}}_{\tau\in\Sg_{L,1}\disjoint\Sg_{L.2}}$ is a basis of $\Q\frako_L^\x$. For $i=1,2$, put $\bfu_i=\prod_{\sg\in\Sg_{L,i}}\sg(\bfu)\in \frako_H^\x$ and 
 \[u_i:=e_\chi \bfu_i=\sum_{\tau\in \Gal(H/K)}\chi(\tau^{-1})\ot\tau(\bfu_i)\bfu_i^{-1} \in  \frako_H^\times[\chi].\]
 Then $\stt{u_1,u_2}$ is a basis of $\frako_H^\x[\chi]$.  Consider the matrix \[M=\left(\begin{smallmatrix}
 \log_p(\sg_1(u_1)) & \log_p(\sigma_2(u_1)) &   \log_p(\sigma_2c(u_1)) \\
\log_p(\sg_1(u_2)) & \log_p(\sigma_2(u_2))  &   \log_p(\sigma_2c(u_2))
\end{smallmatrix}\right).\] 
Since $\chi\neq\chi^c$, we see that the row and the column vectors are also linearly independent over $\Qbar$ respectively by the Baker-Brumer theorem. According to Roy's Six exponentials theorem, $M$ has rank $2$. Therefore either $\left(\begin{smallmatrix}
 \log_p(\sg_1(u_1)) & \log_p(\sigma_2(u_1)) \\
\log_p(\sg_1(u_2)) & \log_p(\sigma_2(u_2)) 
\end{smallmatrix}\right)$ or $\left(\begin{smallmatrix}
 \log_p(\sg_1(u_1)) & \log_p(\sigma_2 c (u_1))  \\
\log_p(\sg_1(u_2)) & \log_p(\sigma_2c (u_2))  
\end{smallmatrix}\right)$ is invertible. This shows that either $\{\sg_1,\sigma_2\}$ or $\{\sg_1,\sigma_2c\}$ is a $p$-adic CM type that satisfies the $\Sg$-Leopoldt conjecture for $\chi$.
 \end{proof}

\bibliographystyle{amsalpha}
\bibliography{mybib}
\end{document}